\documentclass[11pt]{amsart}%
\usepackage{amssymb}
\usepackage{amsfonts}
\usepackage{graphicx}%
\usepackage{amsmath}%
\providecommand{\U}[1]{\protect\rule{.1in}{.1in}}
\newtheorem{theorem}{Theorem}[section]

\newtheorem{corollary}[theorem]{Corollary}

\newtheorem{definition}[theorem]{Definition}

\newtheorem{lemma}[theorem]{Lemma}

\newtheorem{proposition}[theorem]{Proposition}

\numberwithin{equation}{section}
\begin{document}
\title[Sharp results for spherical metric on flat tori]{Sharp results for spherical metric on flat tori with conical angle $6\pi$ at
two symmetric points}

\begin{abstract}
In this paper, we investigate the following curvature equation:%
\begin{equation}
\Delta u+e^{u}=8\pi(\delta_{0}+\delta_{\frac{\omega_{k}}{2}})\text{ in
}E_{\tau}\text{, }\tau\in\mathbb{H}\label{a}%
\end{equation}
Here $E_{\tau}$ represents a flat torus and $\frac{\omega_{k}}{2}$ is one of
the half periods of $E_{\tau}$. Our primary objective is to establish a
necessary and sufficient criterion for the existence of a non-even family of
solutions (see the definition in Section \ref{introduction}). Remarkably, this
is equivalent to determining the presence of solutions for the equation with a
single conical singularity:
\[
\Delta u+e^{u}=8\pi\delta_{0}\text{ in }E_{\tau}\text{, }\tau\in
\mathbb{H}\text{.}%
\]
This study marks the first exploration of the structure of non-even families
of solutions to the curvature equation with multiple singular sources in the
literature. Building on our findings, we provide a comprehensive analysis of
the solution structure for equation (\ref{a}) for \textbf{all} $\tau$. This
analysis is facilitated by Theorem \ref{thm25}, which will play a central role
in our exploration of cases involving general parameters in the future, such
as:
\[
\Delta u+e^{u}=8\pi n(\delta_{0}+\delta_{\frac{\omega_{k}}{2}})\text{ in
}E_{\tau},\text{ }n\in\mathbb{N}\text{.}%
\]
As an application, we offer explicit descriptions for solutions to equation
(\ref{a}) in the context of both rectangle tori and rhombus tori. See
Corollary \ref{cor18 copy(1)} as well as Corollary \ref{cor2}.

\end{abstract}
\author{Ting-Jung Kuo}
\address{Department of Mathematics, National Taiwan Normal University, Taipei 11677,
Taiwan }
\email{tjkuo1215@ntnu.edu.tw, tjkuo1215@gmail.com}
\maketitle

\section{Introduction}

\label{introduction}

Let $\omega_{0}$ $=$ $0,$ $\omega_{1}$ $=$ $1,$ $\omega_{2}$ $=$ $\tau$,
$\omega_{3}$ $=$ $1+\tau$, $\Lambda_{\tau}$ $=$ $\mathbb{Z}$ $\mathbb{\oplus}$
$\mathbb{\tau Z}$, and $E_{\tau}\doteqdot\mathbb{C}/\Lambda_{\tau}$ where
$\tau$ $\in$ $\mathbb{H}$ $=$ $\left\{  \tau|\operatorname{Im}\tau>0\right\}
$. In this paper, we adopt the following notations. Let $\wp\left(
z|\tau\right)  $ be the Weierstrass elliptic function defined by%
\[
\wp\left(  z|\tau\right)  =\frac{1}{z^{2}}+\sum_{(m,n)\in\mathbb{Z}%
^{2}\backslash(0,0)}\left[  \frac{1}{\left(  z-m-n\tau\right)  ^{2}}-\frac
{1}{(m+n\tau)^{2}}\right]  ,
\]
which satisfies the classical differential equation%
\begin{equation}
\wp^{\prime}(z|\tau)^{2}=4\wp^{3}(z|\tau)-g_{2}(\tau)\wp(z|\tau)-g_{3}%
(\tau)=4\prod\limits_{k=1}^{3}(\wp(z)-e_{k}(\tau
)),\label{Classical differential equation}%
\end{equation}
where $^{\prime}=\frac{d}{dz}$ and
\[
e_{k}(\tau)=\wp\left(  \frac{\omega_{k}}{2}|\tau\right)  ,k=1,2,3.
\]
Comparing the coefficients of the equation
(\ref{Classical differential equation}), it is easy to see that%
\begin{equation}
g_{2}(\tau)=-4(e_{1}(\tau)e_{2}(\tau)+e_{1}(\tau)e_{3}(\tau)+e_{2}(\tau
)e_{3}(\tau)),\label{g2}%
\end{equation}%
\begin{equation}
g_{3}(\tau)=4e_{1}(\tau)e_{2}(\tau)e_{3}(\tau),\label{g3}%
\end{equation}
and
\begin{equation}
e_{1}(\tau)+e_{2}(\tau)+e_{3}(\tau)=0.\label{14}%
\end{equation}
Define the Weierstrass zeta function by%
\[
\zeta\left(  z|\tau\right)  \doteqdot-\int^{z}\wp(\xi|\tau)d\xi.
\]
Notice that the zeta function has two quasi-periods $\eta_{i}(\tau),$ $i=1,2,$
defined by
\[
\zeta\left(  z+\omega_{1}|\tau\right)  =\zeta\left(  z|\tau\right)  +\eta
_{1}(\tau)\text{ and }\zeta\left(  z+\omega_{2}|\tau\right)  =\zeta\left(
z|\tau\right)  +\eta_{2}(\tau).
\]
We also denote
\[
\eta_{3}(\tau)\doteqdot\eta_{1}(\tau)+\eta_{2}(\tau)
\]
so that
\[
\zeta(z+\omega_{3}|\tau)=\zeta(z|\tau)+\eta_{3}(\tau).
\]

The curvature equation, known as the singular Liouville equation, with four
conical singularities at $\frac{\omega_{k}}{2}$, $k$ $=$ $0,$ $1,$ $2,$ $3$ is
expressed as follows:
\begin{equation}
\Delta u+e^{u}=8\pi\sum_{k=0}^{3}m_{k}\delta_{\frac{\omega_{k}}{2}}\text{ in
}E_{\tau}\text{, }\tau\in\mathbb{H}\label{MFE1}%
\end{equation}
where $\delta_{\frac{\omega_{k}}{2}}$ is the Dirac measure at $\frac
{\omega_{k}}{2}$ and $m_{k}\in\mathbb{Z}_{\geq0}$ for all $k$ with $\sum
m_{k}\geq1$.

Equation (\ref{MFE1}) originates from conformal geometry, which can be
described briefly as follows: Let $(E_{\tau},$ $dz^{2})$ be a flat torus, with
$dz^{2}$ representing the flat metric. In conformal geometry, a solution $u$
to equation (\ref{MFE1}) leads to a new metric $ds^{2}$ $=$ $\frac{1}{2}%
e^{u}dz^{2},$ which is conformal to the flat metric.

Equation (\ref{MFE1}) states that the new metric $ds^{2}$ has Gaussian
curvature given by:
\[
K(z;u)=-e^{-u}\cdot\Delta u=1-e^{-u}\left(  8\pi\sum_{k=0}^{3}m_{k}%
\delta_{\frac{\omega_{k}}{2}}\right)  .
\]
In simpler terms, this new metric exhibits a spherical geometry featuring
conical singularities positioned at $\frac{\omega_{k}}{2}.$ These
singularities have peculiar conical integer angles of $2\pi$ $(2m_{k}+1)$ for
$k$ $=$ $1,$ $2,$ $3$. It's noteworthy that equation (\ref{MFE1}) bears
significance not only in the realm of mathematics but also finds practical
applications in the field of statistical physics. Specifically, it represents
the mean field limit governing the Euler flow in Onsager's vortex model, which
is why it is also known as the mean field equation in this context.

It's important to emphasize that $m_{k}$ does not necessarily have to be a
positive integer in the general context. In fact, equation (\ref{MFE1})
remains well-defined for any $m_{k}$ $>$ $-1$, regardless of whether $m_{k}$
is a positive integer. Set $\rho$ $\doteqdot$ $\sum_{k=0}^{3}m_{k}$. It has
been proven that, when $\rho$ $\not \in $ $\mathbb{Z}_{\geq0}$, any solutions
of equation (\ref{MFE1}) have an uniform estimate. This implies the
Leray-Schauder degree $d_{\rho}$ is well-defined and $d_{\rho}$ has been
calculated to be nonzero in \cite{CCChen-Lin 1}. Consequently, equation
(\ref{MFE1}) with $\rho\not \in \mathbb{Z}_{\geq0}$ always has at least one
solution for any tori $E_{\tau}$, regardless of the geometry of the torus
$E_{\tau}$.

However, if all $m_{k}$ $\in$ $\mathbb{Z}_{\geq0}$, then $\rho\in
\mathbb{Z}_{\geq0}$. In such a case, equation (\ref{MFE1}) exhibits the
so-called bubbling phenomenon. Namely, there might exist a sequence of
solutions $u_{k}(z)$ to the equation (\ref{MFE1}) that blows up somewhere on
$E_{\tau}$ as $k$ $\rightarrow$ $\infty$. Consequently, there are no a priori
estimates available for these solutions, and the method based on topological
degree proves ineffective in this scenario. The parameter $\rho$ is called a
critical parameter when it assumes a positive integer value, and solving
equation (\ref{MFE1}) with a critical parameter presents significantly greater
challenges compared to the non-critical case.

The study of equation (\ref{MFE1}) with critical parameters $m_{k}$ $\in$
$\mathbb{Z}_{\geq0}$ relies on its integrability, which will be thoroughly
discussed in Section \ref{Integrability}. Refer to \cite{CLW} as well. The
integrability of the equation (\ref{MFE1}) is rooted in the classical
Liouville Theorem which asserts that for any solution $u(z)$ to the equation
(\ref{MFE1}), there exists a function $f(z)$ locally such that $u(z)$ can be
expressed in the following form:%
\begin{equation}
u(z)=\log\frac{8|f^{\prime}(z)|^{2}}{(1+|f(z)|^{2})^{2}}.\label{502}%
\end{equation}
See \cite{CLW,Prajapat -Tarantello} for a proof. In academic discourse, this
function $f(z)$ is commonly referred to as the developing map of the solution
$u$. It's important to note that developing maps are not unique, and the local
behavior of $u$ near $\frac{\omega_{k}}{2}$ even suggests that a developing
map $f$ may take on multiple values in $\mathbb{C}$. Essentially, it may
operate as a meromorphic function on an unramified covering of $\mathbb{C}$
$\setminus$ $\left(  \cup_{k=0}^{3}\{\frac{\omega_{k}}{2}+\Lambda_{\tau
}\}\right)  $ when one of $m_{k}$ is not a half-integer. When all $m_{k}$
belong to $\frac{1}{2}\mathbb{Z}$, the developing map can be extended
universally into a single-valued meromorphic function in $\mathbb{C}$.

When $m_{k}$ $\in$ $\mathbb{Z}_{\geq0}$ for all $k$ in (\ref{MFE1}), it has
been established that any developing map can be adjusted to satisfy the type
II constraint:
\begin{equation}
f(z+\omega_{1})=e^{2i\theta_{1}}f(z)\text{ and }f(z+\omega_{2})=e^{2i\theta
_{2}}f(z),\text{ }\theta_{i}\in\mathbb{R},i=1,2.\label{type ii}%
\end{equation}
The significance of type II solutions is that once equation (\ref{MFE1}) has a
solution $u$ given in (\ref{502}) with a type II developing map $f$, by
replacing $f$ by $e^{\beta}f$ for any $\beta\in\mathbb{R}$, we see that
$e^{\beta}f$ also satisfies type II constraint. This gives rise to a
one-parameter family of solutions $u_{\beta}(z)$ defined as follows:
\[
u_{\beta}(z)=\log\frac{8e^{2\beta}\left\vert f^{\prime}\right\vert ^{2}%
}{\left(  1+e^{2\beta}\left\vert f\right\vert ^{2}\right)  ^{2}}.
\]
From here, there are two possibilities for each family of solutions
$\{u_{\beta}(z)|\beta$ $\in$ $\mathbb{R}\}$:\textit{\medskip}

\noindent(i) \textbf{Even family of solutions}: The family of solutions
$\{u_{\beta}(z)|\beta\in\mathbb{R}\}$ contains at least one even solution. By
scaling, we may assume $u_{0}(z)$ (i.e., $\beta=0$) is an even solution, and
its developing map is denoted by $f(z)$. The type II constraint of $f$ implies
that
\begin{equation}
u_{0}(z)=u_{0}(-z)\text{ if and only if }f(-z)=\frac{1}{f(z)}%
.\label{even solution}%
\end{equation}
Suppose $u_{\beta}(z)$ is any even solution in this family. Since $e^{\beta
}f(z)$ represents the developing map of $u_{\beta}(z)$, due to the evenness of
$u_{\beta}(z)$, we must have $e^{\beta}f(-z)=\frac{1}{e^{\beta}f(z)},$ which,
combined with (\ref{even solution}), yields $\beta=0$. Therefore, $u_{0}(z)$
is the only even solution in this family.\textit{\medskip}

\noindent(ii) \textbf{Non-even family of solutions}: The family of solutions
$\{u_{\beta}(z)|\beta\in\mathbb{R}\}$ contains no even
solution.\textit{\medskip}

Now, the challenge lies in establishing criteria for the existence of both
even and non-even families of solutions. If $(m_{0},$ $m_{1},$ $m_{2},$
$m_{3})$ $=$ $\left(  n,0,0,0\right)  $, then equation (\ref{MFE1}) simplifies
to an equation with a single singular point:
\begin{equation}
\Delta u+e^{u}=8\pi n\delta_{0}\text{ in }E_{\tau}\text{, }\tau\in
\mathbb{H}\text{.}\label{163}%
\end{equation}
It was demonstrated in \cite{LW-AnnMath,CLW} that, due to the presence of only
one singular point in equation (\ref{163}), it gives rise to a one-parameter
family of solutions containing a unique even solution, once a solution to
equation (\ref{163}) is found. Therefore, there are no non-even families of
solutions for equation (\ref{163}), and the focus in this scenario shifts to
the search for even solutions.

The initial breakthrough for equation (\ref{163}) with $n$ $=$ $1$:
\begin{equation}
\Delta u+e^{u}=8\pi\delta_{0},\text{ in }E_{\tau}\text{, }\tau\in
\mathbb{H}\label{1631}%
\end{equation}
was achieved in \cite{LW-AnnMath}. In this work, it was demonstrated that the
existence of a solution to equation (\ref{1631}) is closely tied to the shape
of $E_{\tau}$, extending beyond its topology alone. The study of critical
points of the Green function $G(z|\tau)$ on $E_{\tau}$ plays a fundamental
role. The work by \cite{LW-AnnMath} reveals that for any torus $E_{\tau}$, the
Green function $G(z|\tau)$ exhibits one of two patterns depending on the
torus's shape: it either possesses three trivial critical points at half
periods $\frac{\omega_{k}}{2},$ $k$ $=$ $1,$ $2,$ $3,$ or five critical
points, including the three trivial ones and two additional nontrivial ones.
The Green function $G(z|\tau)$ with nontrivial critical points corresponds to
flat tori characterized by the existence of even solutions to the curvature
equation (\ref{1631}).

The critical points of $G(z|\tau)$ can be characterized more precisely by the
notion of \textit{pre-modular form,} initially introduced in \cite{CLW,LW}.
Let $\Gamma(n)$ $\doteqdot$ $\{ \gamma\in SL(2,\mathbb{Z})|$ $\gamma$ $\equiv$
$I_{2}\mod m\}$ be the principal congruence subgroup.

\begin{definition}
\label{def2}A function $f(r,s,\tau)$ on $\mathbb{C}^{2}\times\mathbb{H}$,
which depends meromorphically on $(r,s)$ $(\operatorname{mod}\mathbb{Z}%
^{2})\in\mathbb{C}^{2}$, is called a pre-modular form if the following hold:

\begin{itemize}
\item[(1)] If $(r,s)\in\mathbb{C}^{2}$ $\setminus$ $\frac{1}{2}\mathbb{Z}^{2}%
$, then $f(r,s,\tau)\not \equiv 0,\infty$ and is meromorphic in $\tau$.
Furthermore, it is holomorphic in $\tau$ if $(r,s)\in\mathbb{R}^{2}$
$\setminus$ $\frac{1}{2}\mathbb{Z}^{2}$.

\item[(2)] There is $k\in\mathbb{N}$ independent of $(r,s)$ such that if
$(r,s)$ is any $n$-torsion point for some $n\geq3$, then $f(r,s,\tau)$ is a
modular form of weight $k$ with respect to $\Gamma(n)$.
\end{itemize}
\end{definition}

For any pair $(r,s)$ $\in\mathbb{C}^{2}\backslash\frac{1}{2}\mathbb{Z}^{2}$,
we introduce the function $Z(r,s,\tau)$ as follows:%
\begin{equation}
Z(r,s,\tau)\doteqdot\zeta(r+s\tau)-r\eta_{1}(\tau)-s\eta_{2}(\tau),\label{zrs}%
\end{equation}
which depends meromorphically on $(r,$ $s)$. Since $(r,$ $s)$ $\not \in
\frac{1}{2}\mathbb{Z}^{2}$, it is evident that $Z(r,s,\tau)$ $\not \equiv
0,\infty$ as a function of $\tau$, and is meromorphic in $\tau$. This
meromorphic function $Z(r,s,\tau)$ was initially introduced by Hecke in
\cite{Heck}. Hecke demonstrated that it is a modular form of weight one with
respect to $\Gamma(n)$ whenever $(r,$ $s)$ is an $n$-torsion point. Thus,
$Z(r,s,\tau)$ qualifies as a pre-modular form, as defined in Definition
\ref{def2}, and is historically known as the Hecke form.

For $z$ $=$ $x$ $+$ $iy$ $=$ $r$ $+$ $s\tau$, where $r,$ $s$ are real numbers,
it can be established that:
\begin{equation}
-4\pi\frac{\partial G(z|\tau)}{\partial z}=Z(r,s,\tau).\label{CG}%
\end{equation}
It is seen from (\ref{CG}) that for $\left(  r,s\right)  $ $\in$ $\{\left(
1/2,0\right)  ,$ $(0,1/2),$ $(1/2,1/2)\}$ mod $\mathbb{Z}^{2}$, $Z(r,s,\tau)$
$\equiv$ $0$ for any $\tau\in\mathbb{H}$. This corresponds to three trivial
critical points $\frac{1}{2},$ $\frac{\tau}{2},$ $\frac{1+\tau}{2}, $
respectively. Consequently, the existence of nontrivial critical points is
associated with the $\tau$ zeros of $Z(r,s,\tau)$ for certain $\left(
r,s\right)  $ $\in$ $\mathbb{R}^{2}$ $\setminus$ $\frac{1}{2}\mathbb{Z}^{2}$.
An important implication is that the equation (\ref{1631}) has no solutions
for rectangular tori, specifically $\tau$ $=$ $ib$ $\in$ $i\mathbb{R}_{>0}$.

In a series of frameworks \cite{CLW,LW}, a general pre-modular form
$Z^{(n,0,0,0)}(r,$ $s,$ $\tau)$, with weights $n(n+1)/2,$ has been
constructed. This pre-modular form serves to describe flat tori for which
equation (\ref{163}) has a solution, or, equivalently, an even family of
solutions. $Z^{(n,0,0,0)}(r,s,\tau)$ is, in fact, a polynomial of
$Z(r,s,\tau)$ of degree $n(n+1)/2$, and in general, writing down their
explicit expressions is challenging. Here are some examples
\cite{Chen-Kuo-Lin-Lame II,Dahmen,LW}:%
\[
Z^{(1,0,0,0)}(r,s,\tau)\text{ }=\text{ }Z,\quad Z^{(2,0,0,0)}(r,s,\tau)\text{
}=\text{ }Z^{3}-3\wp Z-\wp^{\prime},
\]%
\begin{align*}
Z^{(3,0,0,0)}(r,s,\tau)\text{ }= &  Z^{6}-15\wp Z^{4}-20\wp^{\prime}%
Z^{3}+\left(  \tfrac{27}{4}g_{2}-45\wp^{2}\right)  Z^{2}\\
&  -12\wp\wp^{\prime}Z-\tfrac{5}{4}(\wp^{\prime})^{2},
\end{align*}
where $Z$ $=$ $Z(r,s,\tau)$, $\wp$ $=$ $\wp(r+s\tau)$, and $\wp^{\prime}$ $=$
$\wp^{\prime}(r+s\tau)$.

For equation (\ref{MFE1}) with singularities greater than one, a significant
distinction arises: the existence of solutions does not always imply the
existence of even solutions. Nevertheless, the concept of a pre-modular form
has also been applied to equation (\ref{MFE1}) in the quest for an even family
of solutions, a relevant pre-modular form $Z^{(m_{0},m_{1},m_{2},m_{3})}$
$(r,s,\tau)$ for equation (\ref{MFE1}) has also been derived in a series of
papers \cite{Chen-Kuo-Lin-Lame I,Chen-Kuo-Lin-Lame II}. The following result
has been proven:\textit{\medskip}

\noindent\textbf{Theorem A. }\textit{(\cite{Chen-Kuo-Lin-Lame II}) There
exists a pre-modular form }$Z^{\mathbf{(}m_{0},m_{1},m_{2},m_{3})}(r,s,\tau)
$\textit{\ of weights }$\sum_{k=0}^{3}$ $m_{k}(m_{k}$ $+1)/2$\textit{\ such
that the equation (\ref{MFE1})}$_{\tau}$\textit{\ has an even family of
solutions if and only if there is }$(r,s)$ $\in$ $\mathbb{R}^{2}%
\backslash\frac{1}{2}\mathbb{Z}^{2}$\textit{\ such that }$\tau$\textit{\ is a
zero of }$Z^{\mathbf{(}m_{0},m_{1},m_{2},m_{3})}(r,s,\cdot)$\textit{.\medskip}

We also present examples of general pre-modular forms as follows:
\[
Z^{(1,1,0,0)}(r,s,\tau)\text{ }=\text{ }Z^{2}-\wp+e_{1},\text{ \ }%
Z^{(1,0,1,0)}(r,s,\tau)\text{ }=\text{ }Z^{2}-\wp+e_{2},
\]%
\[
Z^{(1,0,0,1)}(r,s,\tau)\text{ }=\text{ }Z^{2}-\wp+e_{3},
\]%
\[
Z^{(2,1,0,0)}(r,s,\tau)\text{ }=\text{ }Z^{4}+3(e_{1}-2\wp)Z^{2}-4\wp
^{\prime2}+e_{1}\wp+e_{1}^{2}-\tfrac{g_{2}}{4}).
\]
Partial results regarding the existence of even solutions to equation
(\ref{MFE1}) for specific torus shapes have been obtained. Recently, in the
paper \cite{Chen-Lin-sharp nonexistence}, it was proven that if the parameters
$m_{k}$ \textit{do not} satisfy either%
\begin{equation}
\frac{m_{1}+m_{2}-m_{0}-m_{3}}{2}\geq1,\text{ }m_{1}\geq1,\text{ }m_{2}%
\geq1,\label{159}%
\end{equation}
or%
\begin{equation}
\frac{m_{1}+m_{2}-m_{0}-m_{3}}{2}\leq-1,\text{ }m_{0}\geq1,\text{ }m_{3}%
\geq1,\label{160}%
\end{equation}
then for any rectangle torus $\tau$ $\in i\mathbb{R}_{>0}$, equation
(\ref{MFE1})\ has no even solutions. In particular, parameters $(n,$ $0,$ $0,$
$0) $ do not satisfy either (\ref{159}) or (\ref{160}), meaning that equation
(\ref{163}) for rectangle tori has no even solution and, consequently, has no solutions.

The conditions specified in (\ref{159}) and (\ref{160}) regarding the $m_{k}%
$'s for rectangle tori are considered sharp. In the paper
\cite{Eremenko-.Gabrielov-on metrics}, Eremenko and Gabrielov demonstrated
that equation (\ref{MFE1}) has an even and symmetric ($u(z)$ $=$ $u(\bar{z})$)
solution on certain rectangle trous $\tau$ $\in i\mathbb{R}_{>0}$ if either
(\ref{159}) or (\ref{160}) is satisfied.

Although Theorem A has yielded some insights into even solutions, a
comprehensive understanding of their structure for equation (\ref{MFE1})
across all $\tau$ requires an examination of the distribution of $\tau$ zeros
within the pre-modular forms $Z^{\mathbf{(}m_{0},m_{1},m_{2},m_{3})}%
(r,s,\tau)$ for $\left(  r,s\right)  $ $\in$ $\mathbb{R}^{2}$ $\setminus$
$\frac{1}{2}\mathbb{Z}^{2}$. However, these pre-modular forms exhibit
substantial complexity, making it impossible to express them explicitly. This
complexity poses a significant challenge in the study of their zero
distribution, with only a limited number of them having been investigated thus far.

The first important case is $Z(r,s,\tau)$ itself, which represents the case
for equation (\ref{1631}) and plays a fundamental role in our study. In
\cite{Chen-Kuo-Lin-Wang-JDG}, we have unveiled its profound connection with
the Painlev\'{e} VI equation and have explored the zero distribution of
$Z(r,s,\tau)$ for $\left(  r,s\right)  \in$ $\mathbb{R}^{2}$ $\setminus$
$\frac{1}{2}\mathbb{Z}^{2}$. Given that $Z(r,s,\tau),$ as defined by
(\ref{zrs}), adheres to certain modularity properties (\ref{609}), and
$Z(r,s,\tau)$ $=$ $\pm Z(r^{\prime},s^{\prime},\tau),$ for any $\left(
r^{\prime},s^{\prime}\right)  $ $\equiv$ $\left(  r,s\right)  $ mod
$\mathbb{Z}^{2}$, it suffices to consider $Z(r,s,\tau)$ for $\left(
r,s\right)  $ $\in$ $\square$ and $\tau$ $\in$ $F_{0}$, where
\[
\square=[0,1/2]\times\lbrack0,1]\setminus\frac{1}{2}\mathbb{Z}^{2},
\]
and%
\[
F_{0}=\left\{  \tau\in\mathbb{H}|0\leq\operatorname{Re}\tau\leq1,\left\vert
\tau-\frac{1}{2}\right\vert \geq\frac{1}{2}\right\}  .
\]

\noindent\textbf{Theorem} \textbf{B. }\textit{(\cite[Theorem 1.3.]%
{Chen-Kuo-Lin-Wang-JDG}) Let }$\left(  r,s\right)  $ $\in$ $\square$\textit{.
Then }$Z(r,s,\tau)$\textit{\ has a zero in }$\tau\in F_{0}$\textit{\ if and
only if }$\left(  r,s\right)  $ $\in$ $\Delta_{0}$\textit{. Moreover, the zero
}$\tau$ $\in$ $F_{0}$\textit{\ is unique.}\medskip

Set%
\[
\Lambda=\left\{  \tau\in F_{0}\left\vert Z(r,s,\tau)=0\text{ for some }\left(
r,s\right)  \in\Delta_{0}\right.  \right\}  .
\]
\medskip

\noindent\textbf{Theorem} \textbf{C. }\textit{(\cite{Chen-Kuo-Lin-Wang-JDG})
The geometry of }$\Lambda$\textit{\ is as follows (see Figure \ref{P1}%
):}\medskip

\textit{\noindent(i) }$\Lambda$ $\subset$ $\left\{  \left\vert \tau-\frac
{1}{2}\right\vert >\frac{1}{2}\right\}  $\textit{\ is a simply connected
domain in }$F_{0}$\textit{\ that is symmetric with respect to }%
$\operatorname{Re}\tau=1/2.$\medskip

\textit{\noindent(ii) }$\partial\Lambda$ $=$ $%
{\displaystyle\bigcup\limits_{i=1}^{3}}
C_{i}$\textit{, where }$C_{i},i=1,2,3,$\textit{\ are smooth connected curves.
}\medskip

\textit{\noindent(iii) Each }$C_{i}$\textit{\ represents the degenerate curve
of the Green function }$G(z|\tau)$ \textit{at} $\frac{\omega_{i}}{2}$
\textit{and} \textit{connects any two cusps }$\{0,$\textit{\ }$1,$%
\textit{\ }$\infty\}$\textit{\ of }$F_{0}$\textit{. Moreover, we have }$C_{1}%
$\textit{\ connecting }$0$\textit{\ and }$1$\textit{, }$C_{2}$%
\textit{\ connecting }$0 $\textit{\ and }$\infty$\textit{, and }$C_{3}%
$\textit{\ connecting }$1$\textit{\ and }$\infty$\textit{.}\medskip
\begin{figure}[h]
\centering
\includegraphics[width=0.7\textwidth]{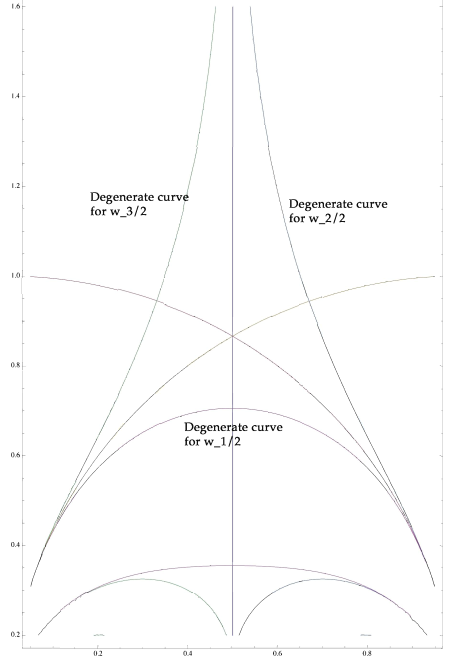}\caption{Doamin $\Lambda$}%
\label{P1}%
\end{figure}

Theorems B and C provide a comprehensive solution to the structure of the even
family of solutions, equivalently the entire solution structure, to the
equation (\ref{1631}) for \textbf{all} $\tau$ but not just rectangle tori. As
previously mentioned, due to the difficulty to derive the explicit expressions
of those pre-modular forms, studying the zero distribution of $Z^{(m_{0}%
,m_{1},m_{2},m_{3})}(r,s,\tau)$, similar to Theorems B and C for $Z(r,s,\tau
)$, is difficult in general, even for the case $Z^{(n,0,0,0)}(r,s,\tau)$. The
structure of even solutions to the equation (\ref{MFE1}) for general
parameters remains a formidable challenge.

Now, let's consider the\textbf{\ non-even} family of solutions. Unlike the
equation (\ref{163}), which has only one singular point, we need to focus not
only on even solutions. To understand the entire solution structure to
equation (\ref{MFE1}), we must also study the non-even family of solutions,
which seems unclear in the literature.

Here we remark that if $u$ is an even solution to the equation (\ref{MFE1}),
then the constants $T_{j},$ where $j$ $=$ $1,$ $2,$ $3,$ in the potential
(\ref{potential-1}) must be $0$ for all $j$; however, if $u$ is contained in a
non-even family of solutions, then the constants $T_{j}$ are generally
non-zero, which further complicates the task of investigating the structure of
non-even solution families.

The main purpose of this paper is to investigate the family of non-even
solutions for the first nontrivial case:%
\begin{equation}
\Delta u+e^{u}=8\pi(\delta_{0}+\delta_{\frac{\omega_{k}}{2}})\text{ in
}E_{\tau}\text{, }\tau\in\mathbb{H}\text{.}\label{164}%
\end{equation}
In other words, we focus on cases where $(m_{0},$ $m_{1},$ $m_{2},$ $m_{3})$
$=$ $(1,$ $1,$ $0,$ $0),$ $(1,$ $0,$ $1,$ $0),$ $(1,$ $0,$ $0,$ $1)$ for $k$
$=$ $1,$ $2,$ $3,$ respectively. Additionally, we aim to integrate the
findings from even solutions to provide a comprehensive understanding of the
equation (\ref{164}).

To simplify the notation, we set
\begin{equation}
\mathbf{m}_{k}\mathbf{=}\left\{
\begin{array}
[c]{c}%
(1,1,0,0)\text{ if }k=1\\
\\
(1,0,1,0)\text{ if }k=2\\
\\
(1,0,0,1)\text{ if }k=3
\end{array}
\right.  .\label{m}%
\end{equation}
In the following discussion, we begin by recalling the result of even
solutions to the equation (\ref{164}). Thanks to Theorem A, it is seen that
the equation (\ref{164}) on $E_{\tau}$ possesses an even solution if and only
if there exists a pair $\left(  r,\text{ }s\right)  $ $\in$ $\mathbb{R}^{2}$
$\setminus$ $\frac{1}{2}\mathbb{Z}^{2}$ such that $Z^{(\mathbf{m}_{k}%
)}(r,s,\tau)$ $=$ $0$. An intriguing outcome, as established in Lemma
\ref{lem11}, is that $Z^{(\mathbf{m}_{k})}(r,s,\tau)$ can be decomposed as
follows:%
\[
Z^{(\mathbf{m}_{1})}(r,s,\tau)\text{ }=\text{ }4Z(r,\frac{s}{2},2\tau)\cdot
Z(r,\frac{s+1}{2},2\tau),
\]%
\[
Z^{(\mathbf{m}_{2})}(r,s,\tau)\text{ }=\text{ }Z(\frac{r}{2},s,\frac{\tau}%
{2})\cdot Z(\frac{r+1}{2},s,\frac{\tau}{2}),
\]
and%
\[
Z^{(\mathbf{m}_{3})}(r,s,\tau)\text{ }=\text{ }Z(\frac{r-s}{2},s,\frac{1+\tau
}{2})\cdot Z(\frac{r-s+1}{2},s,\frac{1+\tau}{2}).
\]
From these equations, we see that the study of even solutions of the equation
(\ref{164}) on $E_{\tau}$ is analogous to that of equation (\ref{1631}), which
has one singular point, on a relevant torus. This simplification for even
solutions has been carried out in \cite{Chen-Kuo-Lin-8pi+8pi}, where the
following was proved:\medskip

\noindent\textbf{Theorem D. }\textit{(\cite{Chen-Kuo-Lin-8pi+8pi}) The
equation (\ref{164}) has an even family of solutions if and only if the
equation (\ref{1631}) defined }$E_{k,\tau}$\textit{\ has a solution (or an
even solution), where }%
\begin{equation}
E_{k,\tau}=\left\{
\begin{array}
[c]{l}%
\mathbb{C}/\Lambda_{\frac{1}{2},\tau}\text{ if }k=1\\
\\
\mathbb{C}/\Lambda_{1,\frac{\tau}{2}}\text{ if }k=2\\
\\
\mathbb{C}/\Lambda_{1,\frac{1+\tau}{2}}\text{ if }k=3
\end{array}
\right.  .\label{Ek}%
\end{equation}

We are now prepared to present our primary finding regarding the non-even
family of solutions for equation (\ref{164}). Typically, the structure of
non-even families of solutions is not anticipated to be associated with an
equation with only one singular point. But surprisingly, the criterion of
non-even families of solutions to equation (\ref{164}) is interconnected with
equation (\ref{1631}) for the same $\tau$.

\begin{theorem}
\label{thm21 copy(1)}Equation (\ref{164}) defined on $E_{\tau_{0}}$, where
$\tau_{0}\in\mathbb{H}$, has a non-even family of solutions if and only if
equation (\ref{1631}), defined on the same $E_{\tau_{0}}$, has a solution,
i.e., there is $(r,s)$ $\in$ $\mathbb{R}^{2}$ $\setminus$ $\frac{1}%
{2}\mathbb{Z}^{2}$ such that $Z(r,s,\tau_{0})$ $=$ $0$.
\end{theorem}

We summarize the above discussions in the following theorem.

\begin{theorem}
\label{thm25}

\noindent(\textbf{i}) The following are equivalent for an \textbf{even} family
of solutions to equation (\ref{164}) defined on $E_{\tau_{0}}$, $\tau_{0}%
\in\mathbb{H}$:\medskip

\noindent(i-1) Equation (\ref{164}) defined on $E_{\tau_{0}}$, $\tau_{0}%
\in\mathbb{H}$ has an even family of solutions.\medskip

\noindent(i-2) Equation (\ref{1631}) defined on $E_{k,\tau_{0}}$ has a
solution, where $E_{k,\tau_{0}}$ is defined in (\ref{Ek})\medskip

\noindent(i-3) There is $(r,s)$ $\in$ $\mathbb{R}^{2}\backslash\frac{1}%
{2}\mathbb{Z}^{2}$ such that $Z(r,s,\tau_{k})$ $=0,$ where $\tau_{k}$ $=$
$2\tau_{0},$ $\frac{\tau_{0}}{2},$ $\frac{1+\tau_{0}}{2}$ for $k$ $=$ $1,$
$2,$ $3,$ respectively.\medskip

\noindent(\textbf{ii}) The following are equivalent for a \textbf{non-even}
family of solutions to equation (\ref{164}) defined on $E_{\tau_{0}}$,
$\tau_{0}\in\mathbb{H}$:\medskip

\noindent(ii-1)Equation (\ref{164}) defined on $E_{\tau_{0}}$, $\tau_{0}%
\in\mathbb{H}$ has a non-even family of solutions.\medskip

\noindent(ii-2)Equation (\ref{1631}) defined on the same $E_{\tau_{0}}$ has a
solution.\medskip

\noindent(ii-3) There is $(r,s)$ $\in$ $\mathbb{R}^{2}\backslash\frac{1}%
{2}\mathbb{Z}^{2}$ such that $Z(r,s,\tau_{0})$ $=0$.\medskip
\end{theorem}

Considering the comprehensive analysis of the zero distribution of
$Z(r,s,\tau)$ for $(r,s)$ $\in$ $\mathbb{R}^{2}\setminus\frac{1}{2}%
\mathbb{Z}^{2}$ in Theorems B and C, it can be concluded that Theorem
\ref{thm25} provides a complete solution to the equation (\ref{164}) for
\textbf{all }$\tau$.

In view of Theorem \ref{thm25}, one might anticipate that the examination of
the non-even family of solutions for the cases where $(m_{0},$ $m_{1},$
$m_{2},$ $m_{3})$ $=$ $(n,$ $n,$ $0,$ $0),$ $(n,$ $0,$ $n,$ $0),$ $(n,$ $0,$
$0,$ $n),$ as given below:%
\begin{equation}
\Delta u+e^{u}=8\pi n(\delta_{0}+\delta_{\frac{\omega_{k}}{2}})\text{ in
}E_{\tau}\text{, }\tau\in\mathbb{H}\text{.}\label{MFE2}%
\end{equation}
can be effectively described through equation (\ref{163}). In practice, this
is an particularly challenging problem primarily due to the absence of an
explicit expression of $Z^{(n,0,0,0)}(r,s,\tau)$ as well as $Z^{(n,n,0,0)}%
(r,s,\tau)$, $Z^{(n,0,n,0)}(r,s,\tau)$, $Z^{(n,0,0,n)}(r,s,\tau)$. To overcome
this difficulty, our forthcoming paper will address this issue by employing
the Darboux transformation technique. Theorem \ref{thm25} for the case where
$n=1$ in equation (\ref{MFE2}) will serve as the initial step in our inductive argument.

Finally, we shall illustrate the applications of Theorem \ref{thm25} for two
specific types of tori: rectangle tori $\tau$ $=$ $ib$ $\in$ $i\mathbb{R}%
_{>0}$ and rhombus torus $\tau$ $=$ $\mathring{\tau}\doteqdot e^{\frac{\pi
i}{3}}.$ By Theorem C, we see that $Z(r,s,\tau)\mathit{\ }\not =$ $0\ $for
any$\mathit{\ }\left(  r,s\right)  $ $\in$ $\mathbb{R}^{2}$ $\setminus$
$\frac{1}{2}\mathbb{Z}^{2}\mathit{\ }$\textit{and }$\tau$ $\in$ $\partial
F_{0}$. Considering $\tau$ $=$ $ib$ $\in$ $i\mathbb{R}_{>0}\subset\partial
F_{0}$, we conclude that $Z(r,s,ib)\mathit{\ }\not =$ $0\mathit{\ }$for
any$\mathit{\ }\left(  r,s\right)  \mathit{\ }\in$ $\mathbb{R}^{2}$
$\setminus$ $\frac{1}{2}\mathbb{Z}^{2}$ and $b>0$. This, coupled with the
(i-3) and (ii-3), leads to the conclusion that equation (\ref{164}) on
rectangle tori for $k$ $=1,$ $2,$ has no solutions. In the case of $k$ $=$
$3$, by (ii-3), we see that equation (\ref{164}) still lacks non-even families
of solutions for any $\tau$ $=$ $ib$. However, by (i-3), the even solutions
are determined by the zeros of $Z(r,s,\frac{1}{2}+\frac{ib}{2})$, $\left(
r,s\right)  $ $\in$ $\mathbb{R}^{2}$ $\setminus$ $\frac{1}{2}\mathbb{Z}^{2}$.
Although the structure of even solutions for equation (\ref{164}) on rectangle
tori has been thoroughly examined in \cite[Theorem 1.2.]{Chen-Kuo-Lin-8pi+8pi}%
, we will briefly touch upon it here for completeness.

According to Theorem C, for $\tau=\frac{1}{2}+\frac{ib}{2}$ and $b\geq1$,
there is $b_{0}$\ with $1$ $<$ $b_{0}$ $<$ $\sqrt{3}$\ such that
$Z(r,s,\frac{1}{2}+\frac{ib}{2})=0$ has a solution for certain $\left(
r,s\right)  $ $\in$ $\mathbb{R}^{2}$ $\setminus$ $\frac{1}{2}\mathbb{Z}^{2}$
as long as $b>b_{0}$, where the determination of $b_{0}$ is such that
$\frac{1}{2}$ $+$ $\frac{ib_{0}}{2}$ corresponds to the point of intersection
between $\operatorname{Re}\tau$ $=$ $\frac{1}{2}$ and the degenerate curve
$C_{1}$. Now, for $\tau=\frac{1}{2}+\frac{ib}{2}$ and $0$ $<$ $b$ $\leq$ $1$,
we can utilize the modularity property (\ref{609}) of $Z(r,s,\tau)$ with
$\gamma$ $=\left(
\begin{array}
[c]{cc}%
1 & -1\\
2 & -1
\end{array}
\right)  $ $\in$ $SL(2,\mathbb{Z})$ to derive the following equation:
\begin{equation}
Z(r^{\prime},s^{\prime},\frac{1}{2}+\frac{i}{2b})=ib\cdot Z(r,s,\frac{1}%
{2}+\frac{ib}{2}),\label{modular}%
\end{equation}
where
\[
\left(  r^{\prime},s^{\prime}\right)  =\left(  r-2s,r-s\right)
\]
for any $b>0$. By (\ref{modular}), we can deduce that $Z(r,s,\frac{1}{2}%
+\frac{ib}{2})$ $=$ $0$ for $0$ $<$ $b$ $\leq$ $1$ possesses a solution for
certain $\left(  r,s\right)  $ $\in$ $\mathbb{R}^{2}$ $\setminus$ $\frac{1}%
{2}\mathbb{Z}^{2}$ provided $0<b<\frac{1}{b_{0}}$.

Suppose $u(z)$ is an even solution to the equation (\ref{164}) for $k$ $=$ $3
$. By the uniqueness of solutions of equation (\ref{1631}), this particular
even solution $u(z)$ must be unique. Furthermore, $u(z)$ is also symmetric,
which follows from both $u(z)$ and $u(\bar{z})$ are even solutions. Therefore,
we obtain the following result:\medskip

\begin{corollary}
\label{cor18 copy(1)}\textit{Let }$\tau$ $=ib$ $\in i\mathbb{R}_{>0}%
$.\textit{\ Then the equation (\ref{164}) has no non-even family of solutions
for }$k$ $=$ $1,$ $2,$ $3.$ More precisely, we have the following:\medskip

\noindent(i) For $k$ $=$ $1,$ $2$, \textit{equation (\ref{164}) has no
solutions.}\medskip

\noindent(ii) For $k$ $=$ $3$, \textit{equation (\ref{164}) has either only
one even family of solutions or no solutions. More precisely, }the following
hold:\medskip

\noindent(ii-1) Equaiton (\textit{\ref{164}}) with $k$ $=$ $3$ has an even
family of solutions if and only if $b$ $\in$ $(0,\frac{1}{b_{0}})\cup
(b_{0},\infty)$, where $b_{0}$ is determined in a way such that $\frac{1}{2} $
$+$ $\frac{ib_{0}}{2}$ is the intersection point between $\operatorname{Re}%
\tau$ $=$ $\frac{1}{2}$ and the degenerate curve $C_{1}$. Moreover, this even
solution is unique and symmetric.\medskip

\noindent(ii-2) Equation (\textit{\ref{164}}) with $k$ $=$ $3$ has no non-even
family of solutions for any $b$ $>0$.\medskip
\end{corollary}

Now, we consider the equation (\ref{164}) defined on the rhombus torus
$\tau=\mathring{\tau}.$ It has been established in
\cite{Chen-Kuo-Lin-Wang-JDG} that $Z(\frac{1}{3},\frac{1}{3},\mathring{\tau
}=e^{\frac{\pi i}{3}})$ $=$ $0$. With the aid of Theorem \ref{thm25} (ii-3),
we see that equation (\ref{164}) has a unique non-even family of solutions for
each $k$ $=$ $1,$ $2,$ $3$. However, since $2\mathring{\tau}$, $\frac
{\mathring{\tau}}{2}$ and $\frac{\mathring{\tau}+1}{2}$ all fall within the
boundary $\partial F_{0}$, as indicated by Theorem \ref{thm25} (i-3), it does
not admit any even family of solutions.

\begin{corollary}
\label{cor2}For each $k$ $=$ $1,$ $2,$ $3$, equation \textit{(\ref{164}) with
}$\tau=\mathring{\tau}=e^{\frac{\pi i}{3}}$ \textit{has a unique non-even
family of solutions but no even-family of solutions.}
\end{corollary}

The organization of this paper is as follows: \medskip

In Section \ref{Integrability}, we discuss the integrability of the curvature
equation (\ref{MFE1}) with four singular points at $\frac{\omega_{k}}{2}$, $k$
$=$ $0,$ $1,$ $2,$ $3$, for general parameters $m_{k}$ $\in\mathbb{Z}_{\geq0}$
for all $k$. Our primary focus in this section is to translate this subtle
issue from a partial differential equation (PDE) to an associated second-order
Fuchsain elliptic ODE and subsequently derive the monodromy constraints. Refer
to Theorem \ref{thm19}.

Sections \ref{BA function} and \ref{Monodromy theory} are dedicated to the
scenario for equation (\ref{164}), with the associated ODE being equation
(\ref{199}). Here, we introduce the Baker-Akhiezer functions and the spectral
curve to investigate the monodromy representation of the equation (\ref{199}).
Our main objective in these sections is to is to formulate a related
pre-modular form capable of describing the monodromy data of the equation
(\ref{199}). See Theorem \ref{main thm 1}.

In Section \ref{Non even family}, we explore the criteria (\ref{C1}) for the
non-even family of solutions to the equation (\ref{164}). Indeed, (\ref{C1})
contains two parts: Firstly, $Z(r_{k},s_{k},\tau_{0})$ $=$ $0$, where $\left(
r_{k},s_{k}\right)  $ is given by (\ref{rsk}), corresponding to the existence
of a solution $u$ for the equation (\ref{164}) on $E_{\tau_{0}}$. Secondly,
$Z^{(\mathbf{m}_{k})}(r,s,\tau_{0})$ $\not =$ $0$, designed to exclude the
possibility that $u$ might belong to an even family of solutions.

Lastly, in Section \ref{The proof}, we employ the modularity properties of
these pre-modular forms to demonstrate that $Z(r_{k},s_{k},\tau_{0})$ and
$Z^{(\mathbf{m}_{k})}(r,s,\tau_{0})$ cannot have common zeros in $\tau$
variable for any $\left(  r,s\right)  $ $\in$ $\mathbb{R}^{2}$ $\setminus$
$\frac{1}{2}\mathbb{Z}^{2}$. This proof is detailed in Theorem \ref{thm24}.

\textbf{Acknowledgement:} The author is grateful to Professor Chin-Lung Wang
for providing the Figure 1.

\section{Integrability of the curvature equation}

\label{Integrability}

Recall the curvature equation
\begin{equation}
\Delta u+e^{u}=8\pi\sum_{k=0}^{3}m_{k}\delta_{\frac{\omega_{k}}{2}}\text{ in
}E_{\tau}\text{, }\tau\in\mathbb{H}.\label{MFE}%
\end{equation}
In the following, we always assume $m_{k}\in\mathbb{Z}_{\geq0}$ for all $k$
with $\sum m_{k}\geq1$. In this section, we study the curvature equation
(\ref{MFE}) from the integrability point of view.

\subsection{Normalized developing map}

\label{Normalized developing map}

According to the Liouville theorem, for any solution $u$, there is a
developing map $f$ such that
\begin{equation}
u(z)=\log\frac{8\left\vert f^{\prime}(z)\right\vert ^{2}}{\left(  1+\left\vert
f(z)\right\vert ^{2}\right)  ^{2}}.\label{Liouville theorem}%
\end{equation}
To recover $f$ from $u$, we note that $u_{zz}$ $-$ $\frac{1}{2}u_{z}^{2}$ is
doubly periodic because $u$ is. In fact, $u_{zz}$ $-$ $\frac{1}{2}u_{z}^{2}$
is not only doubly periodic but also an elliptic function as well. To prove
this fact, it is equivalent to proving the following two
properties:\textit{\medskip}

\noindent(i)
\[
\left(  u_{zz}-\frac{1}{2}u_{z}^{2}\right)  _{\bar{z}}=0\text{ for }z\in
E_{\tau}\backslash\{ \frac{\omega_{k}}{2},k=0,1,2,3\} \text{ }%
\]
and\textit{\medskip}

\noindent(ii)
\begin{equation}
u_{zz}-\frac{1}{2}u_{z}^{2}=\frac{-2m_{k}(m_{k}+1)}{\left(  z-\frac{\omega
_{k}}{2}\right)  ^{2}}+O\left(  \left(  z-\frac{\omega_{k}}{2}\right)
^{-1}\right)  \text{ near }\frac{\omega_{k}}{2}.\label{u elliptic}%
\end{equation}
These two properties imply that $u_{zz}-\frac{1}{2}u_{z}^{2}$ is a doubly
periodic meromorphic function with poles of order $2$ at $\frac{\omega_{k}}%
{2}$, $k=0,$ $1,$ $2,$ $3,$ and therefore it is an elliptic function. The
property (i) can be derived easily by using equation (\ref{MFE}). To see (ii):
Let $v(z):=u(z)$ $-4m_{k}\ln|z-\frac{\omega_{k}}{2}|$. By the equation
(\ref{MFE}), it is seen that
\begin{equation}
u(z)=4m_{k}\ln|z-\frac{\omega_{k}}{2}|+O(1)\text{ near }\frac{\omega_{k}}%
{2}.\label{asymptotic}%
\end{equation}
Applying (\ref{asymptotic}) into $u_{zz}-\frac{1}{2}u_{z}^{2}$, and
simplifying, we obtain (\ref{u elliptic}).

By differentiating (\ref{502}), we obtain the following important identity:
\begin{equation}
u_{zz}-\frac{1}{2}u_{z}^{2}=\{f;z\}\doteqdot\left(  \frac{f^{\prime\prime}%
}{f^{\prime}}\right)  ^{\prime}-\frac{1}{2}\left(  \frac{f^{\prime\prime}%
}{f^{\prime}}\right)  ^{2},\label{ccc}%
\end{equation}
here $\{f;z\}$ denotes the Schwarzian derivative of $f$. According to the
classical Schwarzian theory, any two developing maps $f_{1}$ and $f_{2}$ of
the same solution $u$ must satisfy:
\begin{equation}
f_{2}(z)=U\cdot f_{1}(z):=\frac{af_{1}(z)+b}{cf_{1}(z)+d}\label{cc}%
\end{equation}
for some $U=%
\begin{pmatrix}
a & b\\
c & d
\end{pmatrix}
$ $\in$ $SL(2,\mathbb{C})$. Furthermore, by substituting (\ref{cc}) into
(\ref{502}), a direct computation shows $U\in SU(2)$, i.e.%
\begin{equation}
U=%
\begin{pmatrix}
a & -\bar{b}\\
b & \bar{a}%
\end{pmatrix}
\text{, with }|a|^{2}+|b|^{2}=1.\label{SU(2) connection}%
\end{equation}
It's important to note that $f(z)$ is not an elliptic function in general;
rather, it is $SU(2)$-automorphic for the periods of $\Lambda_{\tau}$. Because
$f(z+\omega_{1})$ and $f(z+\omega_{2})$ are also developing maps of the same
$u(z)$, (\ref{cc}) implies there are $U_{i}\in SU(2)$ such that:
\begin{equation}
f(z+\omega_{1})=U_{1}\cdot f(z)\ \ \text{and}\ \ f(z+\omega_{2})=U_{2}\cdot
f(z).\label{new26}%
\end{equation}
The constraints on $U_{1}$ and $U_{2}$ had been discussed in \cite{LW-AnnMath}%
, but for our self-contained explanation, we briefly discuss them here.
Firstly, the conditions (\ref{new26}) also force the following commutative of
$U_{1}$ and $U_{2}$:
\begin{equation}
U_{1}U_{2}=\pm U_{2}U_{1}\label{ab}%
\end{equation}
By conjugating a common matrix, we can always normalize $U_{1}$ to $%
\begin{pmatrix}
e^{i\theta_{1}} & 0\\
0 & e^{-i\theta_{1}}%
\end{pmatrix}
,$ where $\theta_{1}\in\mathbb{R}$. Denote $U_{2}$ to be $%
\begin{pmatrix}
a & -\bar{b}\\
b & \bar{a}%
\end{pmatrix}
$ with $|a|^{2}+|b|^{2}=1.$ By (\ref{ab}), we have three possibilities:

\noindent(1) $a=0$ and $e^{i\theta_{1}}=\pm i,$ which corresponds to the case
$U_{1}U_{2}=-U_{2}U_{1}$,\textit{\medskip}

\noindent(2) $b=0;$ or

\noindent(3) $e^{i\theta_{1}}=\pm1$.

The last two cases correspond to the case $U_{1}U_{2}=U_{2}U_{1}$. However,
notice that case (3) implies that $U_{1}$ is $\pm Id,$ and by another
conjugation, we may assume $U_{2}=$ $%
\begin{pmatrix}
e^{i\theta_{2}} & 0\\
0 & e^{-i\theta_{2}}%
\end{pmatrix}
,$ where $\theta_{2}\in\mathbb{R}$, which is reduced to case
(2).\textit{\medskip}

According to the above discussion, we have two types of $f$:\textit{\medskip}

\noindent(Type I):
\[
f(z+\omega_{1})=-f(z)\text{ and }f(z+\omega_{2})=-\frac{\bar{b}^{2}}{f(z)},
\]
\textit{\medskip}

\noindent(Type II):
\[
f(z+\omega_{1})=e^{2i\theta_{1}}f(z)\text{ and }f(z+\omega_{2})=e^{2i\theta
_{2}}f(z),\text{ }\theta_{i}\in\mathbb{R},i=1,2.
\]
A developing map $f$ is said to be a \textit{normalized developing map} if it
satisfies either type I or type II. Indeed, when $m_{k}\in\mathbb{Z}_{\geq0}$
for all $k$ in (\ref{MFE}), it is proved that any developing map can be
normalized to satisfy the type II constraint. See Theorem \ref{thm19}.

\begin{lemma}
\label{lem8}Suppose $f$ is a meromorphic function. We define:
\[
u=\log\frac{8\left\vert f^{\prime}\right\vert ^{2}}{(1+\left\vert f\right\vert
^{2})^{2}}.
\]

Then $u=4m\log\left\vert z-z_{0}\right\vert +O(1)$ holds near $z_{0}\in
E_{\tau}$ for some $m\in\mathbb{Z}_{\geq0}$ if and only if $f$ satisfies
either%
\begin{equation}
f(z)=a_{0}+a_{2m+1}(z-z_{0})^{2m+1}+\cdot\cdot\cdot,\text{ }a_{2m+1}%
\not =0,\label{193}%
\end{equation}
or
\begin{equation}
f(z)=a_{-2m-1}(z-z_{0})^{-2m-1}+\cdot\cdot\cdot,\text{ }a_{-2m-1}%
\not =0\label{194}%
\end{equation}
near $z_{0}.$
\end{lemma}

\begin{proof}
Let $z_{0}\in E_{\tau}$ and $\ell=ord_{z=z_{0}}f$. Since $f$ is an meromorphic
function, we may write $f(z)=\left(  z-z_{0}\right)  ^{\ell}h(z)$, where
$h(z)$ is holomorphic at $z_{0}$ with $h(z_{0})\not =0$. Then we have
\[
f^{\prime}(z)=\ell(z-z_{0})^{\ell-1}h(z)+(z-z_{0})^{\ell}h^{\prime}(z),
\]
and so%
\[
u=\log\frac{8\left\vert f^{\prime}\right\vert ^{2}}{(1+\left\vert f\right\vert
^{2})^{2}}=\log8+2\log\frac{\left\vert z-z_{0}\right\vert ^{\ell-1}\left\vert
\ell h+(z-z_{0})h^{\prime}\right\vert }{1+\left\vert z-z_{0}\right\vert
^{2\ell}\left\vert h\right\vert ^{2}}.
\]
From here, it is easy to see that
\[
u=\left\{
\begin{array}
[c]{c}%
2\left(  \left\vert \ell\right\vert -1\right)  \log\left\vert z-z_{0}%
\right\vert +O(1)\text{ if }\ell\not =0\\
\\
2\log\left\vert h^{\prime}\right\vert +O(1)\text{ if }\ell=0
\end{array}
\right.
\]

Suppose $u(z)=4m\log\left\vert z-z_{0}\right\vert +O(1)$ near $z_{0}$. If
$\ell=0$, then we see that
\[
2\log\left\vert h^{\prime}\right\vert =2\log\left(  \left\vert h^{\prime
}\right\vert \right)  =4m\log|z-z_{0}|
\]
which implies that
\[
h^{\prime}(z_{0})=h^{\prime\prime}(z_{0})=\cdot\cdot\cdot=h^{(2m)}(z_{0})=0
\]
and
\[
h^{(2m+1)}(z_{0})\not =0.
\]
Therefore, we have%
\begin{equation}
f=a_{0}+a_{2m+1}(z-z_{0})^{2m+1}+\cdot\cdot\cdot,\text{ }a_{2m+1}%
\not =0\text{, if }\ell=0\text{.}\label{190}%
\end{equation}
If $\ell\not =0$, then
\[
2\left(  \left\vert \ell\right\vert -1\right)  \log\left\vert z-z_{0}%
\right\vert =4m\log|z-z_{0}|
\]
which yields
\[
\left\vert \ell\right\vert =2m+1.
\]
So we have either%
\begin{equation}
f=a_{2m+1}(z-z_{0})^{2m+1}+\cdot\cdot\cdot,\text{ }a_{2m+1}\not =0,\text{ if
}\ell=2m+1,\label{191}%
\end{equation}
or%
\begin{equation}
f=a_{-2m-1}(z-z_{0})^{-2m-1}+\cdot\cdot\cdot,\text{ }a_{-2m-1}\not =0,\text{
if }\ell=-\left(  2m+1\right)  .\label{192}%
\end{equation}
Clearly, (\ref{190}), (\ref{191}), and (\ref{192}) imply (\ref{193}) or
(\ref{194}). Conversely, suppose (\ref{193}) or (\ref{194}) holds. Then
\[
f^{\prime}=(2m+1)a_{2m+1}(z-z_{0})^{2m}(1+O(z-z_{0})),\text{ }a_{2m+1}%
\not =0,\text{ if (\ref{193}) holds}%
\]
or%
\[
f^{\prime}=-(2m+1)a_{-2m-1}(z-z_{0})^{-2m-2}(1+O(z-z_{0})),\text{ }%
a_{-2m-1}\not =0,\text{ if (\ref{194}) holds.}%
\]
Applying these into $u$, we see that
\[
u(z)=4m\log\left\vert z-z_{0}\right\vert +O(1)
\]
near $z_{0}$ for both cases. This completes the proof.
\end{proof}

\begin{theorem}
\label{main thm 2-1}The curvature equation (\ref{MFE}) has a solution if and
only if there is a normalized developing map $f$ that satisfies either
(\ref{193}) or (\ref{194}), where
\begin{equation}
m=\left\{
\begin{array}
[c]{l}%
0\text{, if }z_{0}\not \in \{ \frac{\omega_{k}}{2},k=0,1,2,3\}\\
\\
m_{k}\text{, if }z_{0}=\frac{\omega_{k}}{2}%
\end{array}
\right.  .\label{mm}%
\end{equation}
In such a case, $u(z),$ defined by (\ref{Liouville theorem}), is a solution to
the equation (\ref{MFE}).
\end{theorem}

\begin{proof}
The necessary part is a direct consequence of Lemma \ref{lem8}. For the
sufficient part, it is easy to check $u$ defined by (\ref{Liouville theorem})
satisfies
\[
\Delta u(z)+e^{u(z)}=0\text{ for all }z\not \in \{ \frac{\omega_{k}}{2}\}.
\]
By Lemma \ref{lem8} again, we have
\[
u=4m_{k}\log\left\vert z-\frac{\omega_{k}}{2}\right\vert +O(1)\text{ near
}\frac{\omega_{k}}{2}.
\]
Since $f$ satisfies either type I or type II, both cases imply that $u(z)$
satisfies
\[
u(z+\omega_{i})=u(z)\text{ for }i=1,2.
\]
Therefore, the sufficient part is also proven.
\end{proof}

\subsection{Unitary monodromy constraints}

\label{unitary monodromy}

Denote $q(z)=-\frac{1}{2}(u_{zz}-\frac{1}{2}u_{z}^{2}),$ and let $T_{k}$ be
the residues of the elliptic function $q(z)$ at $z$ $=$ $\frac{\omega_{k}}%
{2},$ $k$ $=$ $0,$ $1,$ $2,$ $3$. The ellipticness of $q(z)$ implies the sum
of its residues must be zero, i.e.,
\[
T_{0}=-\left(  T_{1}+T_{2}+T_{3}\right)  .
\]
Based on the behavior in (\ref{u elliptic}), we can express $q(z)$ in terms of
classical elliptic functions as follows:
\begin{equation}
q(z;T_{j},E)=\sum_{k=0}^{3}m_{k}(m_{k}+1)\wp(z+\frac{\omega_{i}}{2}%
)+\sum_{j=1}^{3}T_{j}\left(  \zeta(z-\frac{\omega_{j}}{2})-\zeta(z)\right)
-E\label{potential-1}%
\end{equation}
for some constant $E\in\mathbb{C}$.

Recalling (\ref{ccc}) that $-2q(z;T_{k},E)=\{f;z\}$. The Schwarzian derivative
has a fundamental relationship with second order linear ordinary differential
equations. Consider the second-order linear ODE of Fuchsian type:%
\begin{equation}
y^{\prime\prime}(z)=q(z;T_{k},E)y(z)\text{, }z\in E_{\tau}\text{, }\tau
\in\mathbb{H}\label{GLE1}%
\end{equation}
where the potential $q(z;T_{k},E)$ is derived from the Schwarzian derivative
of $f$ given in (\ref{potential-1}). Since $q(z;T_{k},E)$ is elliptic,
equation (\ref{GLE1}) can also be considered for $z\in\mathbb{C}$. If $T_{k} $
$=$ $0$ for all $k=0,1,2,3,$ (for example, in the case of $u$ being an
\textit{even} solution), equation (\ref{GLE1}) is reduced to elliptic form of
the Heun equation, denoted by H($\mathbf{m},E,\tau$), with $\mathbf{m}%
=(m_{0},m_{1},m_{2},m_{3})$, as follows:
\begin{equation}
y^{\prime\prime}(z)=\left(  \sum_{k=0}^{3}m_{k}(m_{k}+1)\wp(z+\frac{\omega
_{k}}{2})-E\right)  y(z)\text{, }z\in E_{\tau}\text{, }\tau\in\mathbb{H}%
.\label{Heun}%
\end{equation}
In history, this potential, $\sum_{k=0}^{3}m_{k}(m_{k}+1)\wp(z+\frac
{\omega_{k}}{2}),$ is known as the \emph{Darboux-Treibich-Verdier potential
}as referenced in \cite{TV}. It is recognized as an algebro-geometric
finite-gap potential associated with the stationary KdV hierarchy. For recent
developments in KdV theory and its relation with algebro-geometric finite-gap
potential, we refer readers to
\cite{EK-CMH,GH-Book,GH,GUW1,GW4,GW2,GW3,GW1,Ince,ItM,kri,Maier,Smirnov1989,Smirnov1994,TV,Weikard1999}
and references therein.

In general,\ if residues $T_{k}$ are not all zero, then the potential
$q(z;T_{k},E)$ is never even, and there might be a \textit{logarithmic
singularity} at some $\frac{\omega_{k}}{2}$. However, the following result
states that once equation (\ref{GLE1}) is derived from a solution $u$ of
equation (\ref{MFE}), it is automatically free of logarithmic singularity at
each singular point $\frac{\omega_{k}}{2}$. Equation (\ref{GLE1}) is said to
be \textit{apparent} at $\frac{\omega_{k}}{2}$ if it is free of logarithmic
singularity at $\frac{\omega_{k}}{2}$ and is called \textit{apparent }if it is
apparent at all $\frac{\omega_{k}}{2},$ $k=0,1,2,3$. In such a case, any
solution can be extended to $\mathbb{C\cup\{ \infty\}}$ meromorphically.

\begin{lemma}
\label{lem7}Suppose $u$ is a solution to equation (\ref{MFE}), which
associates equation (\ref{GLE1}). Then equation (\ref{GLE1}) is apparent at
each singularity.
\end{lemma}

\begin{proof}
Suppose equation (\ref{GLE1}) is not apparent at $\frac{\omega_{k}}{2}$. Let
$\xi=z-\frac{\omega_{k}}{2}$. Since the local exponents of equation
(\ref{GLE1}) at $\frac{\omega_{k}}{2}$ are $-m_{k}$ and $m_{k}+1$. There are
two linearly independent solutions, $\hat{y}_{1}$ and $\hat{y}_{2}$ of the
form:%
\[
\hat{y}_{1}=\xi^{m_{k}+1}\sum_{j=0}^{\infty}a_{j}\xi^{j}\text{,}%
\]
and
\[
\hat{y}_{2}=\xi^{-m_{k}}\sum_{j=0}^{\infty}b_{j}\xi^{j}+c\hat{y}_{1}\cdot
\log\xi\text{, where }c\not =0.
\]
Then the local monodromy at $\frac{\omega_{k}}{2}$ with respect to this
fundamental solution $\left(  \hat{y}_{1},\hat{y}_{2}\right)  ^{t}$ is seen
as
\[
\left(
\begin{array}
[c]{c}%
\hat{y}_{1}\\
\hat{y}_{2}%
\end{array}
\right)  (e^{2\pi i}\xi)=\left(
\begin{array}
[c]{cc}%
1 & 0\\
2\pi ic & 1
\end{array}
\right)  \left(
\begin{array}
[c]{c}%
\hat{y}_{1}(\xi)\\
\hat{y}_{2}(\xi)
\end{array}
\right)  .
\]
Since $\{f;z\}=\{ \frac{\hat{y}_{1}}{\hat{y}_{2}};z\}$, there is $\gamma\in
SL(2,\mathbb{C})$ such that $f=\gamma\cdot\frac{\hat{y}_{1}}{\hat{y}_{2}}$.
Since $f$ is meromorphic, we see that
\[
f(e^{2\pi i}\xi)=\gamma\cdot\left(
\begin{array}
[c]{cc}%
1 & 0\\
2\pi ic & 1
\end{array}
\right)  \cdot\frac{\hat{y}_{1}}{\hat{y}_{2}}=f(\xi)=\gamma\cdot\frac{\hat
{y}_{1}}{\hat{y}_{2}}%
\]
which yields that
\[
\gamma\cdot\left(
\begin{array}
[c]{cc}%
1 & 0\\
2\pi ic & 1
\end{array}
\right)  =\pm\gamma.
\]
Consequently, $c$ $=$ $0$, a contradiction.
\end{proof}

To state our main result in this section, we introduce the monodromy
representation of equation (\ref{GLE1}). See the definition for the
generalized Lam\'{e} equation as well as the Heun equation, in
\cite{Chen-Kuo-Lin-Lame I,Chen-Kuo-Lin-Lame II,Chen-Juo-Kin-Lame III}, along
with a discussion about the monodromy there.

Suppose equation (\ref{GLE1}) is apparent. The local exponents of equation
(\ref{GLE1}) at $\frac{\omega_{k}}{2}$ are $-m_{k}$ and $m_{k}+1$, where
$m_{k}\in\mathbb{Z}_{\geq0}$. Therefore, the local monodromy matrix must be
$Id_{2}$ there. We fix a fundamental system of solution $Y(z)$ of equation
(\ref{GLE1}), and the monodromy representation $\rho_{\tau}(E)$ of equation
(\ref{GLE1}) can be reduced to a group homomorphism:
\[
\rho_{\tau}(E):\pi_{1}(E_{\tau})\rightarrow SL(2,\mathbb{C})
\]
defined by%
\begin{equation}
\ell^{\ast}Y(z)=\rho_{\tau}(\ell;T_{k},E)Y(z)\label{monodromy repre}%
\end{equation}
where $\ell$ $\in$ $\pi_{1}(E_{\tau},z_{0})$ and $\ell^{\ast}Y$ denotes the
analytic continuation of $Y(z)$ along $\ell$. The image of $\rho_{\tau}(E)$ is
called the monodromy group of equation (\ref{GLE1}), which is unique up to a
conjugation. We define monodromy matrices $M_{j}(T_{k},E),$ $j=1,2,$ to be the
matrices satisfying%
\begin{equation}
\ell_{j}^{\ast}Y=\left(
\begin{array}
[c]{c}%
y_{1}(z+\omega_{j})\\
y_{2}(z+\omega_{j})
\end{array}
\right)  =M_{j}(T_{k},E)Y(z),\text{ }z\in E_{\tau}\label{Mo}%
\end{equation}
where $\ell_{j}$, $j$ $=$ $1,2,$ denote the two fundamental cycles of
$E_{\tau}$. The monodromy group is generated by these two monodromy matrices.
Since $\ell_{2}^{-1}\ell_{1}^{-1}\ell_{2}\ell_{1}$ $=$ $Id,$ $M_{j}(T_{k},E),$
$j=1,2,$ satisfy
\begin{equation}
M_{1}(T_{k},E)M_{2}(T_{k},E)=M_{2}(T_{k},E)M_{1}(T_{k},E),\label{abelian}%
\end{equation}
which implies the monodromy group of (\ref{GLE1}) is abelian when all
$m_{k}\in\mathbb{Z}_{\geq0}$. Therefore, we have two cases:

\noindent(i) \textbf{Completely reducible}:\textit{\ }By conjugating a common
matrix,\ the matrices\textit{\ }$M_{j}(T_{k},E),$ $j$ $=$ $1,$ $2,$ can be
simultaneously diagonalized as
\begin{equation}%
\begin{array}
[c]{c}%
M_{1}(T_{k},E)=\left(
\begin{array}
[c]{cc}%
e^{-2\pi is(T_{k},E)} & 0\\
0 & e^{2\pi is(T_{k},E)}%
\end{array}
\right)  ,\\
\\
M_{2}(T_{k},E)=\left(
\begin{matrix}
e^{2\pi ir(T_{k},E)} & 0\\
0 & e^{-2\pi ir(T_{k},E)}%
\end{matrix}
\right)  .
\end{array}
\label{m1}%
\end{equation}

\noindent(ii) \textbf{Not completely reducible}: $M_{j}(T_{k},E),$ $j=1,2,$
cannot be simultaneously diagonalized. In such a case, by a suitable choice of
$Y(z)$, one may normalize $M_{j}(T_{k},E),$ $j$ $=$ $1,2,$ by
\begin{equation}
M_{1}(T_{k},E)=\pm\left(
\begin{array}
[c]{cc}%
1 & 0\\
1 & 1
\end{array}
\right)  \text{, }M_{2}(T_{k},E)=\pm\left(
\begin{matrix}
1 & 0\\
\mathcal{D} & 1
\end{matrix}
\right)  ,\text{ }\label{2-3}%
\end{equation}
where $\mathcal{D}$ $\in$ $\mathbb{C\cup\{ \infty\}}$. Remark that if
$\mathcal{D}=\infty$, then (\ref{2-3}) should be understood as%
\begin{equation}
M_{1}(T_{k},E)=\pm%
\begin{pmatrix}
1 & 0\\
0 & 1
\end{pmatrix}
,\text{ \ \ \ }M_{2}(T_{k},E)=\pm%
\begin{pmatrix}
1 & 0\\
1 & 1
\end{pmatrix}
.\label{1-13}%
\end{equation}

\begin{definition}
\label{def1}Suppose equation (\ref{GLE1}) is apparent.

\noindent(i) The pair $\left(  r,\text{ }s\right)  $ $\in\mathbb{C}^{2}$ is
called the monodromy data if equation (\ref{GLE1}) is completely reducible.

\noindent(ii) The constant $\mathcal{D}\in\mathbb{C\cup\{ \infty\}}$ is called
the monodromy data if equation (\ref{GLE1}) is \textbf{not} completely reducible.
\end{definition}

A classical result (see \cite{Whittaker-Watson}) states that the ratio of any
two linearly independent solutions, $y_{1}$ and $y_{2}$, of equation
(\ref{GLE1}) must satisfy:
\[
\{ \frac{y_{1}}{y_{2}};z\}=-2q(z;T_{k},E).
\]
Since $\{ \frac{y_{1}}{y_{2}};z\}$ $=\{f;z\}=-2\hat{q}(z;T_{k},E)$, there is
$\gamma$ $=%
\begin{pmatrix}
a & b\\
c & d
\end{pmatrix}
$ $\in SL(2,\mathbb{C})$ such that $f=\gamma\cdot\frac{y_{1}}{y_{2}}$.
Consequently, $f$ can be expressed as a ratio of two linearly independent
solutions, $y_{i}$, $i=1,2,$ to equation (\ref{GLE1}), i.e., $f=$ $y_{1}%
/y_{2}.$ Later, we will see that the \textit{unitary monodromy} of developing
map $f$ leads to the \textit{unitary monodromy} of equation (\ref{GLE1}) as well.

\begin{theorem}
\label{thm19}The equation (\ref{MFE}) has a solution if and only if there
exist constants $T_{k}$ and $E\in\mathbb{C}$ such that equation (\ref{GLE1})
is both apparent and has unitary monodromy. Moreover, the developing map $f $
satisfies the type II constraint.
\end{theorem}

\begin{proof}
Suppose the equation (\ref{MFE}) has a solution $u$. By Lemma \ref{lem7},
there is an equation (\ref{GLE1}) that is apparent automatically. Also, there
are two linearly independent solutions $y_{1}$ and $y_{2}$ to the equation
(\ref{GLE1}) such that $f=y_{1}/y_{2}$. Let $Y(z)=(y_{1},y_{2})^{t}$ be a
fundamental system of solutions to the equation (\ref{GLE1}) and $M_{j}\in
SL(2,\mathbb{C})$, $j=1,2,$ be the monodromy matrices defined by (\ref{Mo}).
Then we see that
\[
f(z+\omega_{j})=\frac{y_{1}(z+\omega_{j})}{y_{2}(z+\omega_{j})}=M_{j}\cdot f.
\]
By (\ref{new26}), we conclude that $M_{j}=\pm U_{j}\in SU(2)$. Therefore, once
equation (\ref{GLE1}) is derived from $u$, then it must have unitary
monodromy. This completes the necessary part. Since $M_{1}M_{2}=M_{2}M_{1}$ by
(\ref{abelian}), we have $U_{1}U_{2}=U_{2}U_{1}$ which is equivalent to the
type II condition.

For the sufficient part, suppose equation (\ref{GLE1}) with potential
$q(z;T_{k},E)$ is apparent and of unitary monodromy. Let the monodromy
matrices $M_{j}$ be the monodromy matrices with respect to a fundamental
system of solutions to the equation (\ref{GLE1}). Since $M_{j}\in SU(2)$ and
satisfies (\ref{abelian}), after a common conjugation, the two monodromy
matrices $M_{j}$ can be normalized to be
\begin{equation}
M_{j}=\left(
\begin{array}
[c]{cc}%
e^{i\theta_{j}} & 0\\
0 & e^{-i\theta_{j}}%
\end{array}
\right)  ,\text{ }\theta_{j}\in\mathbb{R}\text{, }j=1,2\label{195}%
\end{equation}
with respect to a fundamental system of solutions $\hat{Y}=(\hat{y}_{1}%
,\hat{y}_{2})^{t}$. Define a meromorphic function $f$ by
\begin{equation}
f(z):=\frac{\hat{y}_{1}(z)}{\hat{y}_{2}(z)}.\label{197}%
\end{equation}
By (\ref{195}), $f(z)$ satisfies the type II constraint. Define $u$ by%
\[
u=\log\frac{8\left\vert f^{\prime}\right\vert ^{2}}{(1+\left\vert f\right\vert
^{2})^{2}}.
\]
Clearly, $u$ satisfies
\[
\Delta u+e^{u}=0\text{ for all }z\in E_{\tau}\backslash\{ \frac{\omega_{k}}%
{2}\} \text{.}%
\]
From (\ref{197}), we have
\begin{equation}
\{f;z\}=\frac{f^{\prime\prime\prime}}{f^{\prime}}-\frac{3}{2}\left(
\frac{f^{\prime\prime}}{f^{\prime}}\right)  ^{2}=-2q(z;T_{k},E).\label{198}%
\end{equation}
In the following, we claim from (\ref{198}) that the local behavior of $f$ at
any $z_{0}\in E_{\tau}$ is either (\ref{193}) or (\ref{194}) with $m$ given by
(\ref{mm}). Since the argument is the same for any $z_{0}\in E_{\tau}$, here
we present the proof for $z_{0}$ $=\frac{\omega_{k}}{2}$.

Near$\frac{\omega_{k}}{2}$, we have
\begin{equation}
\{f;z\}=-2q(z;T_{k},E)=\frac{-2m_{k}(m_{k}+1)}{(z-\frac{\omega_{k}}{2})^{2}%
}+\frac{-2T_{k}}{(z-\frac{\omega_{k}}{2})}+O(1).\label{1988}%
\end{equation}

Let $\xi=(z-\frac{\omega_{k}}{2})$. If $f$ is regular at $\frac{\omega_{k}}%
{2}$, then we may write
\[
f=a_{0}+a_{r}\xi^{r}+\cdot\cdot\cdot\text{ }%
\]
where
\[
a_{r}\not =0,\text{ and }r\geq1\text{.}%
\]
Then (\ref{1988}) implies $r=2m_{k}+1.$ Hence%
\[
f=a_{0}+a_{2m_{k}+1}\xi^{2m_{k}+1}+\cdot\cdot\cdot,a_{2m_{k}+1}\not =0.
\]
Now, suppose $f$ has a pole at $\frac{\omega_{k}}{2}$. Then we write
\[
f=a_{-r}\xi^{-r}+\cdot\cdot\cdot,\text{ where }a_{-r}\not =0,\text{ and }%
r\geq1.
\]
Then again (\ref{1988}) implies $r=2m_{k}+1$. Hence
\[
f=a_{-2m_{k}-1}\xi^{-2m_{k}-1}+\cdot\cdot\cdot,a_{-2m_{k}-1}\not =0.
\]
This completes the claim. By Theorem \ref{main thm 2-1}, $u$ is a solution to
equation (\ref{MFE1}). Moreover, by (\ref{ccc}), the associated equation
(\ref{GLE1}) with respect to $u$ is the given one.
\end{proof}

As we mentioned before, if $u(z)$ is an \textit{even} solution to equation
(\ref{MFE1}), then the related second-order linear ODE is reduced to the
elliptic form of Heun equation H($\mathbf{m},E,\tau$) as given in
(\ref{Heun}). One of the main results in \cite{Chen-Kuo-Lin-Lame II} is
Theorem A, which can be paraphrased as follows:\textit{\medskip}

\noindent\textbf{Theorem A.}\textit{\ ( \cite{Chen-Kuo-Lin-Lame II})}
\textit{The curvature equation (\ref{MFE1})}$_{\tau_{0}}$\textit{\ has an
\textbf{even} solution if and only if there exists }$E$ $\in$ $\mathbb{C}%
$\textit{\ such that the monodromy of H}$(\mathbf{m},E,\tau_{0})$\textit{\ is
unitary.\medskip}

\begin{proposition}
\label{prop5}Let $u$ be a solution to equation \textit{(\ref{MFE1}) (not
necessary to be even). If the associated ODE is a Heun equation }%
H$(\mathbf{m},E,\tau)$, i.e., $T_{k}=0$ for all $k$ in equation (\ref{GLE1}),
then $u$ is contained in an even family of solutions.
\end{proposition}

\begin{proof}
By Theorem \ref{thm19}, the equation H$(\mathbf{m},E,\tau)$ is of unitary
monodromy. By Theorem E., there is an \textit{even} solution $\tilde{u}$ to
the equation (\ref{MFE1}) such that its associated Fuchsian equation
is\textit{\ }the same equation H$(\mathbf{m},$ $E,$ $\tau)$. Since
H$(\mathbf{m},E,\tau)$ is of unitary monodromy, there is a fundamental system
of solutions $Y=$ $\left(  y_{1},y_{2}\right)  ^{t}$ such that the monodromy
matrices with respect to $Y$ are given by
\begin{equation}
M_{1}(E)=\left(
\begin{array}
[c]{cc}%
e^{-2\pi is} & 0\\
0 & e^{2\pi is}%
\end{array}
\right)  \text{ and }M_{2}(E)=\left(
\begin{array}
[c]{cc}%
e^{2\pi ir} & 0\\
0 & e^{-2\pi ir}%
\end{array}
\right)  ,\label{MMM}%
\end{equation}
where $\left(  r,s\right)  $ $\in$ $\mathbb{R}^{2}\backslash$ $\frac{1}{2}$
$\mathbb{Z}^{2}$. Let $f$ and $\tilde{f}$ be the normalized developing maps of
$u$ and $\tilde{u},$ respectively, and both satisfy the type II constraints.
Because both $f$ and $\tilde{f}$ satisfy type II constraint, we must have
\[
f(z+\omega_{j})=e^{2\pi i\theta_{j}}f(z)\text{ and }\tilde{f}(z+\omega
_{j})=e^{2\pi i\theta_{j}}\tilde{f}(z)\text{ }%
\]
where
\[
e^{2\pi i\theta_{j}}=\left\{
\begin{array}
[c]{c}%
e^{-4\pi is}\text{, if }j=1\\
e^{4\pi ir}\text{, if }j=2
\end{array}
\right.  \text{. }%
\]
Since $\left(  r,s\right)  \in$ $\mathbb{R}^{2}\backslash$ $\frac{1}{2}$
$\mathbb{Z}^{2}$, we see that $e^{2\pi i\theta_{j}}\not =1$ for $j=1,2$. Also,
there is $\gamma=\left(
\begin{array}
[c]{cc}%
a & b\\
c & d
\end{array}
\right)  \in SL(2,\mathbb{C})$ such that
\[
f=\gamma\cdot\tilde{f}=\frac{a\tilde{f}+b}{c\tilde{f}+d}.
\]
because $\{f;z\}=\{ \tilde{f};z\}=-2q(z)$. Then%
\[
f(z+\omega_{1})=\frac{a\tilde{f}(z+\omega_{1})+b}{c\tilde{f}(z+\omega_{1}%
)+d}=\frac{ae^{2\pi i\theta_{1}}\tilde{f}+b}{ce^{2\pi i\theta_{1}}\tilde{f}+d}%
\]
and on the other hand, we also have
\[
f(z+\omega_{1})=e^{2\pi i\theta_{1}}f(z)=\frac{e^{2\pi i\theta_{1}}a\tilde
{f}+e^{2\pi i\theta_{1}}b}{c\tilde{f}+d}.
\]
Hence,
\begin{align*}
&  ace^{2\pi i\theta_{1}}\tilde{f}^{2}+ade^{2\pi i\theta_{1}}\tilde
{f}+bc\tilde{f}+bd\\
&  =ace^{4\pi i\theta_{1}}\tilde{f}^{2}+ade^{2\pi i\theta_{1}}\tilde
{f}+bce^{4\pi i\theta_{1}}\tilde{f}+e^{2\pi i\theta_{1}}bd.
\end{align*}
From here, by $e^{2\pi i\theta_{1}}\not =1$, we must have $b=c=0$ and
$d=\frac{1}{a}$. Therefore, $f=a^{2}\tilde{f}$. Then%
\[
u=\log\frac{8\left\vert f^{\prime}\right\vert ^{2}}{\left(  1+\left\vert
f\right\vert ^{2}\right)  ^{2}}=\log\frac{8\left\vert a\right\vert
^{4}\left\vert \tilde{f}^{^{\prime}}\right\vert ^{2}}{\left(  1+\left\vert
a\right\vert ^{4}\left\vert \tilde{f}\right\vert ^{2}\right)  ^{2}}.
\]
Since $\left\vert a\right\vert ^{4}>0$, we conclude that $u=\tilde{u}_{\beta}$
where $e^{2\beta}=\left\vert a\right\vert ^{4}$. This completes the proof.
\end{proof}

\section{Spectral curve and Baker-Akhiezer function}

\label{BA function}

In this section, for the purpose of proving Theorem \ref{thm21 copy(1)}, we
consider the following second-order Fuchsian equation:%
\begin{equation}
y^{\prime\prime}(z)=\left(  2\left(  \wp(z)+\wp(z-\frac{\omega_{k}}%
{2})\right)  +T\left(  \zeta(z-\frac{\omega_{k}}{2})-\zeta(z)\right)
-E\right)  y(z),\label{199}%
\end{equation}
where $z\in E_{\tau}$, $T$ and $E$ $\in$ $\mathbb{C}$. According to Theorem
\ref{thm19}, the existence of solutions to equation (\ref{164}) is equivalent
to studying the monodromy of equation (\ref{199}). To do this, first we apply
standard Frobenius's method to derive the apparent condition for equation
(\ref{199}).

\begin{lemma}
\label{lem9}Equation (\ref{199}) is apparent at $z=0$ and $\frac{\omega_{k}%
}{2}$ if and only if the constants $T$ and $E$ satisfy:
\begin{equation}
T\left(  T^{2}-2\eta_{k}T+4E-4e_{k}\right)  =0.\label{app}%
\end{equation}

\end{lemma}

\begin{proof}
Notice that the local exponents of equation (\ref{199}) at $z=0,\frac
{\omega_{k}}{2}$ are $-1$ and $2$. Equation (\ref{199}) is apparent at $z=0$,
$\frac{\omega_{k}}{2}$ is equivalent to having a local solution with exponent
$-1$ at each singularity of the form:
\begin{equation}
y(z)=\sum_{m=0}^{\infty}c_{m}(z-\frac{\omega_{i}}{2})^{-1+m},c_{0}=1,\text{
}i=0,k\label{app form}%
\end{equation}
By inserting (\ref{app form}) into equation (\ref{1999}), the coefficient
$c_{m}$ can be determined recursively from $c_{0}=1$ if and only if $T$ and $E
$ satisfy the following condition:%
\[
T\left(  T^{2}-2\eta_{k}T+4E-4e_{k}\right)  =0
\]
This completes the proof.
\end{proof}

Let's denote the equation (\ref{199}) by L$(\mathbf{m}_{k},T,E,\tau)$, where
$\mathbf{m}_{k}$ is given in (\ref{m}). According to Lemma \ref{lem9}, we have
two situations for the pair $\left(  T,E\right)  $:\textit{\medskip}

\noindent Case (I): $T=0$ and $E\in\mathbb{C}$ is arbitrary. In this case,
equation L$(\mathbf{m}_{k},T,E,\tau)$ is reduced to H$(\mathbf{m}_{k},E,\tau
)$, which is apparent for any $E$ $\in$ $\mathbb{C}$. By Theorem A, if there
exists $E\in\mathbb{C}$ such that H$(\mathbf{m}_{k},E,\tau)$ has unitary
monodromy, then there exists an even solution $u_{even}$ to the equation
(\ref{164}). Moreover, as shown in Proposition \ref{prop5}, any solution $u$
associated with the same H$(\mathbf{m}_{k},E,\tau) $ must be part of an even
family of solutions generated by $u_{even}$.\textit{\medskip}

\noindent Case (II): $T$ is arbitrary in $\mathbb{C}$, and $E$ is uniquely
determined by:%
\begin{equation}
E=-\frac{T^{2}}{4}+\frac{\eta_{k}}{2}T+e_{k}.\label{200}%
\end{equation}
Suppose $E$ is determined by (\ref{200}) with $T\not =0$, and the
corresponding equation L$(\mathbf{m}_{k},T,E,\tau)$ has unitary monodromy. By
Theorem \ref{thm19}, there exists a solution $u$ to the equation (\ref{164})
which is associated with this equation L$(\mathbf{m}_{k},T,E,\tau)$. Since $T$
$\not =$ $0$, this solution $u$ is a non-even solution, and the family of
solutions generated by $u$ constitutes a non-even family of
solutions.\textit{\medskip}

The above discussion leads to the following theorem:

\begin{theorem}
\label{thm22}Equation (\ref{164}) has a non-even family of solutions if and
only if there exists $T\not =0$ such that equation L$(\mathbf{m}_{k}%
,T,E,\tau)$, with $E$ determined by (\ref{200}), exhibits unitary monodromy.
\end{theorem}

From now on, we always assume (\ref{200}) holds true. The notion of spectral
polynomial and Baker-Akhiezer function is established for the Heun equation
H$(\mathbf{m},E,\tau)$ for general parameter $\mathbf{m}$ $=$ $(m_{0}%
,m_{1},m_{2},m_{3})$, $m_{j}\in\mathbb{Z}_{>0}$ in KdV theory. See
\cite{GH-Book} for details. To study the equation L$(\mathbf{m}_{k},T,E,\tau)$
under the assumption (\ref{200}) from monodromy aspect, we borrow the notion
of spectral polynomial and Baker-Akhiezer function from KdV theory to equation
L$(\mathbf{m}_{k},T,E,\tau)$.

Generally, the potential, given by $2(\wp(z)$ $+\wp(z-\frac{\omega_{k}}{2}))$
$+T(\zeta(z-\frac{\omega_{k}}{2})$ $-\zeta(z))$ $-E$ in equation (\ref{199}),
is not even elliptic if $T$ $\not =$ $0$. However, by the translation
$z\rightarrow z+\frac{\omega_{k}}{4}$, we can observe that equation
(\ref{199}) is equivalent to
\begin{equation}
y^{\prime\prime}(z)=q(z;T,E)y(z),z\in E_{\tau}\label{1999}%
\end{equation}
where%
\begin{equation}
q(z;T,E)=\left(
\begin{array}
[c]{c}%
2\left(  \wp(z+\frac{\omega_{k}}{4})+\wp(z-\frac{\omega_{k}}{4})\right) \\
+T\left(  \zeta(z+\frac{\omega_{k}}{4})-\zeta(z-\frac{\omega_{k}}{4})\right)
-E
\end{array}
\right) \label{2}%
\end{equation}
and it is easy to see that after translation, equation (\ref{1999}) becomes an
even elliptic equation. Hereafter, we denote equation (\ref{1999}) with
potential (\ref{2}) as L($\mathbf{m}_{k},T,E,\tau$) if no confusion might
occur. Later, we will see that due to the evenness of $q(z;T,E)$, equation
(\ref{1999}) can be projected from $E_{\tau}$ to the plane $\mathbb{C}$ via
the map $x=\wp(z),$ and this is convenient for us to study the monodromy group.

Before proceeding, we will often use some well-known formulas in elliptic
function theory, and we present them as follows. The proofs can be found in
any textbook of elliptic functions, such as \cite{Akhiezer}.

\begin{lemma}
\label{lem1}%
\begin{equation}
\text{Legendre relation: }\tau\eta_{1}(\tau)-\eta_{2}(\tau)=2\pi
i,\label{F0000}%
\end{equation}%
\begin{equation}
\sigma(u+\omega_{i})=-e^{\eta_{i}(z+\frac{\omega_{i}}{2})}\sigma
(u),\label{F000}%
\end{equation}%
\begin{equation}
\wp^{\prime}(u)^{2}=4\wp^{3}(u)-g_{2}\wp(u)-g_{3}=4\prod\limits_{k=1}^{3}%
(\wp(u)-e_{k}),\label{F00}%
\end{equation}%
\begin{equation}
\wp^{\prime\prime}(u)=6\wp^{2}(u)-\frac{g_{2}}{2},\label{F0}%
\end{equation}%
\begin{equation}
\wp^{\prime\prime\prime}(u)=12\wp(u)\wp^{\prime}(u),\label{F}%
\end{equation}%
\begin{equation}
\zeta(u+v)=\zeta(u)+\zeta(v)+\frac{\wp^{\prime}(u)-\wp^{\prime}(v)}%
{2(\wp(u)-\wp(v))},\label{F1}%
\end{equation}%
\begin{equation}
\wp(u+v)=-\wp(u)-\wp(v)+\frac{\left(  \wp^{\prime}(u)-\wp^{\prime}(v)\right)
^{2}}{4\left(  \wp(u)-\wp(v)\right)  ^{2}},\label{F2}%
\end{equation}%
\begin{equation}
\zeta(2u)=2\zeta(u)+\frac{\wp^{\prime\prime}(u)}{2\wp^{\prime}(u)}\label{F6}%
\end{equation}%
\begin{equation}
\wp(2u)=\frac{-\left(  8\wp(u)\wp^{\prime2}(u)-\wp^{\prime\prime2}(u)\right)
}{4\wp^{\prime2}(u)},\label{F3}%
\end{equation}%
\begin{equation}
\wp^{\prime}(2u)=\frac{-\left(  4\wp^{\prime4}(u)-12\wp(u)\wp^{\prime2}%
(u)\wp^{\prime\prime}(u)+\wp^{\prime\prime3}(u)\right)  }{4\wp^{\prime3}%
(u)},\label{F4}%
\end{equation}%
\begin{equation}
\wp^{\prime\prime}(2u)=\frac{1}{8\wp^{\prime4}(u)}\left(
\begin{array}
[c]{c}%
144\wp^{2}(u)\wp^{\prime4}(u)+8\wp^{\prime4}(u)\wp^{\prime\prime}(u)\\
-48\wp(u)\wp^{\prime2}(u)\wp^{\prime\prime2}(u)+3\wp^{\prime\prime4}(u)
\end{array}
\right) \label{F5}%
\end{equation}

\end{lemma}

To simplify notations, we set
\begin{equation}
\zeta,\wp,\wp^{\prime},\wp^{\prime\prime}\text{ }:=\text{ }\zeta(\frac
{\omega_{k}}{4}),\wp(\frac{\omega_{k}}{4}),\wp^{\prime}(\frac{\omega_{k}}%
{4}),\wp^{\prime\prime}(\frac{\omega_{k}}{4})\text{ et cetera.}\tag{N}%
\label{Notation}%
\end{equation}
Applying Lemma \ref{lem1} with $u=\frac{\omega_{k}}{4}$, we obtain the
following lemma.

\begin{lemma}
\label{lem2}%
\begin{equation}
\frac{\eta_{k}}{2}=2\zeta+\frac{\wp^{\prime\prime}}{2\wp^{\prime}}\label{F7}%
\end{equation}%
\begin{equation}
e_{k}=\frac{-\left(  8\wp\wp^{\prime2}-\wp^{\prime\prime2}\right)  }%
{4\wp^{\prime2}}=-2\wp+\frac{\wp^{\prime\prime2}}{4\wp^{\prime2}}\label{F8}%
\end{equation}%
\begin{equation}
4\wp^{\prime4}=\left(  12\wp\wp^{\prime2}-\wp^{\prime\prime2}\right)
\wp^{\prime\prime}\label{F9}%
\end{equation}%
\begin{equation}
6e_{k}^{2}-\frac{g_{2}}{2}=\frac{2\wp^{\prime4}}{\wp^{\prime\prime2}%
}\label{F10}%
\end{equation}

\end{lemma}

\begin{proof}
In fact, formulas (\ref{F7})-(\ref{F9}) are direct consequences of
(\ref{F6})-(\ref{F4}) with $u=\frac{\omega_{k}}{4}$. Now, we prove
(\ref{F10}). Applying (\ref{F5}) and (\ref{F0}), we obtain%
\[
6e_{k}^{2}-\frac{g_{2}}{2}=\frac{1}{8\wp^{\prime4}}\left(  144\wp^{2}%
\wp^{\prime4}+8\wp^{\prime4}\wp^{\prime\prime}-48\wp\wp^{\prime2}\wp
^{\prime\prime2}+3\wp^{\prime\prime4}\right)  .
\]
A direct computation shows that%
\begin{align*}
6e_{k}^{2}-\frac{g_{2}}{2} &  =\frac{1}{8\wp^{\prime4}}\left(  \left(
12\wp\wp^{\prime2}-\wp^{\prime\prime2}\right)  ^{2}-24\wp\wp^{\prime2}%
\wp^{\prime\prime2}+2\wp^{\prime\prime4}+8\wp^{\prime4}\wp^{\prime\prime
}\right) \\
&  =\frac{1}{8\wp^{\prime4}}\left(  \left(  12\wp\wp^{\prime2}-\wp
^{\prime\prime2}\right)  ^{2}-2(12\wp\wp^{\prime2}-\wp^{\prime\prime2}%
)\wp^{\prime\prime2}+8\wp^{\prime4}\wp^{\prime\prime}\right)  .
\end{align*}
By (\ref{F9}), we have
\[
6e_{k}^{2}-\frac{g_{2}}{2}=\frac{1}{8\wp^{\prime4}}\left(  \left(  \frac
{4\wp^{\prime4}}{\wp^{\prime\prime}}\right)  ^{2}-2\left(  \frac{4\wp
^{\prime4}}{\wp^{\prime\prime}}\right)  \wp^{\prime\prime2}+8\wp^{\prime4}%
\wp^{\prime\prime}\right)  =\frac{2\wp^{\prime4}}{\wp^{\prime\prime2}}.
\]

\end{proof}

We consider the second symmetric product equation of the equation
L($\mathbf{m}_{k},T,E,\tau$) (\ref{1999}) which is a third order Fuchsian
equation defined by
\begin{equation}
\Phi^{\prime\prime\prime}(z)-4q(z;T,E)\Phi^{\prime}(z)-2q^{\prime}%
(z;T,E)\Phi(z)=0,\text{ }z\in E_{\tau}\label{3}%
\end{equation}
where $q(z;T,E)$ is given in (\ref{2}) and $T,E$ satisfy (\ref{200}).

\begin{theorem}
\label{thm2}Assume (\ref{200}). Then, up to a nonzero multiple, there is a
unique even elliptic solution of (\ref{3}) solution $\Phi_{e}(z;T,E)$ given
by
\begin{equation}
\Phi_{e}(z;T,E)=\frac{d_{2}(T,E)}{\left(  \wp(z)-\wp\right)  ^{2}}+\frac
{d_{1}(T,E)}{\left(  \wp(z)-\wp\right)  }+d_{0}(T,E)\label{EE}%
\end{equation}
where
\begin{equation}
d_{2}(T,E)=\wp^{\prime4},\label{EE1}%
\end{equation}%
\begin{equation}
d_{1}(T,E)=\wp^{\prime2}(\wp^{\prime}T+\wp^{\prime\prime}),\label{EE2}%
\end{equation}
and%
\begin{equation}
d_{0}(T,E)=\frac{1}{4}\left[
\begin{array}
[c]{c}%
4\wp^{\prime2}E+3\wp^{\prime2}T^{2}-8\zeta\cdot\wp^{\prime2}T\\
+32\wp\wp^{\prime2}-3\wp^{\prime\prime2}%
\end{array}
\right]  .\label{EE3}%
\end{equation}

\end{theorem}

Recall (\ref{Notation}). To prove Theorem \ref{thm2}, we introduce variable
$x:=\wp(z)-\wp$, the projection map from $z\in E_{\tau}$ to $x\in\mathbb{C}$.
Because, $q(z;T,E)$ is even, by addition formula (\ref{F2}), $q(z;T,E)$ can be
written in terms of variable $x$.

\begin{lemma}
\label{lem0}Recall $q(z;T,E)$ in (\ref{2}). Let $x=\wp(z)-\wp$. Then we have%
\[
q(z;T,E)=\sum_{\ell=-2}^{0}a_{\ell}(T,E)x^{\ell}%
\]
where%
\begin{equation}
a_{-2}=2\wp^{\prime2},\label{36}%
\end{equation}%
\begin{equation}
a_{-1}=2\wp^{\prime\prime}-\wp^{\prime}T,\label{37}%
\end{equation}%
\begin{equation}
a_{0}=-E+2\zeta T+4\wp.\label{38}%
\end{equation}

\end{lemma}

\begin{proof}
By addition formula (\ref{F2}) and (\ref{F00}), it is easy to derive%
\[
\wp(z+\frac{\omega_{k}}{4})+\wp(z-\frac{\omega_{k}}{4})=2\wp+\frac{\wp
^{\prime\prime}}{x}+\frac{\wp^{\prime2}}{x^{2}}%
\]
and%
\[
\zeta(z+\frac{\omega_{k}}{4})-\zeta(z-\frac{\omega_{k}}{4})=2\zeta-\frac
{\wp^{\prime}}{x}.
\]
Then
\begin{align*}
q(z;T,E)  &  =\left(
\begin{array}
[c]{c}%
2\left(  \wp(z+\frac{\omega_{k}}{4})+\wp(z-\frac{\omega_{k}}{4})\right) \\
+T\left(  \zeta(z+\frac{\omega_{k}}{4})-\zeta(z-\frac{\omega_{k}}{4})\right)
-E
\end{array}
\right) \\
&  =2\left(  2\wp+\frac{\wp^{\prime\prime}}{x}+\frac{\wp^{\prime2}}{x^{2}%
}\right)  +T\left(  2\zeta-\frac{\wp^{\prime}}{x}\right)  -E\\
&  =\frac{2\wp^{\prime2}}{x^{2}}+\frac{2\wp^{\prime\prime}-\wp^{\prime}T}%
{x}+\left(  -E+2\zeta T+4\wp\right)  .
\end{align*}
This completes the proof.
\end{proof}

Set $\Phi(x)$ $\doteqdot$ $\Phi(z)$. Using the variable $x$, the third order
equation (\ref{3}) is equivalent to the following equation.

\begin{lemma}
Equation (\ref{3}) is equivalent to%
\begin{equation}%
\begin{array}
[c]{l}%
\left(  4x^{3}+12\wp x^{2}+2\wp^{\prime\prime}x+\wp^{\prime2}\right)
\frac{d^{3}\Phi}{dx^{3}}+\left(  18x^{2}+36\wp x+3\wp^{\prime\prime}\right)
\frac{d^{2}\Phi}{dx^{2}}\\
\\
+\left(  12x+12\wp-4q(x,T,E)\right)  \frac{d\Phi}{dx}-2\frac{dq(x,T,E)}%
{dx}\Phi=0
\end{array}
\label{3-1}%
\end{equation}

\end{lemma}

\begin{proof}
Set $\Phi(x)$ $:=\Phi(z)$ and denote $^{\prime}=\frac{d}{dz}$. Since
\begin{equation}
x=\wp(z)-\wp,\label{32}%
\end{equation}
we have%
\[
\Phi^{\prime}(z)=\frac{d\Phi(x)}{dx}\wp^{\prime}(z)
\]%
\[
\Phi^{\prime\prime}(z)=\frac{d^{2}\Phi(x)}{dx^{2}}\wp^{\prime2}(z)+\frac
{d\Phi(x)}{dx}\wp^{\prime\prime}(z)
\]
and%
\[
\Phi^{\prime\prime\prime}(z)=\frac{d^{3}\Phi(x)}{dx^{3}}\wp^{\prime
3}(z)+3\frac{d^{2}\Phi(x)}{dx^{2}}\wp^{\prime\prime}(z)\wp^{\prime}%
(z)+\frac{d\Phi(x)}{dx}\wp^{\prime\prime\prime}(z).
\]
Equation (\ref{3-1}) can be derived easily by using these relations together
with (\ref{F00})-(\ref{F}) and (\ref{32}).
\end{proof}

\begin{proof}
[Proof of Theorem \ref{thm2}]Let
\begin{equation}
\Phi_{e}(x):=\frac{c_{-2}}{\left(  \wp(z)-\wp\right)  ^{2}}+\frac{c_{-1}%
}{\left(  \wp(z)-\wp\right)  }+c_{0}=\sum_{\ell=-2}^{0}c_{\ell}x^{\ell
}.\label{33}%
\end{equation}

Inserting (\ref{33}) into the LHS of (\ref{3-1}), by Lemma \ref{lem0}, we have%
\[
\text{LHS of (\ref{3-1})}=\sum_{i=-2}^{3}\gamma_{i}\cdot x^{i-3},
\]
where the coefficient $\gamma_{i}$ is given by%
\begin{align}
\gamma_{i} &  =2(i-3)(i-2)(2i-5)c_{i-3}\label{34}\\
&  +4(i-2)\left[  \left(  3i^{2}-12i+8\right)  \wp+E-2\zeta T\right]
c_{i-2}\nonumber\\
&  +\left(  2i-3\right)  \left[  2\wp^{\prime}T+(i^{2}-3i-2)\wp^{\prime\prime
}\right]  c_{i-1}\nonumber\\
&  +\left(  i+2\right)  \left(  i-1\right)  \left(  i-4\right)  \wp^{\prime
2}c_{i}.\nonumber
\end{align}
Here, we set%
\begin{equation}
c_{\ell}=0\text{ for }\ell<-2\text{ and }\ell>0.\label{35}%
\end{equation}
Since the LHS of (\ref{3-1}) with $\Phi(x)=\Phi_{e}(x)$ is even elliptic,
$\Phi_{e}(x)$ is a solution of equation (\ref{3-1}) iff $\gamma_{i}=0$ for
$i=-2,\cdot\cdot\cdot,3$.

Clearly, by (\ref{34}) and (\ref{35}), $\gamma_{-2}=\gamma_{2}=\gamma_{3}=0$
holds automatically. Next, we investigate the conditions for $\gamma
_{-1}=\gamma_{0}=\gamma_{1}=0$:

\noindent(i) $\gamma_{-1}=0$ iff%
\[
c_{-1}=\left(  \frac{T}{\wp^{\prime}}+\frac{\wp^{\prime\prime}}{\wp^{\prime2}%
}\right)  c_{-2},
\]

\noindent(ii) $\gamma_{0}=0$ iff%
\begin{align*}
\text{ }c_{0}  &  =\frac{1}{8\wp^{\prime2}}\left[  6\left(  \wp^{\prime}%
T-\wp^{\prime\prime}\right)  c_{-1}-8\left(  2\zeta T-E-8\wp\right)
c_{-2}\right] \\
&  =\left[  \frac{3}{4}\left(  \frac{T^{2}}{\wp^{\prime2}}-\frac{\wp
^{\prime\prime2}}{\wp^{\prime4}}\right)  -\frac{2\zeta T}{\wp^{\prime2}}%
+\frac{E}{\wp^{\prime2}}+\frac{8\wp}{\wp^{\prime2}}\right]  c_{-2},
\end{align*}

\noindent(iii) $\gamma_{1}=0$ iff%
\begin{align*}
0 &  =\left(  \wp^{\prime}T-2\wp^{\prime\prime}\right)  c_{0}-2\left(  2\zeta
T-E+\wp\right)  c_{-1}+6c_{-2}\\
&  =\frac{3}{4\wp^{\prime}}\left[
\begin{array}
[c]{c}%
T^{3}-4\left(  2\zeta+\frac{\wp^{\prime\prime}}{2\wp^{\prime}}\right)
T^{2}+\left(  4E+8\wp-\frac{\wp^{\prime\prime2}}{\wp^{\prime2}}\right)  T\\
+2\left(  4\wp^{\prime}-12\frac{\wp\wp^{\prime\prime}}{\wp^{\prime}}+\frac
{\wp^{\prime\prime3}}{\wp^{\prime3}}\right)
\end{array}
\right]  c_{-2}.
\end{align*}
By (i) and (ii), $c_{-1}$ and $c_{0}$ can be determined uniquely by $c_{-2}$
and $\gamma_{1}=0$ holds automatically by (\ref{F6})-(\ref{F4}) and
(\ref{200}). Then Theorem \ref{thm2} follows by letting $c_{-2}=\wp^{\prime4}$.
\end{proof}

By (\ref{200}) and (\ref{F7}), we have
\begin{align}
d_{0}(T,E) &  =d_{0}(T)\label{EE4}\\
&  =\frac{1}{4}\left(  2\wp^{\prime2}T^{2}+2\wp^{\prime}\wp^{\prime\prime
}T+4(8\wp+e_{k})\wp^{\prime2}-3\wp^{\prime\prime2}\right)  .\nonumber
\end{align}
Due to (\ref{200}), in the following, we often use the notation $F(T)$ to
denote $F(T,E),$ with $E$ replaced by (\ref{200}). For example, $\Phi
_{e}(z;T)$ denotes $\Phi_{e}(z;T,E)$ with $E$ replaced by (\ref{200}). Since
$\Phi_{e}(z;T)$ solves equation (\ref{3}), it is easy to see that%
\begin{equation}
Q(T)\doteqdot\frac{1}{2}\Phi_{e}\Phi_{e}^{\prime\prime}-\frac{1}{4}(\Phi
_{e}^{\prime})^{2}-q(z;T)(\Phi_{e})^{2}\label{Q}%
\end{equation}
is a constant depending on $T$ and is independent of $z\in$ $E_{\tau}$.
Theorem \ref{thm2} and a direct computation shows that $Q(T)$ is a polynomial
in the variable $T$ given by%
\begin{equation}
Q(T)=d_{0}(T)d_{1}(T)-a_{0}(T)d_{0}(T)^{2}\label{Q1}%
\end{equation}
where $a_{0}(T)$ is given in (\ref{38}). Indeed, by using (\ref{EE1}),
(\ref{EE2}) ,(\ref{EE4}), and identities in Lemma \ref{lem2} efficiently,
$Q(T)$ can be expressed explicitly as follows:%
\begin{equation}
Q(T)=-\frac{\wp^{\prime4}}{16}\left(  T^{2}-12e_{k}\right)  (T^{2}%
-4e_{k}+4e_{i})\left(  T^{4}-4e_{k}+4e_{i^{\prime}}\right) \label{Q2}%
\end{equation}
where $\{k,i,i^{\prime}\}=\{1,2,3\}$. Analogous to KdV theory, this polynomial
is called the \textit{spectral polynomial} here.

Next, we introduce the\textit{\ spectral curve} $\Gamma(\tau)$ associated to
the spectral polynomial $Q(T)$ as follows:%
\[
\Gamma(\tau)\doteqdot\left\{  \left(  T,\mathcal{C}\right)  |\mathcal{C}%
^{2}=Q(T)\right\}  .
\]
Now, we extend the notion of the Baker-Akhiezer function on $\Gamma(\tau)$.
Given $P$ $=$ $\left(  T,\mathcal{C}\right)  $ $\in\Gamma(\tau)$, if
$\mathcal{C}\not =0$, then we denote $P_{\ast}$ $=$ $\left(  T,-\mathcal{C}%
\right)  $ $\in\Gamma(\tau)$ as the dual point of $P$. For any $P$ $=$
$\left(  T,\mathcal{C}\right)  $ $\in$ $\Gamma(\tau)$, we follow
\cite{GH-Book} to define the meromorphic functions $\phi(P;z)$ as:
\begin{equation}
\phi(P;z)\doteqdot\frac{i\mathcal{C}(P)+\frac{1}{2}\Phi_{e}^{\prime}%
(z;T)}{\Phi_{e}(z;T)},\quad z\in\mathbb{C},\label{eq-BA-0}%
\end{equation}
where $i$ $=$ $\sqrt{-1}$, and $\mathcal{C}(P)$ denotes the second coordinate
of $P$ $=$ $(T,\mathcal{C})$. A direct computation shows that $\phi(P;z)$
satisfies the following equation:%
\begin{equation}
\phi^{\prime}(P;z)=q(z,T,E)-\phi(P;z)^{2}.\label{4}%
\end{equation}

\begin{proposition}
\label{prop2-1}The meromorphic function $\phi(P;z)$ has simple poles only, and
the residue is either $-1$ or $1$.
\end{proposition}

\begin{proof}
If $z_{0}$ is a pole of $\phi(P;z)$, then either $z_{0}$ is a pole of
$\Phi_{e}(z;T)$ or is a zero of $\Phi_{e}(z;T)$. When the former case occurs,
the order of pole of $\Phi_{e}(z;T)$ at $z_{0}$ is equal to $1,$ and then
\[
Res_{z=z_{0}}\phi(P;z)=-1.
\]
Suppose $\Phi_{e}(z_{0};T)=0$. Notice that%
\begin{equation}
\left(  i\mathcal{C+}\tfrac{1}{2}\Phi_{e}^{\prime}\right)  \left(
-i\mathcal{C+}\tfrac{1}{2}\Phi_{e}^{\prime}\right)  =\Phi_{e}\left(  \tfrac
{1}{2}\Phi_{e}^{\prime\prime}-\Phi_{e}q\right)  .\label{4-1}%
\end{equation}
Then we have two cases:\medskip

Case 1: $i\mathcal{C}(P)+\frac{1}{2}\Phi_{e}^{\prime}(z_{0};T)=0$: Since
$z_{0}$ is a pole of $\phi(P;z)$, by (\ref{4-1}) we have
\[
\pm i\mathcal{C}(P)+\frac{1}{2}\Phi_{e}^{\prime}(z_{0};T)=0.
\]
This implies $\mathcal{C=}\Phi_{e}^{\prime}(z_{0};T)=0$. Therefore, $\Phi
_{e}(z_{0};T)=d\cdot(z-z_{0})^{2}+O((z-z_{0})^{3})$ for some $d$ $\not =$ $0$
because $\Phi_{e}(z_{0};T)$ is a solution of (\ref{3}). Thus,
\[
Res_{z=z_{0}}\phi(P;z)=\frac{1}{2}Res_{z=z_{0}}\frac{\Phi_{e}^{\prime}%
(z;T)}{\Phi_{e}(z;T)}=1.
\]
Case 2: $i\mathcal{C}(P)+\frac{1}{2}\Phi_{q_{\mathbf{p}}}^{\prime}%
(z_{0};T)\not =0$: Then (\ref{4-1}) implies $-i\mathcal{C}(P)+\frac{1}{2}%
\Phi_{q_{\mathbf{p}}}^{\prime}(z_{0};T)=0$. Therefore,%
\[
Res_{z=z_{0}}\phi(P;z)=\frac{i\mathcal{C}(P)+\frac{1}{2}\Phi_{e}^{\prime
}(z_{0};T)}{\Phi_{e}^{\prime}(z_{0};T)}=1.
\]

\end{proof}

Fixing any $z_{0}$ $\in$ $\mathbb{C}\backslash\{ \text{poles of }q(z)\}$, the
stationary Baker-Akhiezer function $\psi(P;z,z_{0})$ on $\Gamma(\tau)$ is
defined by%
\begin{equation}
\psi(P;z,z_{0}):=\exp\left(  \int_{z_{0}}^{z}\phi(P;\xi)d\xi\right)
,\;z\in\mathbb{C},\label{4-3}%
\end{equation}
where the integration path in (\ref{4-3}) should avoid any singularity of
$\phi(P;z)$. From Proposition \ref{prop2-1}, we see that $\psi(P;z,z_{0})$ is
single-valued in $z\in\mathbb{C}$.

By (\ref{4}) and $\phi(P;z)=\frac{\psi^{\prime}(P;z,z_{0})}{\psi(P;z,z_{0})}$,
it is easy to see that both $\psi(P;z,z_{0})$ and $\psi(P_{\ast};z,z_{0})$
solve the elliptic second order Fuchsian equation (\ref{1999}). Moreover,%
\begin{equation}
\psi(P;z,z_{0})\psi(P_{\ast};z,z_{0})=\exp\left(  \int_{z_{0}}^{z}\frac
{\Phi_{e}^{\prime}(\xi;T)}{\Phi_{e}(\xi;T)}d\xi\right)  =\frac{\Phi_{e}%
(z;T)}{\Phi_{e}(z_{0};T)}\label{5}%
\end{equation}
which implies
\begin{equation}
W(\psi(P;z,z_{0}),\psi(P_{\ast};z,z_{0}))=\frac{2i\mathcal{C}(P)}{\Phi
_{e}(z_{0};T)}\label{6}%
\end{equation}
where $^{\prime}=\frac{d}{dz}$ and $W(f,g)=f^{\prime}g-fg^{\prime}$ denotes
the Wronskian of $f,g$. Since different choices of $z_{0}$ give the same
solution to the linear ODE (\ref{1999}) up to multiplying a constant, we omit
the notation $z_{0}$ and write
\[
\psi(P;z,z_{0})=\psi(P;z),\quad\psi(P^{\ast};z,z_{0})=\psi(P_{\ast};z)
\]
for convenience.

For any $P$ $=$ $\left(  T,\mathcal{C}\right)  $ $\in$ $\Gamma(\tau)$, we can
always associate an apparent second-order Fuchsian equation (\ref{1999}) where
$T$ is the $T$-coordinate of $P$ and $E$ determined by (\ref{200}). Moreover,
$\psi(P;z)$ and $\psi(P_{\ast};z)$ are two solutions to this equation.
Therefore,\textit{\ }%
\begin{equation}
\text{any }P\in\Gamma(\tau)\text{ represents an apparent equation
(\ref{1999}).}\label{018}%
\end{equation}

By (\ref{6}), we immediately have the following result.

\begin{theorem}
\label{thm4}Let $P=\left(  T,\mathcal{C}\right)  \in\Gamma(\tau)$ and assume
(\ref{200}). Then the two Baker-Akhiezer functions $\psi(P;z)$ and
$\psi(P_{\ast};z)$ for the equation (\ref{1999}) are linearly independent if
and only if $\mathcal{C}$ $\mathcal{\not =}0,$ i.e., $Q(T)$ $\not =0$.
\end{theorem}

\section{Monodromy theory for equation (\ref{1999})}

\label{Monodromy theory}

Let $\left(  r,s\right)  $ $\in\mathbb{C}^{2}\backslash\frac{1}{2}%
\mathbb{Z}^{2}$ and define $\left(  r_{k},s_{k}\right)  $ for $k=1,2,3,$ by%
\begin{equation}
\left(  r_{k},s_{k}\right)  =\left\{
\begin{array}
[c]{l}%
\left(  r+\frac{1}{2},s\right)  \text{ if }k=1,\\
\\
\left(  r,s+\frac{1}{2}\right)  \text{ if }k=2,\\
\\
\left(  r+\frac{1}{2},s+\frac{1}{2}\right)  \text{ if }k=3.
\end{array}
\right. \label{rsk}%
\end{equation}
The main theorem in this section is the following:

\begin{theorem}
\label{main thm 1}Let $\tau\in\mathbb{H}$. Suppose $\left(  r,s\right)  $
$\in\mathbb{C}^{2}\backslash\frac{1}{2}\mathbb{Z}^{2}$ such that
$Z(r_{k},s_{k},\tau)$ $=$ $0$ where $\left(  r_{k},s_{k}\right)  $ is defined
in (\ref{rsk}) and $Z(r,s,\tau)$ is the Hecke form defined in (\ref{zrs}).
Then there is $P$ $=$ $\left(  T,\mathcal{C}\right)  $ $\in$ $\Gamma(\tau)$
such that the corresponding L$(\mathbf{m}_{k},T,E,\tau)$ (\ref{1999}) is
completely reducible with monodromy data $\left(  r,s\right)  $.
\end{theorem}

Let
\begin{align*}
\lambda_{j}(P,z) &  :=\exp\int_{z}^{z+\omega_{j}}\phi(P;\xi)d\xi\\
&  =\exp\int_{z}^{z+\omega_{j}}\frac{i\mathcal{C}(P)+\frac{1}{2}\Phi
_{e}^{\prime}(\xi;T)}{\Phi_{e}(\xi;T)}d\xi,\text{ }j=1,2
\end{align*}
where the integral path is the fundamental cycle from $z$ to $z+\omega_{j}$
which stays always from poles and zeros of $\Phi_{e}(z;T)$ and does not
intersect with the straight segment joining $\pm\frac{\omega_{k}}{4}$. Since
$\Phi_{e}(z;T)$ is elliptic, $\lambda_{j}(P,z)$ $\equiv$ $\lambda_{j}(P)$ is
independent of $z$ for $j$ $=$ $1,$ $2,$ which implies that
\begin{equation}
\psi(P;z+\omega_{j})=\lambda_{j}(P)\psi(P;z),\quad j=1,2.\label{ell-sec1}%
\end{equation}
Therefore, $\psi(P;z)$ as well as $\psi(P_{\ast};z)$ are both functions of
elliptic of the second kind. Moreover, by (\ref{5}), it is seen that
\begin{equation}
\lambda_{j}(P_{\ast})=\lambda_{j}^{-1}(P)\text{, }j=1,2.\label{7}%
\end{equation}
We also define $\left(  r(P),s(P)\right)  \in(\mathbb{C}/\mathbb{Z})^{2}$ for
each $P\in\Gamma(\tau)$ by
\begin{equation}
\lambda_{1}(P)=e^{-2\pi is(P)}\text{ and }\lambda_{2}(P)=e^{2\pi
ir(P)}.\label{rs}%
\end{equation}
By (\ref{7}), it is easy to see that
\begin{equation}
\left(  r(P_{\ast}),s(P_{\ast})\right)  =\left(  -r(P),-s(P)\right)
.\label{rsdual}%
\end{equation}

The pair $\left(  r(P),s(P)\right)  $ defined in (\ref{rs}) via the
Baker-Akhiezer function $\psi(P;z)$ plays an important role for us to study
the monodromy of the associated Fuchsian equation (\ref{1999}). The following
two theorems show that once $\left(  r(P),s(P)\right)  $ $\not \in \frac{1}%
{2}\mathbb{Z}^{2}$, then the associated equation (\ref{1999}) must be
completely reducible with monodromy data $\left(  r(P),s(P)\right)  $.

\begin{theorem}
\label{thm5}Let $P$ $=$ $\left(  T,\mathcal{C}\right)  $ $\in$ $\Gamma(\tau) $
and assume (\ref{200}). Equation (\ref{1999}) is completely reducible if and
only if $Q(T)\not =0$. In other words, the two Baker-Akhiezer functions
$\psi(P;z)$ and $\psi(P_{\ast};z)$ are linearly independent.
\end{theorem}

\begin{proof}
Suppose equation (\ref{1999}) is completely reducible, then there are two
linearly independent solutions $y_{1}(z)$ and $y_{2}(z)$ such that the
monodromy matrices $M_{i}$ with respect to $Y(z)=\left(  y_{1},y_{2}\right)
^{t}$ are given by $M_{i}=$ $\left(
\begin{array}
[c]{cc}%
\lambda_{i} & 0\\
0 & \lambda_{i}^{-1}%
\end{array}
\right)  $. Namely,
\begin{equation}
y_{1}(z+\omega_{i})=\lambda_{i}y_{1}(z)\text{ and }y_{2}(z+\omega_{i}%
)=\lambda_{i}^{-1}y_{2}(z).\label{11}%
\end{equation}
Since equation (\ref{1999}) is even elliptic, $\hat{y}_{j}(z):=y_{j}(-z)$,
$j=1,2,$ are also two linearly independent solutions. By (\ref{11}), we have
\begin{equation}
\hat{y}_{1}(z+\omega_{i})=\lambda_{i}^{-1}\hat{y}_{1}(z)\text{ and }\hat
{y}_{2}(z+\omega_{i})=\lambda_{i}\hat{y}_{2}(z).\label{12}%
\end{equation}
(\ref{11}) and (\ref{12}) imply that $y_{j}(z)\hat{y}_{j}(z)=y_{j}%
(z)y_{j}(-z)$ $j=1,2,$ are both even elliptic solutions to the equation
(\ref{3}) and by the uniqueness of $\Phi_{e}$, up to a multiple, we must have
\begin{equation}
y_{1}(z)y_{1}(-z)=y_{2}(z)y_{2}(-z)=\Phi_{e}(z;T)=\psi(P;z)\psi(P_{\ast
};z).\label{1}%
\end{equation}
Since $y_{1}(z)$ and $y_{2}(z)$ are linearly independent, by (\ref{1}), we
have $y_{1}(z)=y_{2}(-z)$ (or $y_{1}(-z)=y_{2}(z)$). Therefore,%
\[
y_{1}(z)y_{1}(-z)=y_{1}(z)y_{2}(z)=\Phi_{e}(z;T)=\psi(P;z)\psi(P_{\ast};z)
\]
which implies $\psi(P;z)$ and $\psi(P_{\ast};z)$ are linearly independent.

Conversely, suppose $Q(T)$ $\not =0$, then $\psi(P;z)$ and $\psi(P_{\ast};z)$
are linearly independent. Take $Y(z)$ $=$ $(\psi(P;z),$ $\psi(P_{\ast}%
;z))^{t}$ as a fundamental system of solutions. The monodromy matrices with
respect to $Y(z)$ are given by $M_{j}$ $=diag(\lambda_{j}(P),\lambda_{j}%
^{-1}(P))$.
\end{proof}

\begin{theorem}
\label{thm6}Let $P$ $=$ $\left(  T,\mathcal{C}\right)  $ $\in\Gamma(\tau)$ and
assume (\ref{200}). Then $Q(T)\not =0$ if and only if $\left(
r(P),s(P)\right)  $ $\not \in $ $\frac{1}{2}\mathbb{Z}^{2}$.
\end{theorem}

\begin{proof}
First, we prove the necessary part by contradiction argument. Suppose
$Q(T)\not =0$ and $\left(  r(P),s(P)\right)  $ $\in$ $\frac{1}{2}%
\mathbb{Z}^{2}$. By (\ref{rs}), we see that
\begin{equation}
\left(  \lambda_{1}(P),\lambda_{2}(P)\right)  \in\{ \left(  1,1\right)
,(1,-1),(-1,1),(-1,-1)\}.\label{8}%
\end{equation}
Since $Q(T)\not =0$, by Theorem \ref{thm5}, equation (\ref{1999}) is
completely reducible, and by (\ref{8}) the monodromy matrices with respect to
$\psi(P;z)$ and $\psi(P_{\ast};z)$ are given by
\begin{equation}
M_{i}=\pm\left(
\begin{array}
[c]{cc}%
1 & 0\\
0 & 1
\end{array}
\right)  \text{ or }\pm\left(
\begin{array}
[c]{cc}%
1 & 0\\
0 & -1
\end{array}
\right)  .\label{9}%
\end{equation}

By the definition (\ref{4-3}), it is easy to see that
\begin{equation}
\psi(P;-z)=d\cdot\psi(P_{\ast};z)\label{10}%
\end{equation}
for some $d\not =0$ and is linearly independent of $\psi(P;z)$. By (\ref{9}),
we see that $\psi(P;z)^{2}+\psi(P;-z)^{2}$ is also an even elliptic solution
to the equation (\ref{3}), and by the uniqueness, we have
\begin{align*}
\psi(P;z)^{2}+\psi(P;-z)^{2} &  =\Phi_{e}(z;T)=\psi(P;z)\psi(P_{\ast};z)\\
&  =\psi(P;z)\psi(P;-z).
\end{align*}
From here, we see that any zero of $\psi(P;z)$ is also a zero of $\psi
(P_{\ast};z)$ which leads to a contradiction to the fact that $\psi(P;z)$ and
$\psi(P_{\ast};z)$ are linearly independent.

For the sufficient part, suppose $\left(  r(P),s(P)\right)  \not \in \frac
{1}{2}\mathbb{Z}^{2}$, then $\lambda_{1}(P),\lambda_{2}(P)\not \in \{ \pm1\}$.
This implies $\lambda_{i}(P)\not =\lambda_{i}^{-1}(P)$ and hence $\psi(P;z)$
and $\psi(P_{\ast};z)$ are linearly independent. So we have $Q(T)\not =0$.
\end{proof}

Next, we want to recover $\left(  r(P),s(P)\right)  $ from $P$ $=$
$(T,\mathcal{C})$ $\in\Gamma(\tau)$. According to (\ref{ell-sec1}), we know
that the Baker-Akhiezer function $\psi(P;z)$ is an elliptic function of second
kind. Consequently, we can express $\psi(P;z)$ as follows:%
\begin{equation}
\psi(P;z)=\frac{e^{c(P)z}\sigma(z-a_{1}(P))\sigma(z-a_{2}(P))}{\sigma
(z-\frac{\omega_{k}}{4})\sigma(z+\frac{\omega_{k}}{4})}.\label{HA}%
\end{equation}
Here, $c(P)\in\mathbb{C}$ is a constant, and $a_{1}(P)\not =$ $a_{2}(P)$ $\in
E_{\tau}\backslash\{ \pm\frac{\omega_{k}}{4}\}$ denotes the two zeros of
$\psi(P;z)$. For the dual point $P_{\ast}$ $=$ $(T,-\mathcal{C})$ $\in
\Gamma(\tau)$, we have a similar expression:%
\[
\psi(P_{\ast};z)=\frac{e^{c(P_{\ast})z}\sigma(z-a_{1}(P_{\ast}))\sigma
(z-a_{2}(P_{\ast}))}{\sigma(z-\frac{\omega_{k}}{4})\sigma(z+\frac{\omega_{k}%
}{4})}.
\]
Since
\[
\psi(P;z)\psi(P_{\ast};z)=\frac{\Phi_{e}(z;E)}{\Phi_{e}(z_{0};E)}%
\]
is an even elliptic function, we conclude that%
\begin{equation}
\left\{  a_{1}(P_{\ast}),a_{2}(P_{\ast})\right\}  =\left\{  -a_{1}%
(P),-a_{2}(P)\right\} \label{HB}%
\end{equation}
and%
\[
c(P_{\ast})=-c(P).
\]
As a corollary of Theorems \ref{thm5} and \ref{thm6}, we obtain the following corollary:

\begin{corollary}
\label{cor1}Let $P$ $=$ $\left(  T,\mathcal{C}\right)  $ $\in$ $\Gamma(\tau) $
and assume (\ref{200}). The following statements are equivalent:

\noindent(i) Equation (\ref{1999}) is completely reducible with monodromy
data
\[
(r(P),s(P))\in\mathbb{C}^{2}\backslash\frac{1}{2}\mathbb{Z}^{2}.
\]

\noindent(ii) $Q(T)$ $\not =$ $0.$

\noindent(iii) $\left\{  a_{1}(P),a_{2}(P)\right\}  $ $\cap\left\{
-a_{1}(P),-a_{2}(P)\right\}  $ $=\emptyset.$
\end{corollary}

Next, we derive the algebraic equations for zeros $a_{1}(P)$ and $a_{2}(P)$.

\begin{proposition}
\label{prop1}Define
\begin{equation}
y(z)=\frac{e^{cz}\sigma(z-a_{1})\sigma(z-a_{2})}{\sigma(z-\frac{\omega_{k}}%
{4})\sigma(z+\frac{\omega_{k}}{4})}\label{67}%
\end{equation}
where $c\in\mathbb{C},$ $a_{1}$ $\not =$ $a_{2}$ $\in$ $E_{\tau}\backslash\{
\pm\frac{\omega_{k}}{4}\}.$ Then $a_{1}$ and $a_{2}$ satisfy the following
algebraic equation:%
\begin{align}
&  \frac{1}{2}\left(  \zeta(a_{2}+\frac{\omega_{k}}{4})+\zeta(a_{2}%
-\frac{\omega_{k}}{4})\right) \label{68}\\
&  =\frac{1}{2}\left(  \zeta(a_{1}+\frac{\omega_{k}}{4})+\zeta(a_{1}%
-\frac{\omega_{k}}{4})\right)  -\zeta(a_{1}-a_{2}).\nonumber
\end{align}
if and only if $y(z)$ is a solution to equation (\ref{1999}) for some $T,E$.
Additionally, $c,$ $T$ and $E$ can be determined by the following equations:%
\begin{equation}
-T=\eta_{k}-\sum_{i=1}^{2}\left(  \zeta(a_{i}+\frac{\omega_{k}}{4}%
)-\zeta(a_{i}-\frac{\omega_{k}}{4})\right)  ,\label{69}%
\end{equation}%
\begin{equation}
-E=\frac{T^{2}}{4}-\frac{\eta_{k}}{2}T-3e_{k}+\sum_{i=1}^{2}\wp(a_{i}%
-\frac{\omega_{k}}{4}),\label{70}%
\end{equation}
and%
\begin{equation}
c=\frac{T}{2}+\frac{\eta_{k}}{2}+\sum_{i=1}^{2}\zeta(a_{i}-\frac{\omega_{k}%
}{4}).\label{71}%
\end{equation}

\end{proposition}

\begin{proof}
By using the expression (\ref{HA}), we have%
\[
\frac{y^{\prime}(z)}{y(z)}=c+\zeta(z-a_{1})+\zeta(z-a_{2})-\zeta
(z+\frac{\omega_{k}}{4})-\zeta(z-\frac{\omega_{k}}{4}),
\]
and%
\[
\left(  \frac{y^{\prime}(z)}{y(z)}\right)  ^{\prime}=-\wp(z-a_{1})-\wp
(z-a_{2})+\wp(z+\frac{\omega_{k}}{4})+\wp(z-\frac{\omega_{k}}{4}).
\]
Therefore, $y(z)$ is a solution to (\ref{1}) if and only if%
\begin{align*}
&  \frac{y^{\prime\prime}(P;z)}{y(P;z)}-q(z;T,E)\\
&  =\left(  \frac{y^{\prime}(P;z)}{y(P;z)}\right)  ^{\prime}+\left(
\frac{y^{\prime}(P;z)}{y(P;z)}\right)  ^{2}-q(z;T,E)\\
&  =-\wp(z-a_{1})-\wp(z-a_{2})+\wp(z+\frac{\omega_{k}}{4})+\wp(z-\frac
{\omega_{k}}{4})\\
&  +\left[  c+\zeta(z-a_{1})+\zeta(z-a_{2})-\zeta(z+\frac{\omega_{k}}%
{4})-\zeta(z-\frac{\omega_{k}}{4})\right]  ^{2}\\
&  -\left[  2\left(  \wp(z+\frac{\omega_{k}}{4})+\wp(z-\frac{\omega_{k}}%
{4})\right)  +T\left(  \zeta(z+\frac{\omega_{k}}{4})-\zeta(z-\frac{\omega_{k}%
}{4})\right)  -E\right]  .
\end{align*}
The proof of Proposition \ref{prop1} is a straightforward computation of local
expansions of $\frac{y^{\prime\prime}(P;z)}{y(P;z)}-q(z;T,E)$ at $z=a_{1},$
$a_{2},$ $\pm\frac{\omega_{k}}{4}$ respectively. See \cite{Chen-Kuo-Lin-Lame
II} for the same computations.
\end{proof}

Suppose $a_{1}$ and $a_{2}$ satisfy (\ref{68}). Since $y(z)$ defined by
(\ref{67}) has exponent $-1$ at both $z=\pm\frac{\omega_{k}}{4}$, the equation
(\ref{1999}) with $T$ and $E$ given by (\ref{69}) and (\ref{70}) must be
apparent automatically. In other words, $T$ and $E$ satisfy (\ref{app}).

The following proposition states that the pair $\left(  r(P),s(P)\right)  $
can be determined by $a_{1}(P),$ $a_{2}(P),$ and $c(P),$ and vice versa.

\begin{proposition}
\label{prop2}Assume (\ref{200}) and let $\psi(P;z)$ be the Baker-Akhiezer
function at $P$ $=$ $(T,\mathcal{C})$ $\in$ $\Gamma(\tau)$. Then we have%
\begin{equation}
r(P)+s(P)\tau=a_{1}(P)+a_{2}(P)\label{72}%
\end{equation}
and%
\begin{equation}
r(P)\eta_{1}(\tau)+s(P)\eta_{2}(\tau)=c(P)\label{73}%
\end{equation}
where $a_{1}(P),$ $a_{2}(P),$ and $c(P),$ are defined in (\ref{HA}).
\end{proposition}

\begin{proof}
Recall
\[
\psi(P;z)=\frac{e^{c(P)z}\sigma(z-a_{1}(P))\sigma(z-a_{2}(P))}{\sigma
(z-\frac{\omega_{k}}{4})\sigma(z+\frac{\omega_{k}}{4})}.
\]
By (\ref{F000}) and (\ref{ell-sec1}), we have
\[
\psi(P;z+\omega_{i})=e^{c(P)\omega_{i}-\left(  a_{1}(P)+a_{2}(P)\right)
\eta_{i}}\psi(P;z).
\]
By (\ref{rs}), we have
\[
c(P)-\left(  a_{1}(P)+a_{2}(P)\right)  \eta_{1}=-2\pi is(P)
\]
and%
\[
c(P)\tau-\left(  a_{1}(P)+a_{2}(P)\right)  \eta_{2}=2\pi ir(P).
\]
Recall the well-known Legendre relation:%
\[
\tau\eta_{1}-\eta_{2}=2\pi i.
\]
Using the Legendre relation, we see that
\[
a_{1}(P)+a_{2}(P)=r(P)+s(P)\tau
\]
and
\[
c(P)=r(P)\eta_{1}(\tau)+s(P)\eta_{2}(\tau).
\]
This completes the proof.
\end{proof}

For any $P$ $\in$ $\Gamma(\tau)$, we define%
\begin{equation}
\sigma(P)\doteqdot a_{1}(P)+a_{2}(P)\label{123}%
\end{equation}
and
\begin{equation}
\kappa(P)\doteqdot\zeta(a_{1}(P)+a_{2}(P))-\frac{1}{2}\sum_{i=1}^{2}\left(
\zeta(a_{i}(P)+\frac{\omega_{k}}{4})+\zeta(a_{i}(P)-\frac{\omega_{k}}%
{4})\right)  .\label{124}%
\end{equation}

By (\ref{72}), we see that $\sigma(P)$ $=$ $r(P)+s(P)\tau$. The following
lemma shows that $\kappa(P)$ defined by (\ref{124}) is exactly the Hecke form
in (\ref{zrs}) with $\left(  r,s\right)  $ $=$ $(r(P),s(P))$.

\begin{lemma}
We have
\begin{align}
\kappa(P) &  =\zeta(r(P)+s(P)\tau)-(r(P)\eta_{1}(\tau)+s(P)\eta_{2}%
(\tau))\label{127}\\
&  =Z(r(P),s(P),\tau)\nonumber
\end{align}
where $Z(r,s,\tau)$ is the Hecke form defined in (\ref{zrs}). Here we remark
that the equality (\ref{127}) always holds even when $\left(
r(P),s(P)\right)  $ $\in$ $\frac{1}{2}\mathbb{Z}^{2}$.
\end{lemma}

\begin{proof}
By (\ref{72}), we have
\begin{equation}
\sigma(P)=a_{1}(P)+a_{2}(P)=r(P)+s(P)\tau\label{125}%
\end{equation}
and by (\ref{69}) and (\ref{71}), we also have%
\begin{equation}
\frac{1}{2}\sum_{i=1}^{2}\left(  \zeta(a_{i}(P)+\frac{\omega_{k}}{4}%
)+\zeta(a_{i}(P)-\frac{\omega_{k}}{4})\right)  =c(P).\label{126}%
\end{equation}
Then (\ref{127}) follows form (\ref{124})-(\ref{126}) and (\ref{73}).
\end{proof}

Recall $\Phi_{e}(z;T)$ in (\ref{EE}):%
\begin{equation}
\Phi_{e}(z;T)=\frac{d_{2}(T)}{\left(  \wp(z)-\wp\right)  ^{2}}+\frac{d_{1}%
(T)}{\left(  \wp(z)-\wp\right)  }+d_{0}(T),\label{76}%
\end{equation}
where
\begin{equation}
d_{2}(T)=\wp^{\prime4},d_{1}(T)=\wp^{\prime2}(\wp^{\prime}T+\wp^{\prime\prime
})\label{761}%
\end{equation}
and%
\begin{equation}
d_{0}(T)=\frac{1}{4}\left(  2\wp^{\prime2}T^{2}+2\wp^{\prime}\wp^{\prime
\prime}T+4(8\wp+e_{k})\wp^{\prime2}-3\wp^{\prime\prime2}\right)  .\label{762}%
\end{equation}
Here we use notations in (\ref{Notation}).

\begin{theorem}
\label{thm7}Let $P=(T,\mathcal{C})$ $\in\Gamma(\tau)$. Assume (\ref{200})
holds true. Then we have%
\begin{equation}
\wp(\sigma(P))=e_{k}+\frac{12e_{k}^{2}-g_{2}}{T^{2}-12e_{k}},\label{131}%
\end{equation}%
\begin{equation}
\wp^{\prime}(\sigma(P))=\frac{-4(12e_{k}^{2}-g_{2})}{\wp^{\prime2}%
(T^{2}-12e_{k})^{2}}\sqrt{-Q(T)},\label{132}%
\end{equation}%
\begin{equation}
\kappa(P)=Z(r(P),s(P),\tau)=\frac{2\sqrt{-Q(T)}}{\wp^{\prime2}(T^{2}-12e_{k}%
)}\label{133}%
\end{equation}
where $\sigma(P)$ and $\kappa(P)$ are defined by (\ref{123}) and (\ref{124}), respectively.
\end{theorem}

For $(r(P),s(P))$ $=$ $(0,0)$, according to Theorem \ref{thm6}, the
corresponding equation (\ref{1999}) is not completely reducible, and
(\ref{131})-(\ref{133}) is reduced to $T^{2}$ $-12e_{k}$ $=0$.

Theorem \ref{thm7} follows from the proposition below and (\ref{761}),
(\ref{762}) easily.

\begin{proposition}
\label{prop3}Adopt notations in (\ref{Notation}). Then $\wp(\sigma(P))$,
$\wp^{\prime}(\sigma(P)),$ and $\kappa(P)$ can be expressed as follows:
\begin{equation}
\wp(\sigma(P))=\frac{-2\wp d_{0}^{2}\left(  d_{1}^{2}-4d_{0}d_{2}\right)
^{2}+d_{1}d_{0}\left(  d_{1}^{2}-4d_{0}d_{2}\right)  ^{2}-Q\left(  d_{1}%
^{2}-2d_{0}d_{2}\right)  ^{2}}{d_{0}^{2}\left(  d_{1}^{2}-4d_{0}d_{2}\right)
^{2}}\label{128}%
\end{equation}%
\begin{equation}
\wp^{\prime}(\sigma(P))=\frac{\left(  -8d_{1}d_{2}^{2}-4\wp^{\prime2}d_{0}%
^{2}d_{2}+2d_{1}^{2}\left(  6\wp d_{2}+d_{0}\wp^{\prime2}\right)  -d_{1}%
^{3}\wp^{\prime\prime}\right)  }{\left(  d_{1}^{2}-4d_{0}d_{2}\right)  ^{2}%
}\sqrt{-Q},\label{129}%
\end{equation}%
\begin{equation}
\kappa(P)=\frac{-2\sqrt{-Q}}{d_{1}^{2}-4d_{0}d_{2}}\label{130}%
\end{equation}
where $Q=Q(T)$ is given in (\ref{Q2}).
\end{proposition}

Proposition \ref{prop3} is a direct consequence of addition formulas in Lemma
\ref{lem1} and the following lemma.

\begin{lemma}
\label{lem3}Denote $a_{i}(P)$ by $a_{i}$, $i=1,2,$ for short. Then we have%
\begin{equation}
\wp(a_{1})+\wp(a_{2})=\frac{2\wp d_{0}-d_{1}}{d_{0}},\label{78}%
\end{equation}%
\begin{equation}
\wp(a_{1})\wp(a_{2})=\frac{\wp^{2}d_{0}-\wp d_{1}+d_{2}}{d_{0}},\label{79}%
\end{equation}%
\begin{equation}
\wp^{\prime}(a_{1})=\frac{-2\sqrt{-Q}\left(  \wp-\wp(a_{1})\right)  ^{2}%
}{d_{0}\left(  \wp(a_{1})-\wp(a_{2})\right)  }\text{ and }\wp^{\prime}%
(a_{2})=\frac{2\sqrt{-Q}\left(  \wp-\wp(a_{2})\right)  ^{2}}{d_{0}\left(
\wp(a_{1})-\wp(a_{2})\right)  }.\label{77}%
\end{equation}

\end{lemma}

\begin{proof}
Let $x=\wp(z)$ and $x_{i}=\wp(a_{i})$. Since $\Phi_{e}(z;T)$ $=\psi
(P;z)\psi(P_{\ast};z)$, up to a nonzero multiple, in terms of zeros, $a_{1},$
$a_{2}$, it is easy to see that%
\begin{equation}
\Phi_{e}(z;T)=\frac{d_{2}(T)}{\left(  x-\wp\right)  ^{2}}+\frac{d_{1}%
(T)}{\left(  x-\wp\right)  }+d_{0}(T)=\frac{d_{0}(T)\left(  x-\wp
(a_{1})\right)  \left(  x-\wp(a_{2})\right)  }{\left(  x-\wp\right)  ^{2}%
}\label{75}%
\end{equation}
which implies (\ref{78}) and (\ref{79}). Put $z=a_{i}$ into (\ref{Q}), and we
have
\[
Q(T)=-\frac{1}{4}(\Phi_{e}^{\prime}(a_{i}))^{2}=-\frac{1}{4}\left(
\frac{d\Phi_{e}(x_{i})}{dx}\cdot\wp^{\prime}(a_{i})\right)  ^{2}%
\]
which leads to
\begin{equation}
\frac{d\Phi_{e}(x_{i})}{dx}=\frac{-2\sqrt{-Q(T)}}{\wp^{\prime}(a_{i}%
)}.\label{74}%
\end{equation}
Then (\ref{77}) follows from a direct computation of $\frac{d\Phi_{e}}{dx}$ by
using the first equality of (\ref{75}).
\end{proof}

Conversely, given $\left(  r,s\right)  $ $\in\mathbb{C}^{2}$, motivated by
(\ref{131})-(\ref{133}), we consider the system of equations:%
\begin{equation}
\left\{
\begin{array}
[c]{l}%
\wp(r+s\tau)=e_{k}+\frac{12e_{k}^{2}-g_{2}}{T^{2}-12e_{k}}\\
\\
\wp^{\prime}(r+s\tau)=\frac{-4(12e_{k}^{2}-g_{2})}{\wp^{\prime2}(T^{2}%
-12e_{k})^{2}}\sqrt{-Q(T)}\\
\\
\kappa(r,s,\tau)=\frac{2\sqrt{-Q(T)}}{\wp^{\prime2}(T^{2}-12e_{k})}%
\end{array}
\right. \label{system}%
\end{equation}
where
\begin{equation}
\kappa(r,s,\tau):=\zeta(r+s\tau)-r\eta_{1}-s\eta_{2}.\label{krs}%
\end{equation}
Indeed, $\kappa(r,s,\tau)=Z(r,s,\tau)$ by definition. Here, for the case
$\sigma=r+s\tau=0$, since the LHS of the (\ref{system}) becomes $\infty$, the
system (\ref{system}) should be understood as
\begin{equation}
T^{2}-12e_{k}=0.\label{sys(0,0)}%
\end{equation}
Notice that in view of (\ref{g2}) and (\ref{14}), we have%
\[
12e_{k}^{2}(\tau)-g_{2}(\tau)=4(e_{k}-e_{i})(e_{k}-e_{i^{\prime}})\not =0
\]
where $\{k,i,i^{\prime}\}=\{1,2,3\}$. Therefore, the system (\ref{system}) is
well-defined only for $\sigma=r+s\tau\not =\frac{\omega_{k}}{2}$.

For any $P$ $=$ $\left(  T,\mathcal{C}\right)  $ $\in$ $\Gamma(\tau)$, we
recall that $\left(  r(P),s(P)\right)  $ is defined in (\ref{rs}) by
$\psi(P;z)$.

\begin{proposition}
\label{prop4}

\noindent(i) Let $P$ $=$ $\left(  T,\mathcal{C}\right)  $ $\in$ $\Gamma(\tau
)$. Then $T$ solves the system (\ref{system}) with $\left(  r,s\right)  $ $=$
$\left(  r(P),s(P)\right)  $.

\noindent(ii) Let $\left(  r,s\right)  $ $\in\mathbb{C}^{2},$ and suppose $T$
solves the system (\ref{system}). Then there is $P$ $=$ $\left(
T,\mathcal{C}\right)  $ $\in$ $\Gamma(\tau)$ such that $\left(  r,s\right)  $
$=$ $\left(  r(P),s(P)\right)  $.
\end{proposition}

\begin{proof}
(i) follows from Theorem \ref{thm7}. To prove (ii), given $\left(  r,s\right)
$ $\in$ $\mathbb{C}^{2}$, suppose $T$ is a solution of system (\ref{system}).
Let $P=\left(  T,\mathcal{C}\right)  \in\Gamma(\tau)$. By Theorem \ref{thm7},
we have%
\[
\wp(r(P)+s(P)\tau)=\wp(r+s\tau),
\]%
\[
\wp^{\prime}(r(P)+s(P)\tau)=\wp^{\prime}(r+s\tau),
\]
and%
\begin{align*}
\kappa(r(P),s(P),\tau)  &  =\zeta(r(P)+s\tau(P))-r(P)\eta_{1}-s(P)\eta_{2}\\
&  =\zeta(r+s\tau)-r\eta_{1}-s\eta_{2}=\kappa(r,s,\tau).
\end{align*}
These equations imply that
\[
r(P)+s(P)\tau=r+s\tau\text{ and }r(P)\eta_{1}+s(P)\eta_{2}=r\eta_{1}+s\eta_{2}%
\]
and by the Legendre relation, we have
\[
\left(  r(P),s(P)\right)  =\left(  r,s\right)  .
\]

\end{proof}

\begin{proposition}
\label{prop7}Let $\left(  r,s\right)  $ $\in$ $\mathbb{C}^{2}$ and set
$\sigma$ $=$ $r+s\tau$ $\not =$ $\frac{\omega_{k}}{2}$. Then system
(\ref{system}) is equivalent to%
\begin{equation}
\left\{
\begin{array}
[c]{l}%
\left(  \wp(\sigma)-e_{k}\right)  \left(  T^{2}-12e_{k}\right)  -\left(
12e_{k}^{2}-g_{2}\right)  =0\\
\\
\wp^{\prime}(\sigma)\left(  T^{2}-12e_{k}\right)  +2(12e_{k}^{2}-g_{2}%
)\kappa(r,s,\tau)=0
\end{array}
\right.  .\label{system2}%
\end{equation}
Here, if $\sigma$ $=$ $r+s\tau$ $=$ $0$, then (\ref{system2}) should be
understood as $T^{2}$ $-12e_{k}$ $=0$.
\end{proposition}

\begin{proof}
Clearly, system (\ref{system}) implies (\ref{system2}). Conversely,
(\ref{system2}) has a solution $T$. By using (\ref{F00}), (\ref{Q2}), and the
first equation of (\ref{system2}), we have%
\[
\frac{\wp^{\prime2}(\sigma)}{-Q(T)}=\left(  \frac{-4(12e_{k}^{2}-g_{2})}%
{\wp^{\prime2}(T^{2}-12e_{k})^{2}}\right)  ^{2}.
\]
By replacing $\sigma$ by $-\sigma$ if necessary, we may have%
\begin{equation}
\wp^{\prime}(\sigma)=\frac{-4(12e_{k}^{2}-g_{2})}{\wp^{\prime2}(T^{2}%
-12e_{k})^{2}}\sqrt{-Q(T)}.\label{95}%
\end{equation}
By the second equation of (\ref{system2}) and (\ref{95}), we also have%
\[
\kappa(r,s,\tau)=\frac{2\sqrt{-Q(T)}}{\wp^{\prime2}(T^{2}-12e_{k})}.
\]

\end{proof}

To solve system (\ref{system}) is equivalent to solving system (\ref{system2}%
). As we mention, for $\sigma=0$, system (\ref{system2}) as well as system
(\ref{system}) are reduced to $T^{2}$ $-12e_{k}$ $=$ $0$ and hence are always
solvable. In such a case, the corresponding $P$ $\in$ $\Gamma(\tau)$ must have
$a_{1}(P)+a_{2}(P)=0$, namely
\[
\{a_{1}(P),a_{2}(P)\}=\{a_{1}(P),-a_{1}(P)\}.
\]
By Corollary \ref{cor1} (iii), the associated ODE (\ref{1999}) must be
\textit{not completely reducible} with $\left(  r(P),s(P)\right)  $ $=$
$(0,0)$ mod $\mathbb{Z}^{2}$.

\begin{proposition}
\label{prop8}Let $\left(  r,s\right)  \in\mathbb{C}^{2}$ such that
$\sigma=r+s\tau\not =0,\frac{\omega_{k}}{2}$. The solvability of system
(\ref{system2}) is equivalent to
\begin{equation}
2Z(r,s,\tau)\left(  \wp(\sigma|\tau)-e_{k}(\tau)\right)  +\wp^{\prime}%
(\sigma|\tau)=0,\label{sys3}%
\end{equation}
that is,
\begin{equation}
Z(r_{k},s_{k},\tau)=0\label{sys4}%
\end{equation}
where $\left(  r_{k},s_{k}\right)  $ is defined in (\ref{rsk}).
\end{proposition}

\begin{proof}
Clearly, (\ref{system2}) is overdetermined in variable $T$. By the standard
elimination theory, it is easy to see that the two equations in (\ref{system2}%
) can be reduced to only one equation if and only if
\begin{equation}
2\kappa(r,s,\tau)\left(  \wp(\sigma)-e_{k}\right)  +\wp^{\prime}%
(\sigma)=0\label{13}%
\end{equation}
which implies (\ref{sys3}) because $\kappa(r,s,\tau)=Z(r,s,\tau)$ by
definition. Since $\sigma\not =0,\frac{\omega_{k}}{2}$, we have $\wp
(\sigma)\not =\infty,e_{k}$. Then (\ref{13}) is equivalent to
\begin{equation}
Z(r,s,\tau)+\frac{\wp^{\prime}(\sigma)}{2\left(  \wp(\sigma)-e_{k}\right)
}=0.\label{15}%
\end{equation}
By addition formula (\ref{F1}),
\begin{align*}
\text{LHS of (\ref{15})}  &  =\zeta(\sigma)-r\eta_{1}(\tau)-s\eta_{2}%
(\tau)+\frac{\wp^{\prime}(\sigma)}{2(\wp(\sigma)-e_{k})}\\
&  =\zeta(\sigma+\frac{\omega_{k}}{2})-r\eta_{1}(\tau)-s\eta_{2}(\tau
)-\frac{1}{2}\eta_{k}(\tau)=Z(r_{k},s_{k},\tau).
\end{align*}
Therefore, (\ref{sys3}) is equivalent to (\ref{sys4}).
\end{proof}

In the following, we discuss the case when $\sigma=r+s\tau=\frac{\omega_{j}%
}{2}$ for $j\in\{1,2,3\} \backslash\{k\}$ occurs.

\begin{proposition}
\label{prop6}Let $\left(  r,s\right)  $ $\in$ $\mathbb{C}^{2}$ such that
$\sigma$ $=$ $r+s\tau$ $=$ $\frac{\omega_{j}}{2}$ for some $j$ $\in$
$\{1,2,3\} \backslash\{k\}$. Suppose (\ref{sys3}) holds true. Then $(r,s)$
$\in$ $\left(  \frac{1}{2}\mathbb{Z}^{2}\right)  $ $\setminus$ $\mathbb{Z}^{2}
$ and there is $P$ $=$ $\left(  T,\mathcal{C}\right)  $ $\in$ $\Gamma(\tau)$
such that the corresponding ODE (\ref{1999}) is \textit{not completely
reducible} with $\left(  r(P),s(P)\right)  $ $=$ $(r,s)$ $\in$ $\left(
\frac{1}{2}\mathbb{Z}^{2}\right)  $ $\setminus$ $\mathbb{Z}^{2}$.
\end{proposition}

\begin{proof}
By (\ref{sys3}), we have
\[
Z(r,s,\tau)\left(  e_{j}-e_{k}\right)  =0.
\]
Since $j\not =k$, we conclude that $Z(r,s,\tau)=0$. By%
\begin{align*}
Z(r,s,\tau)  &  =\zeta(\frac{\omega_{j}}{2})-(\frac{\omega_{j}}{2}-s\tau
)\eta_{1}-s\eta_{2}\\
&  =\frac{\eta_{j}-\omega_{j}\eta_{1}}{2}+2\pi is=\left\{
\begin{array}
[c]{l}%
2\pi is\text{ for }j=1\\
2\pi i(s-\frac{1}{2})\text{ for }j=2,3
\end{array}
\right.  =0,
\end{align*}
we conclude that
\[
\left(  r,s\right)  \in\left(  \frac{1}{2}\mathbb{Z}^{2}\right)
\backslash\mathbb{Z}^{2}.
\]
In this case, (\ref{system}) is reduced to%
\[
\left(  e_{j}-e_{k}\right)  \left(  T^{2}-12e_{k}\right)  -\left(  12e_{k}%
^{2}-g_{2}\right)  =\left(  e_{j}-e_{k}\right)  (T^{2}-4e_{k}+4e_{j})=0
\]
which also implies $Q(T)=0$ by (\ref{Q2}). Consequently, the corresponding ODE
(\ref{1999}) is \textit{not completely reducible} with $\left(
r(P),s(P)\right)  =(r,s)\in\left(  \frac{1}{2}\mathbb{Z}^{2}\right)
\backslash\mathbb{Z}^{2}$.
\end{proof}

\begin{proof}
[Proof of Theorem \ref{main thm 1}]By definition,
\[
Z(r_{k},s_{k},\tau)=\zeta(r_{k}+s_{k}\tau)-r_{k}\eta_{1}-s_{k}\eta_{2}\text{.}%
\]
Let $\sigma=r+s\tau$. Then $r_{k}+s_{k}\tau=\sigma+\frac{\omega_{k}}{2}$. If
$\sigma=\frac{\omega_{k}}{2}$, then $Z(r_{k},s_{k},\tau)=\infty$, a
contradiction. If $\sigma=0$, then $Z(r_{k},s_{k},\tau)$ $=\frac{\eta_{k}}%
{2}-r_{k}\eta_{1}-s_{k}\eta_{2}$ $=-r\eta_{1}-s\eta_{2}$ $=0$ which, together
with $\sigma=r+s\tau=0$ imply $\left(  r,s\right)  =(0,0)$, a contradiction to
$\left(  r,s\right)  \not \in \frac{1}{2}\mathbb{Z}^{2}$. Therefore,
$\sigma=r+s\tau\not =0,\frac{\omega_{k}}{2}$, and by Proposition \ref{prop6},
we have $\sigma=r+s\tau\not =\frac{\omega_{j}}{2}$ for all $j=0,1,2,3.$ By
Proposition \ref{prop8}, there is $T$ solving the system (\ref{system2}) with
$\sigma\not =\frac{\omega_{j}}{2}$ for all $j=0,1,2,3.$ By Proposition
\ref{prop4} (ii), there is $P$ $=$ $\left(  T,\mathcal{C}\right)  $ $\in$
$\Gamma(\tau)$ such that $\left(  r(P),s(P)\right)  $ $=$ $\left(  r,s\right)
\in$ $\mathbb{C}^{2}$ $\backslash$ $\frac{1}{2}\mathbb{Z}^{2}$. By Theorems
\ref{thm5} and \ref{thm6}, the corresponding ODE (\ref{1999}) must be
completely reducible with monodromy data $\left(  r,s\right)  $.
\end{proof}

\section{Non-even family of solutions and pre-modular forms}

\label{Non even family}

Recall that
\[
Z^{(\mathbf{m}_{k})}(r,s,\tau)=Z^{2}(r,s,\tau)-\wp(r+s\tau)+e_{k}(\tau)
\]
where $\mathbf{m}_{k}$ is defined in (\ref{m}). In this section, we are going
to prove the following theorem:

\begin{theorem}
\label{thm21}Equation (\ref{164}) defined on $E_{\tau_{0}}$, $\tau_{0}%
\in\mathbb{H}$ has a non-even family solutions if and only if there is
$\left(  r,s\right)  $ $\in\mathbb{R}^{2}\backslash\frac{1}{2}\mathbb{Z}^{2}$
such that
\begin{equation}
Z(r_{k},s_{k},\tau_{0})=0\text{ and }Z^{(\mathbf{m}_{k})}(r,s,\tau_{0}%
)\not =0,\label{C1}%
\end{equation}
where $\left(  r_{k},s_{k}\right)  $ is defined in (\ref{rsk}).
\end{theorem}

By Theorem \ref{thm22}, we seek an equation (\ref{1999}) (i.e., L$(\mathbf{m}%
_{k},T,E,\tau)$) with $T$ $\not =$ $0$ and $E$ determined by (\ref{200}) such
that its monodromy group is unitary. Let $P_{0}$ $=$ $\left(  0,\mathcal{C}%
_{0}\right)  $ $\in$ $\Gamma(\tau)$, where%
\begin{equation}
\mathcal{C}_{0}^{2}=Q(0)=12\wp^{\prime4}e_{k}(\tau)(e_{i}(\tau)-e_{k}%
(\tau))(e_{i^{\prime}}(\tau)-e_{k}(\tau)).\label{16}%
\end{equation}
Then the corresponding ODE (\ref{1999}) is L$(\mathbf{m}_{k},0,e_{k},\tau
)=$H$\left(  \mathbf{m}_{k},e_{k},\tau\right)  $. In view of (\ref{16}), we
see that L$(\mathbf{m}_{k},0,e_{k},\tau)$ is completely reducible iff
$e_{k}(\tau)\not =0$. Denote $\tau_{0}\in\mathbb{H}$ such that $e_{k}(\tau
_{0})$ $=$ $0$. Suppose $\tau=\tau_{0}$. By (\ref{131})-(\ref{133}), we have
$(r(P_{0}),s(P_{0}))|_{\tau=\tau_{0}}$ $=(0,0)$. Suppose $\tau\not =\tau_{0}$,
i.e., $e_{k}(\tau)\not =0$. Then L$(\mathbf{m}_{k},0,e_{k},\tau)$ is
completely reducible with monodromy data $(r(P_{0}),s(P_{0}))$ $\not \in
\frac{1}{2}\mathbb{Z}^{2}$ determined by (\ref{131})-(\ref{133}). In
particular, we have%
\begin{equation}
\wp(\sigma(P_{0}))=\frac{g_{2}}{12e_{k}}.\label{17}%
\end{equation}
Moreover, by (\ref{sys3}), we must have
\[
Z(r_{k}(P_{0}),s_{k}(P_{0}),\tau)=Z(r(P_{0}),s(P_{0}),\tau)+\frac{\wp^{\prime
}(\sigma(P_{0})}{2(\wp(\sigma(P_{0}))-e_{k})}=0
\]
which leads to
\begin{equation}
Z(r(P_{0}),s(P_{0}),\tau)=\frac{-\wp^{\prime}(\sigma(P_{0})}{2(\wp
(\sigma(P_{0}))-e_{k})}.\label{18}%
\end{equation}
By (\ref{17}) and (\ref{18}), a direct computation shows that%
\begin{align}
&  Z^{2}(r(P_{0}),s(P_{0}),\tau)-\wp(r(P_{0})+s(P_{0})\tau)+e_{k}\label{19}\\
&  =\frac{\wp^{\prime}(\sigma(P_{0})^{2}}{4(\wp(\sigma(P_{0}))-e_{k})^{2}}%
-\wp(\sigma(P_{0})+e_{k}\nonumber\\
&  =\frac{\left(  \frac{g_{2}}{12e_{k}}-e_{i}\right)  \left(  \frac{g_{2}%
}{12e_{k}}-e_{i^{\prime}}\right)  -\left(  \frac{g_{2}}{12e_{k}}-e_{k}\right)
^{2}}{(\frac{g_{2}}{12e_{k}}-e_{k})}\nonumber\\
&  =\frac{\frac{g_{2}}{4}+e_{i}e_{i^{\prime}}-e_{k}^{2}}{(\frac{g_{2}}%
{12e_{k}}-e_{k})}=0\nonumber
\end{align}
The converse statement also holds true.

\begin{lemma}
\label{lem10}Let $(r,s)$ $\not \in $ $\frac{1}{2}\mathbb{Z}^{2}$. Then
$Z(r_{k},s_{k},\tau)$ and $Z^{2}(r,s,\tau)-\wp(r+s\tau)+e_{k}(\tau)$ have a
common zero in $\tau$ if and only if $(r,s)$ $=$ $(r(P_{0}),$ $s(P_{0}))$,
$P_{0}=\left(  0,\mathcal{C}_{0}\right)  \in\Gamma(\tau)$. Consequently, there
is L$(\mathbf{m}_{k},T,E,\tau)$ with $T$ $\not =$ $0$ and $E$ determined by
(\ref{200}) such that it has monodromy data $\left(  r,s\right)  $ $\not \in $
$\frac{1}{2}\mathbb{Z}^{2}$ if and only if
\[
Z(r_{k},s_{k},\tau)=0,\text{ and }Z^{(\mathbf{m}_{k})}(r,s,\tau)\not =0.
\]

\end{lemma}

\begin{proof}
The sufficient part has already been proven. We only prove the necessary part.
Suppose there is $\tau$ such that
\[
Z(r_{k},s_{k},\tau)=0\text{ and }Z^{2}(r,s,\tau)-\wp(\sigma)+e_{k}(\tau)=0.
\]
Then
\begin{align*}
&  Z^{2}(r,s,\tau)-\wp(r+s\tau)+e_{k}(\tau)\\
&  =\frac{\wp^{\prime}(\sigma)^{2}}{4(\wp(\sigma)-e_{k})^{2}}-\wp
(\sigma)+e_{k}\\
&  =\frac{(\wp(\sigma)-e_{i})(\wp(\sigma)-e_{i^{\prime}})-\left(  \wp
(\sigma)-e_{k}\right)  ^{2}}{(\wp(\sigma)-e_{k})^{2}}\\
&  =\frac{3e_{k}\wp(\sigma)-e_{i}e_{i^{\prime}}}{(\wp(\sigma)-e_{k})^{2}}=0
\end{align*}
which implies
\[
\wp(\sigma)=\frac{g_{2}}{12e_{k}}%
\]
and hence $T=0$ by (\ref{131}) and $E=e_{k}$ by (\ref{200}). Since
$(r,s)\not \in \frac{1}{2}\mathbb{Z}^{2}$, we have L$(\mathbf{m}_{k}%
,0,e_{k},\tau)$ is completely reducible with monodromy data $(r(P_{0}%
),s(P_{0}))$ $=(r,s)\not \in \frac{1}{2}\mathbb{Z}^{2}$.
\end{proof}

\begin{proof}
[Proof of Theorem \ref{thm21}]Theorem \ref{thm21} follows from Theorem
\ref{thm22} and Lemma \ref{lem10} because equation (\ref{1999}) L$(\mathbf{m}%
_{k},T,E,\tau)$ is of unitary monodromy if and only if $(r(P),$ $s(P))$ $\in$
$\mathbb{R}^{2}$ $\setminus$ $\frac{1}{2}\mathbb{Z}^{2}$, where $P=$
$(T,\mathcal{C})$ $\in$ $\Gamma(\tau)$.
\end{proof}

\section{The proof of Theorem \ref{thm21 copy(1)}}

\label{The proof}

The main result in this section is the following theorem.

\begin{theorem}
\label{thm24}For any $\left(  r,s\right)  $ $\in$ $\mathbb{R}^{2}%
\backslash\frac{1}{2}\mathbb{Z}^{2}$, $Z(r_{k},s_{k},\tau)$ and
$Z^{(\mathbf{m}_{k})}(r,s,\tau)$ cannot have common zeros in $\tau$ variable.
\end{theorem}

\begin{proof}
[Proof of Theorem \ref{thm21 copy(1)}]Theorem \ref{thm21 copy(1)} follows from
Theorem \ref{thm21} and Theorem \ref{thm24} easily.
\end{proof}

In order to prove Theorem \ref{thm24}, we need to examine the zeros of
$Z(r,s,\tau)$ and $Z^{(\mathbf{m}_{k})}(r,s,\tau)$. By definition, it is clear
that
\begin{equation}
Z(r,s,\tau)=\pm Z(r^{\prime},s^{\prime},\tau)\text{, and }Z^{(\mathbf{m}_{k}%
)}(r,s,\tau)=Z^{(\mathbf{m}_{k})}(r^{\prime},s^{\prime},\tau)\label{612}%
\end{equation}
if
\begin{equation}
\left(  r^{\prime},s^{\prime}\right)  \equiv\pm\left(  r,s\right)  \text{ mod
}\mathbb{Z}^{2}.\label{613}%
\end{equation}
Motivated by the above property, we can focus on studying the zeros of
$Z(r,s,\tau)$ and $Z^{(\mathbf{m}_{k})}(r,s,\tau)$ in $\tau$-variable only
when $\left(  r,s\right)  $ belongs to $[0,1/2]\times\lbrack0,1]\backslash
\frac{1}{2}\mathbb{Z}^{2}$. For our purposes, we introduce the following
notations:%
\[
\square=[0,1/2]\times\lbrack0,1]\backslash\frac{1}{2}\mathbb{Z}^{2},
\]%
\[%
\begin{array}
[c]{l}%
\Delta_{0}=\left\{  \left(  r,s\right)  \left\vert 0<r,s<\frac{1}{2}%
,r+s>\frac{1}{2}\right.  \right\}  ,\\
\\
\Delta_{1}=\left\{  \left(  r,s\right)  \left\vert \frac{1}{2}<r<1,0<s<\frac
{1}{2},r+s>1\right.  \right\}  ,\\
\\
\Delta_{2}=\left\{  \left(  r,s\right)  \left\vert \frac{1}{2}<r<1,0<s<\frac
{1}{2},r+s<1\right.  \right\}  ,\\
\\
\Delta_{3}=\left\{  \left(  r,s\right)  \left\vert 0<r,s<\frac{1}{2}%
,r+s<\frac{1}{2}\right.  \right\}  ,
\end{array}
\]
Clearly,%
\[
\square=%
{\displaystyle\bigcup\limits_{k=0}^{3}}
\overline{\Delta_{k}}.
\]
Let%
\[
F=\left\{  \tau\in\mathbb{H}|0\leq\operatorname{Re}\tau<1,\left\vert
\tau\right\vert \geq1,\left\vert \tau-1\right\vert >1\right\}  \cup\left\{
\rho=e^{\frac{\pi i}{3}}\right\}
\]%
\[
F_{0}=\left\{  \tau\in\mathbb{H}|0\leq\operatorname{Re}\tau\leq1,\left\vert
\tau-\frac{1}{2}\right\vert \geq\frac{1}{2}\right\}
\]
Recall Theorems B and C as follows:\medskip

\noindent\textbf{Theorem} \textbf{B. }\textit{(\cite[Theorem 1.3.]%
{Chen-Kuo-Lin-Wang-JDG}) Let }$\left(  r,s\right)  $ $\in$ $\square$\textit{.
Then }$Z(r,s,\tau)$\textit{\ has a zero in }$\tau\in F_{0}$\textit{\ if and
only if }$\left(  r,s\right)  $ $\in$ $\Delta_{0}$\textit{. Moreover, the zero
}$\tau\in F_{0}$\textit{\ is unique.}\medskip

Set%
\[
\Lambda=\left\{  \tau\in F_{0}\left\vert Z(r,s,\tau)=0\text{ for some }\left(
r,s\right)  \text{ }\in\text{ }\Delta_{0}\right.  \right\}  .
\]

\noindent\textbf{Theorem} \textbf{C. }\textit{(\cite{Chen-Kuo-Lin-Wang-JDG})
The geometry of }$\Lambda$\textit{\ is as follows:\medskip}

\textit{\noindent(i) }$\Lambda$ $\subset$ $\left\{  \left\vert \tau-\frac
{1}{2}\right\vert >\frac{1}{2}\right\}  $\textit{\ is a simply connected
domain in }$F_{0}$\textit{\ that is symmetric with respect to }%
$\operatorname{Re}\tau$ $=$ $1/2.$\textit{\medskip}

\textit{\noindent(ii) }$\partial\Lambda$ $=$ $%
{\displaystyle\bigcup\limits_{i=1}^{3}}
C_{i}$\textit{, where }$C_{i},i=1,2,3,$\textit{\ are smooth connected
curves.\medskip}

\textit{\noindent(iii)} \textit{Each }$C_{i}$\textit{\ connects any two cusps
}$\{0,$\textit{\ }$1,$\textit{\ }$\infty\}$\textit{\ of }$F_{0}$\textit{.
Moreover, we have }$C_{1}$\textit{\ connecting }$0$\textit{\ and }$1$\textit{,
}$C_{2}$\textit{\ connecting }$0$\textit{\ and }$\infty$\textit{, and }$C_{3}%
$\textit{\ connecting }$1$\textit{\ and }$\infty$\textit{.\medskip}

To proceed further, we also require the following simple zero property of
$Z(r,s,\tau)$.\medskip

\noindent\textbf{Theorem} \textbf{D. }\textit{(\cite{Chen-Kuo-Lin-Dahmen})}
\textit{For any }$\left(  r,s\right)  $ $\in$ $\mathbb{C}^{2}$ $\setminus$
$\frac{1}{2}\mathbb{Z}^{2}$, $Z(r,s,\tau)$ \textit{has simple zeros in }$\tau
$\textit{-variable only. Namely, if }$\tau_{0}$\textit{\ is a zero of
}$Z(r,s,\tau)$\textit{, then }$\left.  \frac{dZ(r,s,\tau)}{d\tau}\right\vert
_{\tau=\tau_{0}}$ $\not =$ $0$\textit{.}\medskip

As a direct consequence of Theorems B, C, and D, we have the following
Theorem.\textit{\medskip}

\begin{theorem}
\label{thm23}There is a real analytic map $\tau^{(0)}$ $:$ $\Delta_{0}$
$\rightarrow$ $\Lambda$ $\subset$ $F_{0}$ such that the following
hold:\textit{\medskip}

\noindent(i) $\tau^{(0)}$ is an one-to-one and onto map from $\Delta_{0}$ to
$\Lambda$.\textit{\medskip}

\noindent(ii) $Z(r,s,\tau)$ $=$ $0$ in $F_{0}$ if and only if $\left(
r,s\right)  $ $\in$ $\Delta_{0}$ and $\tau$ $=$ $\tau^{(0)}(r,s)$ $\in$
$\Lambda$.\textit{\medskip}
\end{theorem}

\begin{lemma}
\label{lem13}Let $\ell_{0}$ $\doteqdot$ $\left\{  \left(  r,s\right)
\in\Delta_{0}|2r+s=1\right\}  $, $\ell_{0}^{\prime}$ $\doteqdot$ $\left\{
\left(  r,s\right)  \in\Delta_{0}|r+2s=1\right\}  $, $\ell_{0}^{\prime\prime}$
$\doteqdot$ $\left\{  \left(  r,s\right)  \in\Delta_{0}|r-s=0\right\}  $
Then\textit{\medskip}

\noindent(i) The image $\tau^{(0)}(\ell_{0})$ is contained in the vertical
line $\operatorname{Re}\tau=\frac{1}{2}$ that connects the middle point of
$C_{1}$ to $\infty$.\textit{\medskip}

\noindent(ii) The image $\tau^{(0)}(\ell_{0}^{\prime})$ is contained in the
circle $\left\{  \left\vert \tau-1\right\vert =1\right\}  $ that connects a
point on $C_{3}$ to $0$.\textit{\medskip}

\noindent(ii)) The image $\tau^{(0)}(\ell_{0}^{\prime\prime})$ is contained in
the circle $\left\{  \left\vert \tau\right\vert =1\right\}  $ that connects a
point on $C_{2}$ to $1$.\textit{\medskip}
\end{lemma}

\begin{proof}
(i) has been proved in \cite{Chen-Lin-Eisenstein}; however, for our
self-contained, we also give a proof here. Given any $(r,s)\in\ell_{0}$ and
$z=r+s\tau$. By the definitions of $\zeta(z|\tau)$ and $\wp(z|\tau)$, it is
easy to prove%
\[
\overline{\zeta(z|\tau)}=\zeta(\bar{z}|1-\bar{\tau}),\text{ }\overline
{\wp(z|\tau)}=\wp(\bar{z}|1-\bar{\tau}),
\]%
\[
\overline{\wp^{\prime}(z|\tau)}=\wp^{\prime}(\bar{z}|1-\bar{\tau}).
\]
Thus%
\[
\overline{\eta_{1}(\tau)}=2\overline{\zeta(1/2|\tau)}=2\zeta(1/2|1-\bar{\tau
})=\eta_{1}(1-\bar{\tau}),
\]%
\begin{align*}
\overline{\eta_{2}(\tau)} &  =2\overline{\zeta(\tau/2|\tau)}=2\zeta(\bar{\tau
}/2|1-\bar{\tau})\\
&  =2\zeta(1/2|1-\bar{\tau})-2\zeta((1-\bar{\tau})/2|1-\bar{\tau})\\
&  =\eta_{1}(1-\bar{\tau})-\eta_{2}(1-\bar{\tau}).
\end{align*}
Then, we have%
\begin{align*}
\overline{Z(r,s,\tau)} &  =\overline{\zeta(r+s\tau|\tau)}-r\overline{\eta
_{1}(\tau)}-s\overline{\eta_{2}(\tau)}\\
&  =\zeta(r+s-s(1-\bar{\tau})|1-\bar{\tau})-(r+s)\eta_{1}(1-\bar{\tau}%
)+s\eta_{2}(1-\bar{\tau})\\
&  =Z(r+s,-s,1-\bar{\tau})=-Z(-(r+s),s,1-\bar{\tau})\\
&  =-Z(-(1-r),s,1-\bar{\tau})=-Z(r,s,1-\bar{\tau})
\end{align*}
which implies that $1-\overline{\tau}$ is also a zero of $Z(r,s,\tau)$ in
$F_{0}$ if $\tau$ is. So by Theorem B, we have
\[
1-\overline{\tau^{(0)}(r,s)}=\tau^{(0)}(r(s),s),\text{ }r(s)=\frac{1-s}%
{2}\text{ and }0<s<\frac{1}{2},
\]
i.e.
\[
\tau^{(0)}(r,s)=\frac{1}{2}+ib(s),\text{ }b(s)\in(\frac{1}{2},+\infty).
\]
Suppose that, up to a sequence,%
\[
\lim_{s\rightarrow\frac{1}{2}}b(s)=b\in\lbrack1/2,+\infty).
\]
Then $\frac{1}{2}+ib$ is a zero of $Z(\frac{1}{4},\frac{1}{2},\tau)$ in
$F_{0}$, which is a contradiction with Theorem B, because $(\frac{1}{4}%
,\frac{1}{2})$ $\in$ $\partial\triangle_{1}$. This proves $\lim_{s\rightarrow
\frac{1}{2}}b(s)$ $=$ $+\infty$. On the other hand, let $s\rightarrow0$, we
have $\left(  r(s),s\right)  \rightarrow\left(  \frac{1}{2},0\right)  $. Since
$\lim_{s\rightarrow\frac{1}{2}}b(s)=+\infty$ and $\tau^{(0)}$ is an one-to-one
map, we see that $\tau^{(0)}(r(s),s)=\frac{1}{2}+ib(s)\rightarrow\frac{1}%
{2}+ib_{\ast}$ with $\frac{1}{2}<b_{\ast}<\infty$. Clearly, the limit
$\tau_{\ast}:=\frac{1}{2}+ib_{\ast}$ $\in$ $\partial\Lambda$ $=%
{\displaystyle\bigcup\limits_{i=1}^{3}}
C_{i}$ because $\left(  \frac{1}{2},0\right)  \in\partial\triangle_{1}$. Since
$\operatorname{Re}\tau_{\ast}=\frac{1}{2}$, we conclude that $\tau_{\ast}\in
C_{1}$. The proofs of (ii) and (iii) follow by choosing $\gamma$ $=$ $\left(
\begin{array}
[c]{cc}%
0 & 1\\
-1 & 1
\end{array}
\right)  $ and $\left(
\begin{array}
[c]{cc}%
-1 & 1\\
-1 & 0
\end{array}
\right)  $ in (\ref{609}), respectively.
\end{proof}

The asymptotic behavior of $\tau^{(0)}(r,s)$ as $\left(  r,s\right)  $
$\rightarrow$ $\partial\Delta_{0}\left\backslash \left\{  \left(  \frac{1}%
{2},0\right)  ,\left(  0,\frac{1}{2}\right)  ,\left(  \frac{1}{2},\frac{1}%
{2}\right)  \right\}  \right.  $ are given as follows.

\begin{lemma}
\label{lem15}Let $\tau^{(0)}$ be the map defined in Theorem \ref{thm23} and
$\left(  r,s\right)  \in\Delta_{0}$ such that $\left(  r,s\right)  $
$\rightarrow$ $(r_{0},s_{0})$ $\in$ $\overline{\Delta_{0}}$. The following
hold true:\textit{\medskip}

\noindent(i) Let $0<r_{0}<\frac{1}{2}$ and $s_{0}$ $=$ $\frac{1}{2}$. Then
$\tau^{(0)}(r,s)$ $\rightarrow\infty$.\textit{\medskip}

\noindent(ii) Let $0<s_{0}<\frac{1}{2}$ and $r_{0}$ $=$ $\frac{1}{2}$. Then
$\tau^{(0)}(r,s)$ $\rightarrow0$.\textit{\medskip}

\noindent(iii) Let $0<s_{0}<\frac{1}{2}$ and $r_{0}+s_{0}$ $=$ $\frac{1}{2}$.
Then $\tau^{(0)}(r,s)$ $\rightarrow1$.\textit{\medskip}
\end{lemma}

\begin{proof}
Recall the following asymptotic behavior of $Z(r,s,\tau)$ from
\cite{Chen-Kuo-Lin-Wang-JDG}: Let $q:=e^{2\pi i\tau}$. The $q$-expansion of
$Z(r,s,\tau)$ is given as follows:%
\begin{equation}
Z(r,s,\tau)=2\pi is-\pi i\frac{1+e^{2\pi iz}}{1-e^{2\pi iz}}-\sum
_{n=1}^{\infty}\left(  \frac{e^{2\pi iz}q^{n}}{1-e^{2\pi iz}q^{n}}%
-\frac{e^{-2\pi iz}q^{n}}{1-e^{-2\pi iz}q^{n}}\right) \label{q-Z}%
\end{equation}
where $z=r+s\tau$. See also \cite{Diamond-Shurman,Heck}. For fixed
$s\in\lbrack0,1)$, (\ref{q-Z}) implies
\begin{equation}
\lim_{\tau\rightarrow\infty}Z(r,s,\tau)=\left\{
\begin{array}
[c]{c}%
2\pi i(s-\tfrac{1}{2})\text{ if }s\neq0\\
\pi\cot\pi r\text{ if }s=0.
\end{array}
\right.  .\label{as1}%
\end{equation}
Suppose $\tau^{(0)}(r,s)\rightarrow\infty$. By (\ref{q-Z}), we have%
\begin{equation}
0=Z(r,s,\tau^{(0)}(r,s))=\left\{
\begin{array}
[c]{l}%
2\pi i(s_{0}-\tfrac{1}{2})\left(  1+o(1)\right)  \text{ if }s_{0}\neq0\\
\\
\pi\cot\pi r_{0}\left(  1+o(1)\right)  \text{ if }s_{0}=0
\end{array}
\right.  ,\label{as11}%
\end{equation}
which implies
\begin{equation}
s_{0}=\tfrac{1}{2}\text{ or }\left(  r_{0},s_{0}\right)  =\left(  \frac{1}%
{2},0\right)  .\label{as111}%
\end{equation}

To discuss the cases for $\tau^{(0)}(r,s)\rightarrow0$ and $\tau
^{(0)}(r,s)\rightarrow1$, we will apply the modularity of $Z(r,s,\tau)$. For
this purpose, recall the modularity of $Z(r,s,\tau)$ from
\cite{Chen-Kuo-Lin-Wang-JDG}. For any pair $\left(  z,\tau\right)
\in\mathbb{C}\times\mathbb{H}$ and $\gamma=\left(
\begin{array}
[c]{cc}%
a & b\\
c & d
\end{array}
\right)  \in SL(2,\mathbb{Z})$, the action of $\gamma$ on the pair $\left(
z,\tau\right)  $ is defined by
\begin{equation}
\gamma\cdot(z,\tau)=\left(  z^{\prime},\tau^{\prime}\right)  :=\left(
\frac{z}{c\tau+d},\frac{a\tau+b}{c\tau+d}\right)  .\label{618}%
\end{equation}
Then by writing $z=r+s\tau$, we have
\begin{equation}
z^{\prime}=\frac{z}{c\tau+d}=\frac{r+s\tau}{c\tau+d}=r^{\prime}+s\tau^{\prime
}\label{619}%
\end{equation}
where%
\begin{equation}
\tau^{\prime}=\gamma\cdot\tau=\frac{a\tau+b}{c\tau+d}\text{,}\label{6060}%
\end{equation}
and%
\begin{equation}
\left(  s^{\prime},r^{\prime}\right)  =(s,r)\cdot\gamma^{-1}=\left(
ds-cr,-bs+ar\right) \label{6061}%
\end{equation}
Recall the well-known modularity of $\zeta(z|\tau)$, $\wp(z|\tau),$ and
$\eta_{i}(\tau)$:%
\begin{equation}
\wp(z^{\prime}|\tau^{\prime})=\left(  c\tau+d\right)  ^{2}\wp(z|\tau)\text{
and }\zeta(z^{\prime}|\tau^{\prime})=\zeta(r^{\prime}+s\tau^{\prime}%
|\tau^{\prime})\label{606}%
\end{equation}%
\begin{equation}
\left(
\begin{array}
[c]{c}%
\eta_{2}(\tau^{\prime})\\
\eta_{1}(\tau^{\prime})
\end{array}
\right)  =\left(  c\tau+d\right)  \gamma\cdot\left(
\begin{array}
[c]{c}%
\eta_{2}(\tau)\\
\eta_{1}(\tau)
\end{array}
\right)  .\label{607}%
\end{equation}
By using (\ref{606}) and (\ref{607}), it is easy to derive the modularity of
$Z(r,s,\tau)$:%
\begin{equation}
Z(r^{\prime},s^{\prime},\tau^{\prime})=\left(  c\tau+d\right)  Z(r,s,\tau
).\label{609}%
\end{equation}

Suppose $\tau^{(0)}(r,s)\rightarrow0$. Let $\gamma=\left(
\begin{array}
[c]{cc}%
0 & 1\\
-1 & 0
\end{array}
\right)  \in SL(2,\mathbb{Z})$. By (\ref{609}), we have%
\[
0=Z(-s,r,\frac{-1}{\tau^{(0)}(r,s)})=-\tau^{(0)}(r,s)Z(r,s,\tau^{(0)}(r,s)).
\]
Since $\tau^{(0)}(r,s)\rightarrow0$, $\frac{-1}{\tau^{(0)}(r,s)}%
\rightarrow\infty$. By (\ref{as11}), we have%
\begin{equation}
0=Z(r,s,\tau^{(0)}(r,s))=\frac{-1}{\tau^{(0)}(r_{0},s_{0})}\left\{
\begin{array}
[c]{l}%
2\pi i(r_{0}-\tfrac{1}{2})\left(  1+o(1)\right)  \text{ if }r_{0}\neq0\\
\\
-\pi\cot\pi s_{0}\left(  1+o(1)\right)  \text{ if }r_{0}=0
\end{array}
\right. \label{as12}%
\end{equation}
which implies
\begin{equation}
r_{0}=\tfrac{1}{2}\text{ or }\left(  r_{0},s_{0}\right)  =\left(  0,\frac
{1}{2}\right)  .\label{as122}%
\end{equation}

Suppose $\tau^{(0)}(r,s)\rightarrow1$. Let $\gamma=\allowbreak\left(
\begin{array}
[c]{cc}%
0 & 1\\
-1 & 1
\end{array}
\right)  \in SL(2,\mathbb{Z})$. Again, by (\ref{609}), we have%
\[
0=Z(-s,r+s,\frac{-1}{\tau^{(0)}(r,s)-1})=-\left(  \tau^{(0)}(r,s)-1\right)
Z(r,s,\tau^{(0)}(r,s)).
\]
Since $\tau^{(0)}(r,s)\rightarrow1$, $\frac{-1}{\tau^{(0)}(r,s)-1}%
\rightarrow\infty$. By (\ref{as1}), we have%
\begin{equation}
0=Z(r,s,\tau^{(0)}(r,s))=\frac{-1}{\tau^{(0)}(r,s)-1}\left\{
\begin{array}
[c]{l}%
2\pi i(r_{0}+s_{0}-\tfrac{1}{2})\left(  1+o(1)\right)  \text{ if }r_{0}%
+s_{0}\neq0\\
\\
-\pi\cot\pi s_{0}\left(  1+o(1)\right)  \text{ if }r_{0}+s_{0}=0.
\end{array}
\right.  .\label{as13}%
\end{equation}
And again, this implies
\begin{equation}
r_{0}+s_{0}=\tfrac{1}{2}\text{ or }\left(  r_{0},s_{0}\right)  \equiv\left(
\frac{1}{2},\frac{1}{2}\right)  \text{ mod }\mathbb{Z}^{2}.\label{as133}%
\end{equation}

Now, we prove (i) of Lemma \ref{lem15}. Let $0<r_{0}<\frac{1}{2}$ and
$s_{0}=\frac{1}{2}$. Then $\tau^{(0)}(r,s)\rightarrow\tau^{(0)}(r_{0},\frac
{1}{2})\doteqdot\tau_{0}$. Since $\left(  r_{0},\frac{1}{2}\right)
\in\mathbb{R}^{2}\backslash\frac{1}{2}\mathbb{Z}^{2}$, $\tau_{0}%
\not \in \partial\Lambda$ $=%
{\displaystyle\bigcup\limits_{i=1}^{3}}
C_{i}$. Hence, $\tau_{0}$ must belong to one of the three cusps $\{0,1,\infty
\}$. By (\ref{as122}) and (\ref{as133}), we see that $\tau_{0}\not =0,1$ and
consequently, $\tau_{0}=\infty$. This proves (i) of Lemma \ref{lem15}. The
assertions for (ii) and (iii) are similar to (i), so we omit the details.
\end{proof}

Replace $\left(  r,s\right)  $ by $\left(  r_{k},s_{k}\right)  $. Based on
Theorem \ref{thm23} and Lemma \ref{lem15}, we immediately have the following proposition:

\begin{proposition}
\label{prop9}For any $k$ $=$ $1,$ $2,$ $3,$ there is a map $\tau^{(k)}$
$:\Delta_{k}$ $\rightarrow\Lambda$ such that (i)-(ii) in Theorem \ref{thm23}
hold true with respect to $\left(  r,s\right)  \in\Delta_{k},$ where $k$ $=$
$1,$ $2,$ $3$. Moreover, let $\left(  r,s\right)  $ $\in$ $\Delta_{k} $ and
$\left(  r,s\right)  $ $\rightarrow$ $\left(  r_{0},s_{0}\right)  $ $\in$
$\partial\Delta_{k}$, then we have\textit{\medskip}

\noindent(a) $\tau^{(k)}(r,s)$ $\rightarrow\infty$ if $s_{0}$ $\in$ $\left\{
0,\frac{1}{2}\right\}  $ and $r_{0}$ $\in\left(  0,\frac{1}{2}\right)
\cup\left(  \frac{1}{2},1\right)  .$\textit{\medskip}

\noindent(b) $\tau^{(k)}(r,s)$ $\rightarrow0$ if $r_{0}$ $\in$ $\left\{
0,\frac{1}{2},1\right\}  $ and $s_{0}$ $\in\left(  0,\frac{1}{2}\right)
.$\textit{\medskip}

\noindent(c) $\tau^{(k)}(r,s)$ $\rightarrow1$ if $r_{0}+s_{0}$ $\in$
$\{\frac{1}{2},1\}$ and $s_{0}$ $\in\left(  0,\frac{1}{2}\right)
.$\textit{\medskip}
\end{proposition}

Next, we want to examine the $\tau$ zeros of $Z^{(\mathbf{m}_{k})}(r,s,\tau)$
for any real pair $\left(  r,s\right)  $ $\in\square$. The following lemma
demonstrates that they can be achieved through an appropriate variation of
$Z(r,s,\tau)$:

\begin{lemma}
\label{lem11}For any $\left(  r,s\right)  $ $\in$ $\mathbb{C}^{2}%
\backslash\frac{1}{2}\mathbb{Z}^{2}$, we have\textit{\medskip}

\noindent(i)
\begin{equation}
Z^{(\mathbf{m}_{1})}(r,s,\tau)=4Z(r,\frac{s}{2},2\tau)\cdot Z(r,\frac{s+1}%
{2},2\tau),\label{601}%
\end{equation}
\medskip

\noindent(ii)
\begin{equation}
Z^{(\mathbf{m}_{2})}(r,s,\tau)=Z(\frac{r}{2},s,\frac{\tau}{2})\cdot
Z(\frac{r+1}{2},s,\frac{\tau}{2}),\label{602}%
\end{equation}
\medskip

\noindent(iii)%
\begin{equation}
Z^{(\mathbf{m}_{3})}(r,s,\tau)=Z(\frac{r-s}{2},s,\frac{1+\tau}{2})\cdot
Z(\frac{r-s+1}{2},s,\frac{1+\tau}{2}).\label{603}%
\end{equation}

\end{lemma}

\begin{proof}
First, we prove (i). Note that there is a constant $C(\tau)$ depending on
$\tau\in\mathbb{H}$ such that%
\begin{equation}
\wp(z|\tau)=\wp(z|2\tau)+\wp(z-\tau|2\tau)+C(\tau)\label{626}%
\end{equation}
for any $z\in\mathbb{C}$ because $\wp(z|\tau)-\left(  \wp(z|2\tau)+\wp
(z-\tau|2\tau)\right)  $ are doubly periodic with periods $1$ and $\tau$, and
it has no poles. Indeed, the constant $C(\tau)$ can be easily derived from
equation (\ref{626}). By inserting $z=\frac{1}{2}$ and $z=0$ into (\ref{626}),
respectively, we immediately have%
\[
e_{1}(\tau)=e_{1}(2\tau)+e_{3}(2\tau)+C(\tau)=-e_{2}(2\tau)+C(\tau)
\]
and%
\[
C(\tau)+e_{2}(2\tau)=0
\]
which implies
\begin{equation}
C(\tau)=-e_{2}(2\tau)\text{ and }e_{1}(\tau)+2e_{2}(2\tau)=0.\label{630}%
\end{equation}
Now, by integrating equation (\ref{626}) with respect to $z$, there is a
constant $A(\tau)$ such that%
\begin{equation}
\zeta(z|\tau)=\zeta(z|2\tau)+\zeta(z-\tau|2\tau)-C(\tau)z+A(\tau).\label{627}%
\end{equation}
Furthermore, by substituting $z=\frac{1}{2}$ and $\frac{\tau}{2}$ into
equation (\ref{627}), we obtain%
\[
\eta_{1}(\tau)=2\eta_{1}(2\tau)-\eta_{2}(2\tau)+2A(\tau)-C(\tau)
\]
and%
\[
\eta_{2}(\tau)=2A(\tau)-\tau C(\tau)
\]
which, together with the Legendre relation (\ref{F0000}), imply%
\begin{equation}
\eta_{1}(\tau)=2\eta_{1}(2\tau)-C(\tau)\text{ and }A(\tau)=\tau\eta_{1}%
(2\tau)-\pi i.\label{628}%
\end{equation}
By using (\ref{627}), (\ref{628}), and the Legendre relation, we obtain%
\begin{equation}
Z(r,s,\tau)=Z\left(  r,\frac{s}{2},2\tau\right)  +Z\left(  r,\frac{s+1}%
{2},2\tau\right)  .\label{629}%
\end{equation}
Furthermore, utilizing (\ref{629}), (\ref{626}), and (\ref{630}), we find
\begin{align}
Z^{(\mathbf{m}_{1})}(r,s,\tau) &  =Z^{2}(r,s,\tau)-\wp(r+s\tau|\tau
)+e_{1}(\tau)\label{631}\\
&  =\left[  Z\left(  r,\frac{s}{2},2\tau\right)  +Z\left(  r,\frac{s+1}%
{2},2\tau\right)  \right]  ^{2}\nonumber\\
&  -\left(  \wp(r+\frac{s}{2}2\tau|2\tau)+\wp(r+\frac{s-1}{2}2\tau
|2\tau)+e_{2}(2\tau)\right)  .\nonumber
\end{align}
Next, applying the addition formulas (\ref{F1}) and (\ref{F2}) of $\zeta$ and
$\wp$, we get%
\begin{equation}
\wp(r+\frac{s}{2}2\tau|2\tau)+\wp(r+\frac{s-1}{2}2\tau|2\tau)+e_{2}%
(2\tau)=\frac{\wp^{\prime}(r+s\tau|2\tau)^{2}}{4\left(  \wp(r+s\tau
|2\tau)-e_{2}(2\tau)\right)  ^{2}}\label{632}%
\end{equation}
and%
\begin{align}
&  \left(  Z(r,\frac{s}{2},2\tau)-Z(r,\frac{s+1}{2},2\tau)\right)
^{2}\label{633}\\
&  =\left(  \zeta(r+s\tau|2\tau)-\zeta(r+s\tau+\tau|2\tau)+\frac{\eta
_{2}(2\tau)}{2}\right)  ^{2}\nonumber\\
&  =\frac{\wp^{\prime}(r+s\tau|2\tau)^{2}}{4\left(  \wp(r+s\tau|2\tau
)-e_{2}(2\tau)\right)  ^{2}}.\nonumber
\end{align}
By applying (\ref{632}) and (\ref{633}) to (\ref{631}), we observe that%
\[
Z^{(\mathbf{m}_{1})}(r,s,\tau)=4Z(r,\frac{s}{2},2\tau)\cdot Z(r,\frac{s+1}%
{2},2\tau)
\]
and this proves (i).

To prove (ii) and (iii), we need to exploit the modular property of
$Z^{(\mathbf{m}_{k})}(r,s,\tau)$. By (\ref{606}), we have
\begin{equation}
e_{k^{\prime}}(\tau^{\prime})=\wp\left(  \left.  \frac{\omega_{k^{\prime}%
}(\tau^{\prime})}{2}\right\vert \tau^{\prime}\right)  =(c\tau+d)^{2}\left\{
\begin{array}
[c]{l}%
\wp(\frac{c\tau+d}{2}|\tau)\text{ if }k^{\prime}=1\\
\\
\wp(\frac{a\tau+b}{2}|\tau)\frac{\tau^{\prime}}{2}\text{ if }k^{\prime}=2\\
\\
\wp(\frac{(a+c)\tau+(b+d)}{2}|\tau)\text{ if }k^{\prime}=3
\end{array}
\right. \label{611}%
\end{equation}
where%
\begin{equation}
\frac{\omega_{k^{\prime}}(\tau^{\prime})}{2}=\left\{
\begin{array}
[c]{l}%
\frac{1}{2}\text{ if }k^{\prime}=1\\
\\
\frac{\tau^{\prime}}{2}=\frac{a\tau+b}{2(c\tau+d)}\text{ if }k^{\prime}=2\\
\\
\frac{1+\tau^{\prime}}{2}=\frac{(a+c)\tau+(b+d)}{2(c\tau+d)}\text{ if
}k^{\prime}=3
\end{array}
\right.  .\label{620}%
\end{equation}
More explicitly, we have
\begin{equation}
e_{1}(\tau^{\prime})=(c\tau+d)^{2}\left\{
\begin{array}
[c]{l}%
e_{1}(\tau)\text{ if }\left(  c,d\right)  =(\text{even, odd}),\\
\\
e_{2}(\tau)\text{ if }\left(  c,d\right)  =(\text{odd, even}),\\
\\
e_{3}(\tau)\text{ if }\left(  c,d\right)  =(\text{odd, odd}),
\end{array}
\right. \label{621}%
\end{equation}%
\begin{equation}
e_{2}(\tau^{\prime})=(c\tau+d)^{2}\left\{
\begin{array}
[c]{l}%
e_{1}(\tau)\text{ if }\left(  a,b\right)  =(\text{even, odd}),\\
\\
e_{2}(\tau)\text{ if }\left(  a,b\right)  =(\text{odd, even}),\\
\\
e_{3}(\tau)\text{ if }\left(  a,b\right)  =(\text{odd, odd}),
\end{array}
\right. \label{622}%
\end{equation}
and%
\begin{equation}
e_{3}(\tau^{\prime})=(c\tau+d)^{2}\left\{
\begin{array}
[c]{l}%
e_{1}(\tau)\text{ if }\left(  a+c,b+d\right)  =(\text{even, odd}),\\
\\
e_{2}(\tau)\text{ if }\left(  a+c,b+d\right)  =(\text{odd, even}),\\
\\
e_{3}(\tau)\text{ if }\left(  a+c,b+d\right)  =(\text{odd, odd}).
\end{array}
\right. \label{623}%
\end{equation}
By (\ref{621})-(\ref{623}), for any $k^{\prime}\in\{1,2,3\}$, we have%
\begin{equation}
e_{k^{\prime}}(\tau^{\prime})=\left(  c\tau+d\right)  ^{2}e_{k}(\tau
)\label{624}%
\end{equation}
for some $k\in\{1,2,3\}$ depending on the choice of $\gamma$ which implies the
following modular property of $Z^{(\mathbf{m}_{k})}(r,s,\tau)$:%
\begin{equation}
Z^{(\mathbf{m}_{k^{\prime}})}(r^{\prime},s^{\prime},\tau^{\prime}%
)=(c\tau+d)^{2}Z^{(\mathbf{m}_{k})}(r,s,\tau)\label{625}%
\end{equation}
where $\left(  k^{\prime},k\right)  $ corresponds to the pair defined in
(\ref{624}).

Now, we apply the modularity of $Z^{(\mathbf{m}_{k})}(r,s,\tau)$ derived in
equation (\ref{625}) to prove (ii) and (iii). Take $\gamma_{1}=\left(
\begin{array}
[c]{cc}%
0 & -1\\
1 & 0
\end{array}
\right)  $, then we have%
\begin{align*}
\tau^{2}Z^{(\mathbf{m}_{2})}(r,s,\tau) &  =Z^{(\mathbf{m}_{1})}(r^{\prime
},s^{\prime},\tau^{\prime})=Z^{(\mathbf{m}_{1})}(s,-r,\frac{-1}{\tau})\\
&  =4Z(s,\frac{-r}{2},\frac{-2}{\tau})\cdot Z(s,\frac{-r+1}{2},\frac{-2}{\tau
})\\
&  =4Z(s,\frac{-r}{2},\frac{-2}{\tau})\cdot Z(s,\frac{-r-1}{2},\frac{-2}{\tau
}).
\end{align*}
Applying (\ref{625}) with respect to $\gamma_{1}$ to $Z(s,\frac{-r}{2}%
,\frac{-2}{\tau})$ and $Z(s,\frac{-r-1}{2},\frac{-2}{\tau})$, we have%
\[
Z(s,\frac{-r}{2},\frac{-2}{\tau})=\frac{\tau}{2}Z(\frac{r}{2},s,\frac{\tau}%
{2})\text{ and }Z(s,\frac{-r-1}{2},\frac{-2}{\tau})=\frac{\tau}{2}Z(\frac
{r+1}{2},s,\frac{\tau}{2})
\]
which clearly implies
\[
Z^{(\mathbf{m}_{2})}(r,s,\tau)=Z(\frac{r}{2},s,\frac{\tau}{2})\cdot
Z(\frac{r+1}{2},s,\frac{\tau}{2}).
\]
This proves (ii).

For (iii), we choose $\gamma_{2}=\left(
\begin{array}
[c]{cc}%
1 & 1\\
0 & 1
\end{array}
\right)  $ and then
\begin{align*}
Z^{(\mathbf{m}_{3})}(r,s,\tau) &  =Z^{(\mathbf{m}_{2})}(r^{\prime},s^{\prime
},\tau^{\prime})=Z^{(\mathbf{m}_{2})}(r-s,s,\tau+1)\\
&  =Z(\frac{r-s}{2},s,\frac{\tau+1}{2})\cdot Z(\frac{r-s+1}{2},s,\frac{\tau
+1}{2}).
\end{align*}
This proves (iii) and completes the proof of Lemma \ref{lem11}.
\end{proof}

Based on Lemma \ref{lem11}, we may write
\[
Z^{(\mathbf{m}_{k})}(r,s,\tau)=Z_{1}^{(\mathbf{m}_{k})}\left(  r,s,\tau
\right)  \cdot Z_{2}^{(\mathbf{m}_{k})}\left(  r,s,\tau\right)
\]
where%
\[
Z_{1}^{(\mathbf{m}_{k})}\left(  r,s,\tau\right)  =\left\{
\begin{array}
[c]{l}%
Z(r,\frac{s}{2},2\tau)\text{ \ if \ }k=1,\\
\\
Z(\frac{r}{2},s,\frac{\tau}{2})\text{ \ if \ }k=2,\\
\\
Z(\frac{r-s}{2},s,\frac{1+\tau}{2})\text{ \ if \ }k=3,
\end{array}
\right.
\]
and%
\[
Z_{2}^{(\mathbf{m}_{k})}\left(  r,s,\tau\right)  =\left\{
\begin{array}
[c]{l}%
Z(r,\frac{s+1}{2},2\tau)\text{ \ if \ }k=1,\\
\\
Z(\frac{r+1}{2},s,\frac{\tau}{2})\text{ \ if \ }k=2,\\
\\
Z(\frac{r-s+1}{2},s,\frac{1+\tau}{2})\text{ \ if \ }k=3.
\end{array}
\right.
\]

\begin{lemma}
\label{lem12}For any $\left(  r,s\right)  $ $\in$ $\mathbb{C}^{2}%
\backslash\frac{1}{2}\mathbb{Z}^{2}$, the following statements hold
true:\medskip

\noindent(i) $Z(r,s,\tau)$ and $Z(r,s+\frac{1}{2},\tau)$ cannot have any
common zero in $\tau$.\medskip

\noindent(ii) $Z(r,s,\tau)$ and $Z(r+\frac{1}{2},s,\tau)$ cannot have any
common zero in $\tau$.\textit{\medskip}
\end{lemma}

\begin{proof}
The Lemma follows easily from the following identities:
\[
Z(r,s+\frac{1}{2},\tau)=Z(r,s,\tau)+\frac{\wp^{\prime}(r+s\tau|\tau)}%
{2(\wp(r+s\tau|\tau)-e_{2}(\tau))},
\]
and%
\[
Z(r+\frac{1}{2},s,\tau)=Z(r,s,\tau)+\frac{\wp^{\prime}(r+s\tau|\tau)}%
{2(\wp(r+s\tau|\tau)-e_{1}(\tau))}.
\]

\end{proof}

By Lemma \ref{lem12}, for any $\left(  r,s\right)  $ $\in\square$,
$Z_{1}^{(\mathbf{m}_{k})}\left(  r,s,\tau\right)  $ and $Z_{2}^{(\mathbf{m}%
_{k})}\left(  r,s,\tau\right)  $ do not have any common zero in $\tau$ $\in$
$\mathbb{H}$. Now, we define sets $\square_{(\mathbf{m}_{k})}$ $\subset$
$\square,$ where $k$ $=$ $1,$ $2,$ $3,$ as follows:%
\[
\square_{(\mathbf{m}_{1})}=\left\{  \left(  r,s\right)  \in\square\left\vert
\begin{array}
[c]{c}%
\text{either }0<s\leq\frac{1}{2}\text{ and }\frac{1-s}{2}<r<\frac{1}{2},\text{
or }\\
0<s\leq\frac{1}{2}\text{ and }\frac{1}{2}<r<\frac{2-s}{2}%
\end{array}
\right.  \right\}  ,
\]
\medskip%
\[
\square_{(\mathbf{m}_{2})}=\left\{  \left(  r,s\right)  \in\square\left\vert
\begin{array}
[c]{c}%
0<r<1\text{ and }0<s<\frac{1}{2}\text{ and }\\
r+2s>1
\end{array}
\right.  \right\}  ,
\]
\medskip%
\[
\square_{(\mathbf{m}_{3})}=\left\{  \left(  r,s\right)  \in\square\left\vert
\begin{array}
[c]{c}%
\text{either }0<r<\frac{1}{2},0<s<\frac{1}{2},r-s<0\text{, or}\\
\left(  r,s\right)  \in\Delta_{1}%
\end{array}
\right.  \right\}  .
\]
Let $A$ denote any subset of $\mathbb{H}$. In the following, we use the
notation
\[
\lambda A-\beta:=\left\{  \tau\in\mathbb{H}\left\vert \frac{\tau+\beta
}{\lambda}\in A\right.  \right\}
\]
for any $\lambda\in\mathbb{R}\backslash\{0\}$ and $\beta\in\mathbb{C}$. Set
\begin{equation}
\left(  \lambda_{k},\beta_{k}\right)  =(\frac{1}{2},0),(2,0),(2,1)\text{ for
}k=1,2,3,\text{ respectively.}\label{604}%
\end{equation}

\begin{proposition}
\label{prop10}Let $\left(  r,s\right)  $ $\in$ $\square$ and $\left(
\lambda_{k},\beta_{k}\right)  $ be given in (\ref{604}). Then $Z^{(\mathbf{m}%
_{k})}(r,s,\tau)$ has a zero in $\tau$ $\in$ $\lambda_{k}F_{0}$ $-\beta_{k}$
if and only if $\left(  r,s\right)  $ $\in$ $\square_{(\mathbf{m}_{k})} $ for
$k$ $=$ $1,$ $2,$ $3.$ Moreover, for each $\left(  r,s\right)  $ $\in$
$\square_{(\mathbf{m}_{k})}$, there are exactly two different zeros $\tau
_{1}^{(\mathbf{m}_{k})}$ and $\tau_{2}^{(\mathbf{m}_{k})}$of $Z^{(\mathbf{m}%
_{k})}(r,s,\tau)$ in $\lambda_{k}\Lambda-\beta_{k},$ and both of them are
simple zeros.
\end{proposition}

\begin{proof}
For $k=1$, in view of (\ref{601}), we see that $Z^{(\mathbf{m}_{1})}%
(r,s,\tau)=$ $0$ for $\left(  r,s\right)  $ $\in\square$ and $\tau
\in\mathbb{H}$ iff either $Z(r,\frac{s}{2},2\tau)$ $=0$ or $Z(r,\frac{s+1}%
{2},2\tau)=0$. By Theorem \ref{thm23} or Theorem B., equation $Z(r,\frac{s}%
{2},2\tau)$ $=0$ (resp. $Z(r,\frac{s+1}{2},2\tau)=0$) is solvable for
$2\tau\in F_{0}$ if and only if $\left(  r,\frac{s}{2}\right)  \in$
$\Delta_{0}$ (resp. $\left(  r,\frac{s+1}{2}\right)  $ $\in$ $\Delta_{0}$) and
$2\tau$ $\in\Lambda$. A direct computation shows that $\left(  r,\frac{s}%
{2}\right)  \in$ $\Delta_{0}$ and $\left(  r,\frac{s+1}{2}\right)  $ $\in$
$\Delta_{0}$ are both equivalent to $\left(  r,s\right)  \in\square
_{(\mathbf{m}_{1})}$. By (\ref{601}) and Theorem B., for each $\left(
r,s\right)  \in\square_{(\mathbf{m}_{1})}$, there are associates $\tau
_{1}^{(\mathbf{m}_{1})}$ and $\tau_{2}^{(\mathbf{m}_{1})}$ $\in\frac{1}%
{2}\Lambda$ such that $Z^{(\mathbf{m}_{1})}(r,s,\tau_{j}^{(\mathbf{m}_{1}%
)})=0,$ $j=1,2,$ where $\tau_{1}^{(\mathbf{m}_{1})}$ and $\tau_{2}%
^{(\mathbf{m}_{1})}$ correspond to $Z(r,\frac{s}{2},2\tau_{1}^{(\mathbf{m}%
_{1})})$ $=0$ and $Z(r,\frac{s+1}{2},2\tau_{2}^{(\mathbf{m}_{1})})$ $=0,$
respectively. By Lemma \ref{lem12}, we conclude that $\tau_{1}^{(\mathbf{m}%
_{1})}\not =\tau_{2}^{(\mathbf{m}_{1})}$. Take the derivative of
$Z^{(\mathbf{m}_{1})}(r,s,\tau)$ with respect to $\tau$. By $\tau
_{1}^{(\mathbf{m}_{1})}\not =\tau_{2}^{(\mathbf{m}_{1})}$ and Theorem D., we
can also conclude that both $\tau_{1}^{(\mathbf{m}_{1})}$ and $\tau
_{2}^{(\mathbf{m}_{1})}$ are simple zeros of $Z^{(\mathbf{m}_{1})}(r,s,\tau)$.
This completes the proof for $k$ $=$ $1$. The cases for $k$ $=$ $2,3$ are
similar, so we skip the details here.
\end{proof}

Proposition \ref{prop10} and Theorem \ref{thm23} imply the following theorem.

\begin{theorem}
For each $k$ $=$ $1,$ $2,$ $3,$ there are two real analytic maps $\tau
_{1}^{(\mathbf{m}_{k})},\tau_{2}^{(\mathbf{m}_{k})}$ $:$ $\square
_{(\mathbf{m}_{k})}$ $\rightarrow$ $\lambda_{k}\Lambda-\beta_{k}$, such that
the followings hold true:\textit{\medskip}

\noindent(i) Both $\tau_{1}^{(\mathbf{m}_{k})}$ and $\tau_{2}^{(\mathbf{m}%
_{k})}$ are one-to-one and onto maps from $\square_{(\mathbf{m}_{k})}$ to
$\lambda_{k}\Lambda-\beta_{k}$.\textit{\medskip}

\noindent(ii) $\tau_{1}^{(\mathbf{m}_{k})}(r,s)$ $\not =$ $\tau_{2}%
^{(\mathbf{m}_{k})}(r,s)$ for any $\left(  r,s\right)  \in\square
_{(\mathbf{m}_{k})}$.\textit{\medskip}

\noindent(iii) $Z^{(\mathbf{m}_{k})}(r,s,\tau)$ $=$ $0$ in $\tau$ $\in$
$\lambda_{k}F_{0}$ if and only if $\left(  r,s\right)  $ $\in$ $\square
_{(\mathbf{m}_{k})}$ and either $\tau$ $=$ $\tau_{1}^{(\mathbf{m}_{k})}(r,s)$
or $\tau$ $=$ $\tau_{2}^{(\mathbf{m}_{k})}(r,s)$. More precisely, $\tau
_{j}^{(\mathbf{m}_{k})}(r,s)$, $j$ $=$ $1,2,$ correspond to the unique simple
zero of $Z_{j}^{(\mathbf{m}_{k})}\left(  r,s,\tau\right)  $ for $j$ $=$ $1,2,$
respectively.\textit{\medskip}
\end{theorem}

\begin{lemma}
\label{lem14}Let $\left(  r,s\right)  $ $\in$ $\square$. Then $Z(r_{k}%
,s_{k},\tau)$ and $Z^{(\mathbf{m}_{k})}(r,s,\tau)$ cannot have any common zero
in $\tau$ $\in$ $F_{0}$ $\cap$ $(\lambda_{k}F_{0}$ $-\beta_{k})$.
\end{lemma}

\begin{proof}
First, we prove the case for $k=1$. Suppose there exist $\left(  r,s\right)
\in\square$ and $\tau_{0}\in F_{0}\cap\frac{1}{2}F_{0}$ such that
\begin{equation}
Z(r+\frac{1}{2},s,\tau_{0})=0=Z^{(\mathbf{m}_{1})}(r,s,\tau_{0}).\label{605}%
\end{equation}

Then we must have
\[
\left(  r,s\right)  \in\Delta_{1}\cap\square_{(\mathbf{m}_{1})}=\left\{
\left(  r,s\right)  \in\square\left\vert 0<s<\frac{1}{2}\text{ and }%
s<r<\frac{2-s}{2}\right.  \right\}
\]
and
\[
\tau_{0}\in\Lambda\cap\frac{1}{2}\Lambda\subset\left\{  0<\operatorname{Re}%
\tau<\frac{1}{2}\right\}  \cap F_{0}.
\]
We want to claim that the image of $\Delta_{1}\cap\square_{(\mathbf{m}_{1})}$
under the map $\tau^{(1)}$ lies in the region $\left\{  \frac{1}%
{2}<\operatorname{Re}\tau\right\}  \cap F_{0}$, i.e.,%
\[
\tau^{(1)}\left(  \Delta_{1}\cap\square_{(\mathbf{m}_{1})}\right)
\subset\left\{  \frac{1}{2}<\operatorname{Re}\tau\right\}  \cap F_{0}.
\]
If so, for each $\tau_{0}\in\Lambda\cap\frac{1}{2}\Lambda$, there exist
$\left(  r_{1},s_{1}\right)  \in\Delta_{1}$ and $\left(  r_{1}^{\prime}%
,s_{1}^{\prime}\right)  \in\square_{(\mathbf{m}_{1})}$ such that
\[
Z(r_{1}+\frac{1}{2},s_{1},\tau_{0})=0\text{ and }Z^{(\mathbf{m}_{1})}%
(r_{1}^{\prime},s_{1}^{\prime},\tau_{0})=0
\]
but
\[
\left(  r_{1},s_{1}\right)  \not =\left(  r_{1}^{\prime},s_{1}^{\prime
}\right)  ,
\]
a contradiction to (\ref{605}). Now, we prove the claim. Let $\ell_{1}$ be the
line segment defined by $\ell_{1}:2r+s=2,$ where $\frac{3}{4}<r<1$. By Lemma
\ref{lem13} and shift $r$ by $r+\frac{1}{2}$, the image $\tau^{(1)}(\ell
_{1})=\tau^{(0)}(\ell_{0})$ is contained in the vertical line $\left\{
\operatorname{Re}\tau=\frac{1}{2}\right\}  $ that connects the middle point
$\tau_{\ast}=\frac{1}{2}+ib_{\ast}$ of $C_{1}$ to $\infty$. Since the map
$\tau^{(1)}$ is one-to-one from $\Delta_{1}$ onto $\Lambda$ and
\[
\Delta_{1}=\left(  \Delta_{1}\cap\square_{(\mathbf{m}_{1})}\right)  \cup\{
\ell_{1}\} \cup\left\{  \frac{2-s}{2}<r<1\text{ and }0<s<\frac{1}{2}\right\}
,
\]
we have either
\[
\tau^{(1)}(\Delta_{1}\cap\square_{(\mathbf{m}_{1})})\subset\left\{  \frac
{1}{2}<\operatorname{Re}\tau\right\}  \cap F_{0}%
\]
or
\[
\tau^{(1)}(\Delta_{1}\cap\square_{(\mathbf{m}_{1})})\subset\left\{
0<\operatorname{Re}\tau<\frac{1}{2}\right\}  \cap F_{0}.
\]
Fix $s_{0}$ such that $0<s_{0}<\frac{1}{2}$. Then $\left(  r,s_{0}\right)
\in\Delta_{1}\cap\square_{(\mathbf{m}_{1})}$ provided $s_{0}<r<\frac{2-s_{0}%
}{2}$. Since $\tau^{(1)}(r,s_{0})\rightarrow1$ as $r\rightarrow$ $\frac{1}{2}$
$-s_{0}$ by Lemma \ref{lem15}, we conclude that the image $\tau^{(1)}%
(\Delta_{1}\cap\square_{(\mathbf{m}_{1})})$ must be contained in $\left\{
\frac{1}{2}<\operatorname{Re}\tau\right\}  \cap F_{0}$. This proves the claim
and finishes the case $k$ $=$ $1$.

Second, we prove the case for $k$ $=$ $2$. Suppose there exist $\left(
r,s\right)  \in\square$ and $\tau_{0}\in F_{0}\cap2F_{0}$ such that%
\begin{equation}
Z(r,s+\frac{1}{2},\tau_{0})=0=Z^{(\mathbf{m}_{2})}(r,s,\tau_{0}).\label{608}%
\end{equation}
Then we must have
\[
\left(  r,s\right)  \in\Delta_{2}\cap\square_{(\mathbf{m}_{2})}=\left\{
\left(  r,s\right)  \in\square\left\vert 0<s<\frac{1}{2}\text{ and }%
s<r<\frac{2-s}{2}\right.  \right\}
\]
and
\[
\tau_{0}\in\Lambda\cap2\Lambda\subset F_{0}\cap\left\{  1<\left\vert
\tau-1\right\vert \right\}  .
\]
We want to claim that the image of $\Delta_{2}\cap\square_{(\mathbf{m}_{2})}$
under the map $\tau^{(2)}$ lies in the region $\left\{  \left\vert
\tau-1\right\vert <1\right\}  \cap F_{0}$, i.e.,%
\[
\tau^{(2)}\left(  \Delta_{2}\cap\square_{(\mathbf{m}_{2})}\right)
\subset\left\{  \left\vert \tau-1\right\vert <1\right\}  \cap F_{0}.
\]
If so, for each $\tau_{0}\in\Lambda\cap2\Lambda$, there exist $\left(
r_{2},s_{2}\right)  \in\Delta_{2}$ and $\left(  r_{2}^{\prime},s_{2}^{\prime
}\right)  \in\square_{(\mathbf{m}_{2})}$ such that
\[
Z(r_{2},s_{2}+\frac{1}{2},\tau_{0})=0\text{ and }Z^{(\mathbf{m}_{2})}%
(r_{2}^{\prime},s_{2}^{\prime},\tau_{0})=0
\]
but
\[
\left(  r_{2},s_{2}\right)  \not =\left(  r_{2}^{\prime},s_{2}^{\prime
}\right)  ,
\]
a contradiction to (\ref{608}). Let $\ell_{2}$ be the line segment defined by
$\ell_{2}:r+2s=1,$ where $\frac{1}{2}<r<1$. It is easy to see that $\tau
^{(2)}(\ell_{2})=$ $\tau^{(0)}(\ell_{0}^{\prime})$ is contained in the circle
$\left\{  \left\vert \tau-1\right\vert =1\right\}  $ that connects a point on
$C_{3}$ to $0$. Since the map $\tau^{(2)}$ is one-to-one from $\Delta_{2}$
onto $\Lambda$ and
\[
\Delta_{2}=\left(  \Delta_{2}\cap\square_{(\mathbf{m}_{2})}\right)  \cup\{
\ell_{2}\} \cup\left\{  0<s<\frac{1-r}{2}\text{ and }\frac{1}{2}<r<1\right\}
,
\]
we have either
\[
\tau^{(2)}(\Delta_{1}\cap\square_{(\mathbf{m}_{1})})\subset\left\{  \left\vert
\tau-1\right\vert >1\right\}  \cap\Lambda
\]
or
\[
\tau^{(1)}(\Delta_{1}\cap\square_{(\mathbf{m}_{1})})\subset\left\{  \left\vert
\tau-1\right\vert <1\right\}  \cap\Lambda.
\]
Fix $s_{0}$ such that $0<s_{0}<\frac{1}{2}$. Then $\left(  r,s_{0}\right)
\in\Delta_{2}\cap\square_{(\mathbf{m}_{2})}$ provided $1-2s_{0}<r<1-s_{0}$.
Since $\tau^{(2)}(r,s_{0})\rightarrow1$ as $r\rightarrow$ $1$ $-s_{0}$ by
Lemma \ref{lem15}, we conclude that the image $\tau^{(2)}(\Delta_{2}%
\cap\square_{(\mathbf{m}_{2})})$ must be contained in $\left\{  \left\vert
\tau-1\right\vert <1\right\}  \cap\Lambda$. This proves the claim and finishes
the case $k=2$.

For $k=3$, the proof is similar to the case $k$ $=2$, Suppose there exist
$\left(  r,s\right)  \in\square$ and $\tau_{0}\in F_{0}\cap\left(
2F_{0}-1\right)  $ such that%
\begin{equation}
Z(r+\frac{1}{2},s+\frac{1}{2},\tau_{0})=0=Z^{(\mathbf{m}_{3})}(r,s,\tau
_{0}).\label{610}%
\end{equation}

Then we must have
\[
\left(  r,s\right)  \in\Delta_{3}\cap\square_{(\mathbf{m}_{3})}=\left\{
\left(  r,s\right)  \in\square\left\vert 0<r<\frac{1}{4}\text{ and }%
r<s<\frac{1}{2}-r\right.  \right\}
\]
and
\[
\tau_{0}\in\Lambda\cap\left(  2\Lambda-1\right)  \subset\Lambda\cap\left\{
1<\left\vert \tau\right\vert \right\}  .
\]
We want to claim that the image of $\Delta_{3}\cap\square_{(\mathbf{m}_{3})}$
under the map $\tau^{(3)}$ lies in the region $\left\{  \left\vert
\tau\right\vert <1\right\}  \cap\Lambda$. If so, for each $\tau_{0}\in
\Lambda\cap\left(  2\Lambda-1\right)  $, there exist $\left(  r_{3}%
,s_{3}\right)  \in\Delta_{3}$ and $\left(  r_{3}^{\prime},s_{3}^{\prime
}\right)  $ $\in\square_{(\mathbf{m}_{3})}$ such that
\[
Z(r_{3}+\frac{1}{2},s_{3}+\frac{1}{2},\tau_{0})=0\text{ and }Z^{(\mathbf{m}%
_{3})}(r_{3}^{\prime},s_{3}^{\prime},\tau_{0})=0
\]
but
\[
\left(  r_{3},s_{3}\right)  \not =\left(  r_{3}^{\prime},s_{3}^{\prime
}\right)  ,
\]
a contradiction to (\ref{610}). Let $\ell_{3}:r-s=0$, where $0<r<\frac{1}{4}
$. It is easy to see that $\tau^{(3)}(\ell_{3})=$ $\tau^{(0)}(\ell_{0}%
^{\prime\prime})$ is contained in the circle $\left\{  \left\vert
\tau\right\vert =1\right\}  $ that connects a point on $C_{2}$ to $1$. Note
that
\[
\Delta_{3}=\left(  \Delta_{3}\cap\square_{(\mathbf{m}_{3})}\right)  \cup\{
\ell_{3}\} \cup\left\{  0<s<\frac{1}{4},s<r<\frac{1}{2}-s\text{ }\right\}  .
\]
Since the map $\tau^{(3)}$ is one-to-one from $\Delta_{3}$ onto $\Lambda$, we
have either
\[
\tau^{(3)}(\Delta_{3}\cap\square_{(\mathbf{m}_{3})})\subset\left\{  \left\vert
\tau\right\vert >1\right\}  \cap\Lambda
\]
or
\[
\tau^{(3)}(\Delta_{3}\cap\square_{(\mathbf{m}_{3})})\subset\left\{  \left\vert
\tau\right\vert <1\right\}  \cap\Lambda.
\]
Take $\left(  r,\frac{1}{4}\right)  \in\Delta_{3}\cap\square_{(\mathbf{m}%
_{3})}$ for any $0<r<\frac{1}{4}$. Since $\tau^{(3)}(r,\frac{1}{4}%
)\rightarrow0$ as $r\rightarrow$ $0$ by Lemma \ref{lem15}, we conclude that
the image $\tau^{(3)}(\Delta_{3}\cap\square_{(\mathbf{m}_{3})})$ must be
contained in $\left\{  \left\vert \tau\right\vert <1\right\}  $ $\cap\Lambda$.
This proves the claim and finishes the case $k=3$.
\end{proof}

\begin{proof}
[Proof of Theorem \ref{thm24}]In the following, we only prove for the case
$k=1$ and cases $k$ $=2,3$ can be obtained by the same argument as $k=1$.

Suppose there is $\tau_{0}\in\mathbb{H}$ and $\left(  r,s\right)
\in\mathbb{R}^{2}\backslash\frac{1}{2}\mathbb{Z}^{2}$ such that
\begin{equation}
Z(r_{1},s_{1},\tau_{0})=Z^{(m_{1})}(r,s,\tau_{0})=0\label{66666}%
\end{equation}
where
\begin{equation}
\left(  r_{1},s_{1}\right)  =\left(  r+\frac{1}{2},s\right)  .\label{666667}%
\end{equation}
Note that there is $\gamma=\left(
\begin{array}
[c]{cc}%
a & b\\
c & d
\end{array}
\right)  \in SL(2,\mathbb{Z})$ such that $\tau_{0}^{\prime}=\gamma\cdot
\tau_{0}=\frac{a\tau+b}{c\tau+d}\in F_{0}$. By (\ref{6061}) and (\ref{666667}%
), we have
\[
\left(  s_{1}^{\prime},r_{1}^{\prime}\right)  =(s_{1},r_{1})\cdot\gamma
^{-1}=\left(  s^{\prime}-\frac{c}{2},r^{\prime}+\frac{a}{2}\right)  .
\]
By the modularity (\ref{609}) and (\ref{625}) of $Z(r,s,\tau)$ and
$Z^{(\mathbf{m}_{k})}(r,s,\tau)$, we have%
\begin{equation}
Z(r_{1}^{\prime},s_{1}^{\prime},\tau_{0}^{\prime})=Z\left(  r^{\prime}%
+\frac{a}{2},s^{\prime}-\frac{c}{2},\tau_{0}^{\prime}\right)  =\left(
c\tau+d\right)  Z(r_{1},s_{1},\tau_{0})=0,\label{614}%
\end{equation}
and%
\begin{align}
Z^{(m_{j})}(r^{\prime},s^{\prime},\tau_{0}^{\prime}) &  =Z^{2}(r^{\prime
},s^{\prime},\tau_{0}^{\prime})-\wp(r^{\prime}+s^{\prime}\tau_{0}^{\prime
}|\tau_{0}^{\prime})+e_{j}(\tau_{0}^{\prime})\label{615}\\
&  =\left(  c\tau+d\right)  ^{2}Z^{(m_{1})}(r,s,\tau_{0})=0\text{ }\nonumber
\end{align}
for some $j\in\{1,2,3\},$ depending on $\gamma$. There are three possibilities
corresponding to $j$ $=$ $1,2,$ and $3,$ respectively. If $j=1$ in
(\ref{615}), i.e.,%
\begin{equation}
Z^{(m_{1})}(r^{\prime},s^{\prime},\tau_{0}^{\prime})=\left(  c\tau+d\right)
^{2}Z^{(m_{1})}(r,s,\tau_{0})=0,\label{616}%
\end{equation}
then $e_{1}(\tau_{0}^{\prime})=\left(  c\tau+d\right)  ^{2}e_{1}(\tau_{0})$.
By (\ref{621}), we have $\left(  c,d\right)  =\left(  \text{even}%
,\text{odd}\right)  $. Since $ad-bc=1$, $a$ must be odd. By (\ref{614}), we
see that
\begin{equation}
Z\left(  r^{\prime}+\frac{a}{2},s^{\prime}-\frac{c}{2},\tau_{0}^{\prime
}\right)  =Z\left(  r^{\prime}+\frac{1}{2},s^{\prime},\tau_{0}^{\prime
}\right)  =0.\label{617}%
\end{equation}
By (\ref{612}) and (\ref{613}), we may replace $\left(  r^{\prime},s^{\prime
}\right)  $ by $\pm(r^{\prime}+m,s^{\prime}+n),\left(  m,n\right)
\in\mathbb{Z}^{2}$ if necessary such that $\left(  r^{\prime},s^{\prime
}\right)  \in\mathbb{\square}$. By acting $\tau_{0}^{\prime}$ by $\left(
\begin{array}
[c]{cc}%
1 & -1\\
2 & -1
\end{array}
\right)  $ if $\operatorname{Re}\tau_{0}^{\prime}>\frac{1}{2}$, we may assume
$\operatorname{Re}\tau_{0}^{\prime}<\frac{1}{2}$. By (\ref{616}) and
(\ref{617}), there is $\left(  r^{\prime},s^{\prime}\right)  \in
\mathbb{\square}$ and $\tau_{0}^{\prime}\in F_{0}\cap\{ \operatorname{Re}%
\tau_{0}^{\prime}<\frac{1}{2}\} $ such that
\[
Z\left(  r^{\prime}+\frac{1}{2},s^{\prime},\tau_{0}^{\prime}\right)
=Z^{(m_{1})}(r^{\prime},s^{\prime},\tau_{0}^{\prime})=0,
\]
which together with $\tau_{0}^{\prime}\in F_{0}\cap\{ \operatorname{Re}%
\tau_{0}^{\prime}<\frac{1}{2}\}$ imply
\[
\tau_{0}^{\prime}\in F_{0}\cap\frac{1}{2}F_{0},
\]
a contradiction to Lemma \ref{lem14}. The proofs for $j=2,3$ are similar to
the case $j=1$ and also lead to a contradiction to Lemma \ref{lem14}, so we
skip the details. This completes the proof for the case $k=1$.
\end{proof}

\label{Reference}

\end{document}